%% file: main.tex
\documentclass[11pt,letterpaper]{amsart}

\usepackage{package}
\usepackage{command}
\usepackage{bibstyle}
\usepackage{setting}

\begin{document}

\input{section/Author}

\input{section/Abstract}

\maketitle

\input{section/Introduction}
\input{section/Hybrid_Periodic_Orbit}
\input{section/lemmas_orbits_existence}
\input{section/Rigidity_Standard}
\input{section/Injectivity}

\input{section/Acknowledgements}

\printbibliography

\end{document}

%% file: section/Author.tex
\title[Infinitesimal dynamical spectral rigidity]{Infinitesimal dynamical spectral rigidity of the Bunimovich stadium}

\author{Xianghui Shi}

\address{Xianghui Shi, Department of Mathematics, University of California, Los Angeles,
    520 Portola Plaza, Los Angeles, CA 90095, USA}
\email{xhshi@math.ucla.edu}

\subjclass[2020]{Primary: 37C83; Secondary: 58J53, 37J46}

\keywords{Bunimovich stadium, billiard, length spectrum, dynamical spectral rigidity,
    hybrid periodic orbit, linearized isospectral operator}

%% file: section/Abstract.tex
\begin{abstract}
	We prove the infinitesimal dynamical spectral rigidity of the circular Bunimovich stadium with sufficiently long flat edges, in the class of domains obtained by replacing its two circular arcs with strictly convex $C^{4}$ arcs and keeping its reflection symmetry across the axis parallel to the flat edges.
	That is, if a $C^{1}$ one-parameter family in this class passing through the stadium preserves the length spectrum, then its deformation function at the stadium vanishes identically.
	For flat edges of arbitrary length, the deformation functions of such families still span a finite-dimensional space.
\end{abstract}

%% file: section/Introduction.tex
\section{Introduction}
\label{sec:Introduction}

The \emph{length spectrum} of a planar domain $\Omega$, denoted by $\mathcal{L}(\Omega)$, is the set of lengths of all closed billiard trajectories in $\Omega$.
It is the dynamical counterpart of the Laplace spectrum: for a convex domain with smooth boundary, the singular support of the wave trace of the Dirichlet Laplacian is contained in the closure of $\pm \parentheses[\big]{ \mathcal{L}(\Omega) \cup \n \abs{\partial\Omega} } \cup \set{0}$~\cite{AnderssonMelrose1977, PetkovStoyanov1992}, with equality under generic non-degeneracy conditions~\cite{PetkovStoyanov1992} (see also the surveys~\cite{Zelditch2014survey, KaloshinSorrentino2022survey}).\footnote{
Koval and Vig~\cite{KovalVig2024} showed that this equality can almost fail, in the sense that for each $N$, every ellipse admits arbitrarily $C^{\infty}$-small deformations $\Omega$ whose wave trace is $C^{N}$ near some point of $\mathcal{L}(\Omega)$.
}
In parallel to Kac's famous question~\cite{Kac_1966}, ``Can one hear the shape of a drum?'', one asks whether the length spectrum $\mathcal{L}(\Omega)$ determines $\Omega$ up to isometry.
While Gordon, Webb, and Wolpert~\cite{GordonWebbWolpert1992} answered Kac's question in the negative for domains with corners, both the Laplace and length spectral inverse problems remain open for smooth convex domains.

This global question has a local counterpart.
Rather than compare two arbitrary domains, one restricts to nearby ones and asks (i)~whether two nearby non-isometric domains can have the same length spectrum, or (ii)~whether the length spectrum can remain unchanged under a deformation.
We are concerned with the latter.
More precisely, a $C^{1}$ one-parameter family $\set{\Omega_{\tau}}_{\abs{\tau} \leqslant 1}$ of domains is \emph{dynamically isospectral} if $\mathcal{L}(\Omega_{\tau}) = \mathcal{L}(\Omega_{0})$ for each $\tau \in [-1, 1]$.
A planar domain $\Omega$ is \emph{dynamically spectrally rigid} in a class $\mathcal{M}$ of domains if every dynamically isospectral family $\set{\Omega_{\tau}}_{\abs{\tau} \leqslant 1}$ in $\mathcal{M}$ with $\Omega_{0} = \Omega$ consists of domains isometric to $\Omega$.

The study of this rigidity begins naturally at the infinitesimal level.
While rigidity is a condition on the family for all $\tau$, its linearization at $\tau = 0$ concerns only the initial velocity of the deformation.
For a $C^{1}$ family $\set{\Omega_{\tau}}_{\abs{\tau} \leqslant 1}$ of domains with $\Omega_{0} = \Omega$, differentiating the boundary at $\tau = 0$ gives a vector field along $\partial\Omega$; its tangential part merely reparametrizes the boundary, while its normal component defines a function on $\partial\Omega$, the \emph{deformation function} of the family at $\Omega$.
A planar domain $\Omega$ is \emph{infinitesimally dynamically spectrally rigid} in $\mathcal{M}$ if the deformation function at $\Omega$ of every dynamically isospectral family $\set{\Omega_{\tau}}_{\abs{\tau} \leqslant 1}$ in $\mathcal{M}$ with $\Omega_{0} = \Omega$ vanishes identically.
Rigidity implies its infinitesimal counterpart; the converse requires the deformation function to vanish at every member of the family, not at $\Omega$ alone.

While dynamical spectral rigidity has been established for strictly convex domains close to a circle~\cite{MR3665005}, where Lazutkin coordinates~\cite{Lazutkin_1973} make the billiard map nearly integrable near the boundary, the Bunimovich stadium lies at the opposite extreme, with flat edges, curvature discontinuities at the junction points, and hyperbolic, ergodic dynamics~\cite{Bunimovich_1979}.
In the Riemannian setting, Guillemin and Kazhdan~\cite{GuilleminKazhdan1980} proved that negatively curved surfaces are spectrally rigid for the Laplace spectrum, and Guillarmou and Lefeuvre~\cite{GuillarmouLefeuvre2019} recently proved, in all dimensions, that the marked length spectrum locally determines a non-positively curved Anosov metric.
De Simoi, Kaloshin, and Wei~\cite{MR3665005} asked whether Guillemin and Kazhdan's theorem has a counterpart for hyperbolic billiards.
In the affirmative direction, we prove that the circular Bunimovich stadium with sufficiently long flat edges is infinitesimally dynamically spectrally rigid among symmetric deformations.

\subsection{Main results}
\label{sub:main results}

For $a > 0$, let $\Omega_{*}$ be the circular Bunimovich stadium whose two parallel flat edges have length $a$ and whose two arcs are unit semicircles.
Let $\mathcal{M}^{r}$ denote the class of domains obtained from $\Omega_{*}$ by replacing its two circular arcs with strictly convex $C^{r}$ arcs, keeping both flat edges fixed.
The two arcs are required to meet the fixed flat edges with $C^{1}$ tangency at the four junction points, and the domain to be convex and symmetric under reflection across the midline parallel to the two flat edges (precise statement in Definition~\ref{def:Mr}).
Both $\Omega_{*}$ and $\mathcal{M}^{r}$ depend on $a$, which we suppress from the notation.
Since all domains of $\mathcal{M}^{r}$ share the flat edges, and hence the junction points, a $C^{1}$ family $\set{\Omega_{\tau}}_{\abs{\tau} \leqslant 1}$ in $\mathcal{M}^{r}$ with $\Omega_{0} = \Omega_{*}$ moves only the two arcs, so its deformation function at $\Omega_{*}$ vanishes on the flat edges and is determined by its restriction to the two semicircles of $\Omega_{*}$.

\begin{theorem}[Infinitesimal dynamical spectral rigidity of the circular Bunimovich stadium]\label{thm:main rigidity}
	Set $r = 4$. There exists $a_{0} > 0$ such that for each $a \geqslant a_{0}$, the deformation function at $\Omega_{*}$ of every dynamically isospectral $C^{1}$ family $\set{\Omega_{\tau}}_{\abs{\tau} \leqslant 1}$ in $\mathcal{M}^{r}$ with $\Omega_{0} = \Omega_{*}$ vanishes identically.
\end{theorem}

For arbitrary $a > 0$, we have the following finite-dimensionality result.
The bound on the dimension may depend on $a$.

\begin{theorem}[Finite-dimensionality of isospectral deformations]\label{thm:intro finite dimensional}
	Set $r = 4$. For each $a > 0$, the deformation functions at $\Omega_{*}$ of all dynamically isospectral $C^{1}$ families $\set{\Omega_{\tau}}_{\abs{\tau} \leqslant 1}$ in $\mathcal{M}^{r}$ with $\Omega_{0} = \Omega_{*}$ span a finite-dimensional space.
\end{theorem}

Theorem~\ref{thm:main rigidity} establishes infinitesimal rigidity at $\Omega_{*}$.
This naturally raises the question of whether local or global rigidity holds in the deformation class $\mathcal{M}^{r}$.
Let $\mathcal{M}^{r}_{\delta}$ denote the subclass of $\mathcal{M}^{r}$ formed by the domains whose arcs lie within $C^{r}$ distance $\delta$ of those of $\Omega_{*}$.

\begin{question}\label{q:full rigidity}
	\leavevmode
	\begin{enumerate}[label=\textup{(\roman*)}, leftmargin=*, topsep=2pt, itemsep=2pt, parsep=0pt]
		\item Is the circular Bunimovich stadium $\Omega_{*}$ dynamically spectrally rigid in $\mathcal{M}^{r}_{\delta}$ for some $\delta > 0$?
		\item Is every domain in $\mathcal{M}^{r}_{\delta}$ dynamically spectrally rigid in $\mathcal{M}^{r}_{\delta}$ for some $\delta > 0$?
		\item Is every domain in $\mathcal{M}^{r}$ dynamically spectrally rigid in $\mathcal{M}^{r}$?
	\end{enumerate}
\end{question}

\subsection{Related results}
\label{sub:related results}

De Simoi, Kaloshin, and Wei~\cite{MR3665005} proved dynamical spectral rigidity for $\z_2$-symmetric strictly convex domains that are sufficiently smooth and sufficiently close to a circle. 
Recently, Fierobe, Kaloshin, and Sorrentino~\cite{FierobeKaloshinSorrentino2025} extended this to domains close to an ellipse, under dihedral symmetry.
In the integrable setting, Huang and Kaloshin~\cite{HuangKaloshin2019} showed that the deformations of a rationally integrable strictly convex table preserving integrability near the boundary are tangent to a finite-dimensional space; Theorem~\ref{thm:intro finite dimensional} is the analogue of this for the stadium.
The same linearization underlies the perturbative results on the Birkhoff conjecture~\cite{AvilaDeSimoiKaloshin2016, HuangKaloshinSorrentino2018, KaloshinSorrentino2018, Koval2026}.

On the Laplace side, Hezari and Zelditch~\cite{HezariZelditch2012} proved that the ellipse is infinitesimally spectrally rigid among $C^{\infty}$ domains with its two reflection symmetries, and later~\cite{HezariZelditch2022} showed that ellipses of small eccentricity are determined by their Dirichlet spectrum among all smooth domains.

For billiards that are not strictly convex, much less is known.
De Simoi, Kaloshin, and Leguil \cite{DeSimoiKaloshinLeguil2023} showed that the marked length spectrum determines the geometry of a billiard obtained by removing from the plane finitely many strictly convex analytic obstacles satisfying a non-eclipse condition, under symmetry and genericity assumptions.
For three obstacles without symmetry, Osterman~\cite{Osterman2023} showed that two such systems have the same marked length spectrum if and only if their collision maps are analytically conjugate near a homoclinic orbit.
For stadia, Chen, Kaloshin, and Zhang~\cite{ChenKaloshinZhang2023} proved dynamical spectral rigidity in two classes of piecewise analytic domains, one of squash-type stadia and one of stadia whose two flat edges are the opposite sides of a fixed rectangle.
The second class, on which no symmetry is imposed, is the analytic analogue of the class $\mathcal{M}^{r}$ of this article, and $\Omega_{*}$ belongs to it.
For that class they prove that the deformation function of a dynamically isospectral family vanishes to infinite order at the four junction points, and the analyticity of the arcs then forces it to vanish identically.
In both open dispersing billiards~\cite{DeSimoiKaloshinLeguil2023, Osterman2023} and stadia~\cite{ChenKaloshinZhang2023}, real-analyticity was essential to establish rigidity from local information.
Theorem~\ref{thm:main rigidity} provides the first infinitesimal dynamical spectral rigidity result for a hyperbolic billiard in the finitely smooth category.

\subsection{Strategy and organization}
\label{sub:strategy}

As in~\cite{MR3665005}, we construct a linearized isospectral operator and reduce the rigidity problem to showing that its kernel is trivial.
Write $\Gamma^{p_1, p_2}$ for the \emph{hybrid periodic orbit} colliding $p_1$ times with one arc, $p_2$ times with the other, and never with the flat edges (Section~\ref{sec:Hybrid Periodic Orbit on Bunimovich Stadia}).
Each such orbit yields a linear functional, proportional to the first variation of its length, that annihilates the deformation function of every dynamically isospectral family (Lemma~\ref{lem:isospectral orbit functionals}); the linearized isospectral operator $\mathcal{J}$ collects these functionals over all pairs $(p_1, p_2)$, so the deformation function lies in its kernel.

To analyze the kernel, we expand the deformation function in Fourier series on the two arcs and the functionals in the Fourier basis.
The circular orbits $\Gamma^{q, q}$ force the deformation function to be anti-symmetric on the two arcs (Lemma~\ref{lem:standard stadium circular orbits}), reducing the unknowns to a single sequence $(\widehat{\mathbf{n}}_{j})_{j \geqslant 0}$ of Fourier coefficients.
Since the collision angles of $\Gamma^{q, mq}$ are asymptotic to $\pi/(2q)$ and $\pi/(2mq)$ (Section~\ref{sec:Hybrid Periodic Orbit on Bunimovich Stadia}), the Fourier coefficients of the corresponding functional are explicit up to corrections in powers of $1/q$.
Combining the functionals associated with $\Gamma^{q, 2q}$, $\Gamma^{q, 3q}$, and $\Gamma^{q, 4q}$ with weights chosen to cancel the leading correction, we obtain for each large $q$ an equation whose coefficient at the mode $q$ tends to $1$, while the other coefficients, weighted by $(q/j)^{\gamma}$ for a suitable exponent $\gamma \in (3, 4)$ (this is where $C^{4}$ smoothness is used), sum to less than $1$.
Hence, for every $a > 0$, the system is invertible on all sufficiently high Fourier modes, and every deformation in the kernel is determined by finitely many low-frequency modes (Theorem~\ref{thm:intro finite dimensional}).
To go further and prove that the kernel is trivial, we assume that the flat edges are sufficiently long in order to carry out a perturbation near infinity.
As their length tends to infinity, the collisions of each hybrid orbit become equally spaced along the arcs as in a circular billiard; perturbing off this explicit limit eliminates the remaining low-frequency modes, completing the proof of Theorem~\ref{thm:main rigidity}.

It is worth noting that the orbits $\Gamma^{q, q}$ and $\Gamma^{q, q + 1}$, the analogues on the stadium of the near-circular orbits used in~\cite{MR3665005}, do not work in our setting.
The orbits $\Gamma^{q, q}$ act identically on the two arcs and only yield the reduction noted above.
Based on the reduction, the functional corresponding to $\Gamma^{q, q + 1}$ has asymptotically equal coefficients at the Fourier modes $q$ and $q + 1$; since $(q / (q + 1))^{\gamma} \to 1$ for every $\gamma$, diagonal dominance fails, and the approach of~\cite{MR3665005} breaks down.
Instead, we consider orbits $\Gamma^{q, mq}$ with $m \geqslant 2$, for which the next leading mode off the diagonal is at $2 q$ and the factor $2^{-\gamma}$ restores diagonal dominance.

\smallskip

We now describe the structure of the article.
Section~\ref{sec:Hybrid Periodic Orbit on Bunimovich Stadia} constructs the hybrid periodic orbits and estimates their collision angles.
Section~\ref{sec:linearized isospectral operator} investigates the linearized isospectral operator and proves Theorem~\ref{thm:intro finite dimensional}.
Section~\ref{sec:triviality of the kernel} proves the triviality of the kernel for sufficiently large $a$, completing the proof of Theorem~\ref{thm:main rigidity}.

%% file: section/Hybrid_Periodic_Orbit.tex
\section{Hybrid Periodic Orbit on Bunimovich Stadia}
\label{sec:Hybrid Periodic Orbit on Bunimovich Stadia}

\subsection{Construction of the hybrid periodic orbits}
\label{sub:The construction of hybrid periodic orbit}

For $a > 0$, let $\Omega_{*} \subset \real^{2}$ be the circular Bunimovich stadium whose two flat edges are horizontal segments of length $a$, connected by two unit semicircular arcs $\Gamma_{1}$ (the left arc) and $\Gamma_{2}$ (the right arc).
The domain $\Omega_{*}$ is symmetric under reflection across the horizontal midline parallel to the two flat edges.

For two positive integers $p_1$ and $p_2$, a periodic billiard orbit $\Gamma^{p_1, p_2}$ is called a \emph{$(p_1, p_2)$-type hybrid periodic orbit} if:
\begin{enumerate}
    \smallskip
    \item The orbit has period $p_1 + p_2$, consisting of $p_1$ bounces on $\Gamma_{1}$ followed by $p_2$ bounces on $\Gamma_{2}$, and its segments never touch the flat edges;

    \smallskip
    \item The orbit is symmetric under reflection across the horizontal midline. Concretely, for the $p_1$ collisions on the left arc $\Gamma_{1}$:
    \begin{itemize}
        \item if $p_1$ is odd, the $\frac{p_1+1}{2}$-th collision point lies on the symmetry axis;
        \item if $p_1$ is even, the chord connecting the $\frac{p_1}{2}$-th and $(\frac{p_1}{2}+1)$-th collision points is perpendicular to the symmetry axis and bisected by it.
    \end{itemize}
    The same reflection symmetry holds for the $p_2$ collisions on the right arc $\Gamma_{2}$.
\end{enumerate}
The arcs meet the flat edges at four \emph{junction points}, and we call $\Gamma^{p_1, p_2}$ \emph{regular} if none of its collision points is a junction point.
The existence of such an orbit is justified in Lemma~\ref{lem:Taylor expansion of collision angle} below for the pairs $(q, q)$ with $q \in \n$ and, for each fixed $m \in \n$, the pairs $(q, mq)$ with $q$ sufficiently large.
That lemma also shows that a pair $\juxtapose{p_1}{p_2}$ carries at most one such orbit, so that the notation $\Gamma^{p_1, p_2}$ and the collision angles introduced below are unambiguous.

\begin{figure}[hpt!]
    \centering
        \includegraphics[width=0.8\linewidth]{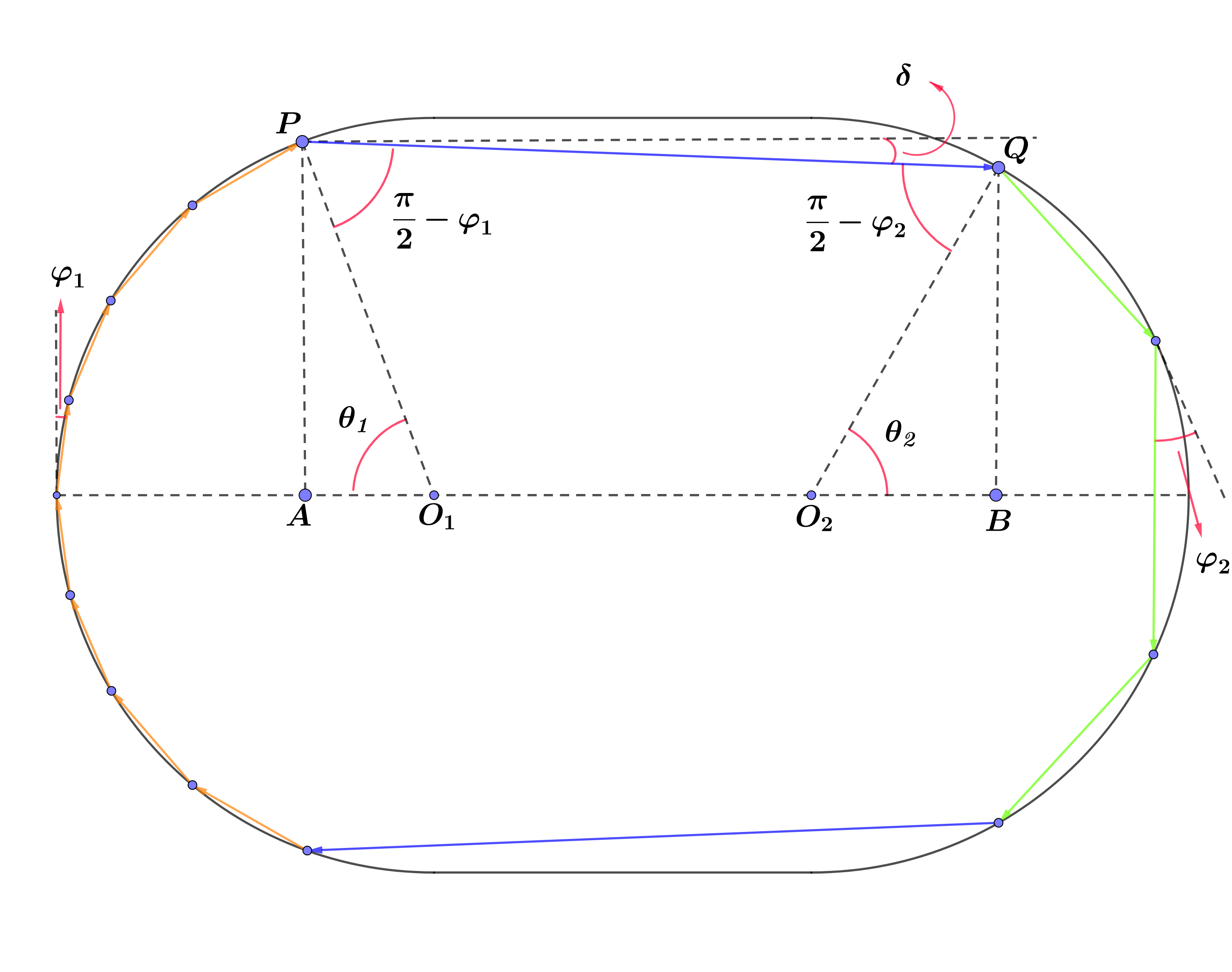}
    \caption{The $(p_1, p_2)$-type hybrid periodic orbit}
    \label{fig: 2q2h}
\end{figure}

We parametrize the arcs $\Gamma_1$ and $\Gamma_2$ by arclength parameters $s_1, s_2 \in [-\frac{\pi}{2}, \, \frac{\pi}{2}]$, normalized so that $s_1 = 0$ and $s_2 = 0$ are the intersections of the arcs with the horizontal symmetry axis.
In these coordinates, the collision points of $\Gamma^{p_1, p_2}$ on the left arc are
\begin{equation}\label{eq: Gamma qi s varphi}  
    \parentheses[\big]{ s_1^{p_1, p_2}(k), \varphi_1^{p_1, p_2}(k) } = \parentheses[\big]{  (2k-p_1-1)\varphi_1^{p_1, p_2}, \varphi_1^{p_1, p_2} } \quad  \text{for }  1 \leqslant k \leqslant p_1,
\end{equation}
where $\varphi_1^{p_1, p_2}$ is the angle the outgoing trajectory forms with the positively oriented tangent to the left arc.
Similarly, the collision points on the right arc are
\begin{equation}\label{eq: Gamma qi s varphi-}
    \parentheses[\big]{ s_2^{p_1, p_2}(k), \, \varphi_2^{p_1, p_2}(k) } = \parentheses[\big]{ (2k - p_2 - 1)\, \varphi_2^{p_1, p_2}, \, \varphi_2^{p_1, p_2} } \quad \text{for } 1 \leqslant k \leqslant p_2,
\end{equation}
with $\varphi_2^{p_1, p_2}$ defined analogously.

\subsection{Asymptotics of the collision angles}
\label{sub:Asymptotics of the collision angles}

The following lemma establishes the existence of $\Gamma^{q, mq}$ for each fixed $m \in \n$ and all sufficiently large $q$, and gives quantitative asymptotic estimates for the collision angles.

\begin{lemma}  \label{lem:Taylor expansion of collision angle}
    For each pair $\juxtapose{p_1}{p_2} \in \n$, the circular stadium $\Omega_{*}$ carries at most one $(p_1, p_2)$-type hybrid periodic orbit; in particular, the collision angles $\varphi_1^{p_1, p_2}$ and $\varphi_2^{p_1, p_2}$ are well defined.
    For each $m \in \n$, there exists an integer $q_{\ast} = q_{\ast}(m, a) \in \n$ such that for each integer $q \geqslant q_{\ast}$, the hybrid periodic orbit $\Gamma^{q, mq}$ exists and is regular.
    Moreover, for all sufficiently large $\juxtapose{p_{1}}{p_{2}} \in \n$ such that the hybrid periodic orbit $\Gamma^{p_1, p_2}$ exists, we have
    \begin{equation} \label{eq:Taylor expansion of collision angle}
        \begin{split}
            & \varphi_1^{p_1, p_2} = \frac{\pi}{2 p_1} + \frac{c_{3}}{p_1} \parentheses[\bigg]{ \frac{1}{p_2^2} - \frac{1}{p_1^2} } - \frac{c_5}{p_1} \parentheses[\bigg]{ \frac{1}{p_2^4}-\frac{1}{p_1^4} } + \mathcal{O}\parentheses[\bigg]{ \frac{1}{p_1} \parentheses[\bigg]{ \frac{1}{p_1^5} + \frac{1}{p_2^5} } }, \\
            & \varphi_2^{p_1, p_2} = \frac{\pi}{2 p_2} - \frac{c_{3}}{p_2} \parentheses[\bigg]{ \frac{1}{p_2^2} - \frac{1}{p_1^2} } + \frac{c_5}{p_2} \parentheses[\bigg]{ \frac{1}{p_2^4}-\frac{1}{p_1^4} } + \mathcal{O}\parentheses[\bigg]{ \frac{1}{p_2} \parentheses[\bigg]{ \frac{1}{p_1^5} + \frac{1}{p_2^5} } }, 
        \end{split}
    \end{equation}
    where $c_3 \define \frac{\pi^{2}}{8a}$ and $c_5 \define \frac{\pi^3}{16}\parentheses[\big]{ \frac{1}{a^2} + \frac{\pi}{24 a} }$.
    In particular, for each $q \in \n$ the hybrid periodic orbit $\Gamma^{q, q}$ exists and is regular, and $\varphi_1^{q, q} = \varphi_2^{q, q} = \frac{\pi}{2q}$.
\end{lemma}
    \begin{proof}
    Fix $p_1, p_2 \in \n$ and assume the orbit $\Gamma^{p_1, p_2}$ exists.
    For brevity we drop the superscripts, writing $\varphi_1, \varphi_2$ in place of $\varphi_1^{p_1, p_2}, \varphi_2^{p_1, p_2}$.
    Let $O_1$ and $O_2$ denote the centers of the left and right circular arcs, respectively.
    Let $P$ be the final bouncing point on the left arc and $Q$ the first on the right arc, with $PA$ and $QB$ perpendicular to the symmetry axis at $A$ and $B$, respectively.
    Set $\theta_1 \define \angle AO_1 P$ and $\theta_2 \define \angle BO_2 Q$, and let $\delta$ be the angle between $PQ$ and the horizontal line through $P$, counted positive when $\abs{AP} > \abs{BQ}$ (see Figure~\ref{fig: 2q2h}).
    Elementary geometry gives
     \[
        \theta_1= (p_1 - 1) \varphi_1 = \parentheses[\Big]{ \frac{\pi}{2} - \varphi_1 } + \delta \quad \text{and} \quad
        \theta_2= (p_2 - 1) \varphi_2 = \parentheses[\Big]{ \frac{\pi}{2} - \varphi_2 } - \delta.
    \]
    It follows that the angles $\varphi_1$, $\theta_1$, $\varphi_2$, $\theta_2$ are functions of $\xi_1$, $\xi_2$, $\delta$, given by 
    \begin{equation}  \label{eq:temp:lem:Taylor expansion of collision angle:relation of angle}
        \begin{split}
            & \varphi_1 = \xi_1 \parentheses[\Big]{ \frac{\pi}{2}+\delta }, \quad \text{and} \quad \theta_1 = \parentheses[\Big]{ \frac{\pi}{2}+\delta } (1-\xi_1); \\
            & \varphi_2 = \xi_2 \parentheses[\Big]{ \frac{\pi}{2}-\delta }, \quad \text{and} \quad \theta_2 = \parentheses[\Big]{ \frac{\pi}{2}-\delta } (1-\xi_2),
        \end{split}
    \end{equation}
    where $\xi_i \define 1/p_i$ for each $i \in \set{1, \, 2}$.
    Note that $\theta_{1} \leqslant \pi / 2$ and $\theta_{2} \leqslant \pi / 2$.
    This is equivalent to
    \begin{equation}    \label{eq:temp:lem:Taylor expansion of collision angle:boundary condition}
        - \frac{\xi_{2}}{1 - \xi_{2}} \frac{\pi}{2} \leqslant \delta \leqslant \frac{\xi_{1}}{1 - \xi_{1}} \frac{\pi}{2}.
    \end{equation}
    For $p_i = 1$ the corresponding fraction is to be read as $+\infty$, since the angle $\theta_i$ vanishes identically and imposes no constraint.

    If $\Gamma^{p_1, p_2}$ exists, the following relation follows from elementary geometry (Figure~\ref{fig: 2q2h}):
    \begin{equation} \label{eq:temp:lem:Taylor expansion of collision angle:geometric relation}
        \tan \delta = \frac{\abs{AP} - \abs{BQ}}{\abs{AB}} = \frac{\sin\theta_1 -  \sin\theta_2}{ a + \cos\theta_1 + \cos\theta_2 }.
    \end{equation}
    In view of~\eqref{eq:temp:lem:Taylor expansion of collision angle:geometric relation}, we define
    \[
        F(\xi_1, \xi_2, \delta) \define \parentheses{ \sin\theta_1 -  \sin\theta_2 } - \tan\delta \parentheses{ a + \cos\theta_1 + \cos\theta_2 },
    \]
    where $\theta_{1}$ and $\theta_{2}$ are given by \eqref{eq:temp:lem:Taylor expansion of collision angle:relation of angle} and depend on $\xi_1$, $\xi_2$, and $\delta$.
    Then for each pair of periods $\juxtapose{p_1}{p_2} \in \n$, the $(p_1, p_2)$-type hybrid periodic orbit $\Gamma^{p_1, p_2}$ exists if and only if there exists $\delta \in (- \frac{\pi}{2}, \frac{\pi}{2})$ such that $F(\xi_1, \xi_2, \delta) = 0$ and \eqref{eq:temp:lem:Taylor expansion of collision angle:boundary condition} holds.
    Necessity follows from the relations just derived.
    Conversely, suppose $\delta \in (- \frac{\pi}{2}, \frac{\pi}{2})$ satisfies $F(\xi_1, \xi_2, \delta) = 0$ together with~\eqref{eq:temp:lem:Taylor expansion of collision angle:boundary condition}, and define $\varphi_i$ and $\theta_i$ by~\eqref{eq:temp:lem:Taylor expansion of collision angle:relation of angle}.
    Place collision points on the two arcs at the positions prescribed by~\eqref{eq: Gamma qi s varphi} and~\eqref{eq: Gamma qi s varphi-}; they lie on the closed arcs, since $\abs{s_i} \leqslant (p_i - 1)\, \varphi_i = \theta_i$ and $\theta_i \leqslant \frac{\pi}{2}$ by~\eqref{eq:temp:lem:Taylor expansion of collision angle:boundary condition}.
    Within each cycle, consecutive points subtend the central angle $2 \varphi_i$, so every chord of the cycle makes the angle $\varphi_i$ with the tangents at its endpoints, and the reflection law holds at each interior collision.
    The segment joining the final point $P$ of the left cycle and the first point $Q$ of the right cycle has inclination $\delta'$ with $\tan \delta' = \frac{\sin\theta_1 - \sin\theta_2}{a + \cos\theta_1 + \cos\theta_2}$, by the same elementary geometry as in~\eqref{eq:temp:lem:Taylor expansion of collision angle:geometric relation}; the equation $F(\xi_1, \xi_2, \delta) = 0$ states that $\delta' = \delta$, and the relations~\eqref{eq:temp:lem:Taylor expansion of collision angle:relation of angle} then express the reflection law at $P$ and at $Q$.
    The lower crossing segment is the mirror image of $PQ$ across the symmetry axis, with the reflection law at its endpoints holding by symmetry, so the trajectory closes and is symmetric under reflection across the horizontal midline.
    Finally, each segment of the trajectory joins two boundary points that do not lie on a common flat edge, so by convexity its interior lies in the interior of $\Omega_{*}$; in particular the flat edges are untouched.
    This proves the stated equivalence.

    The admissible $\delta$ is moreover unique.
    Indeed, set
    \begin{equation}    \label{eq:temp:lem:Taylor expansion of collision angle:H}
        \begin{split}
            H(\xi_1, \xi_2, \delta) &\define F(\xi_1, \xi_2, \delta) \cos\delta \\
            &= \cos\delta \parentheses{ \sin\theta_1 - \sin\theta_2 } - \sin\delta \parentheses{ a + \cos\theta_1 + \cos\theta_2 } \\
            &= -a \sin\delta + \sin(\theta_1 - \delta) - \sin(\theta_2 + \delta),
        \end{split}
    \end{equation}
    which is smooth on $\parentheses[\big]{ -\frac{\pi}{2}, \frac{\pi}{2} }$ and has the same zeros as $F$, since $\cos\delta > 0$.
    Differentiating~\eqref{eq:temp:lem:Taylor expansion of collision angle:H} with respect to $\delta$ and using $\partial_{\delta}\theta_1 - 1 = -\xi_1$ and $\partial_{\delta}\theta_2 + 1 = \xi_2$, we obtain
    \[
        \partial_{\delta} H = -a \cos\delta - \xi_1 \cos\parentheses{ \theta_1 - \delta } - \xi_2 \cos\parentheses{ \theta_2 + \delta }.
    \]
    Since $\theta_1 - \delta = \frac{\pi}{2} - \varphi_1$ and $\theta_2 + \delta = \frac{\pi}{2} - \varphi_2$ by~\eqref{eq:temp:lem:Taylor expansion of collision angle:relation of angle}, this becomes
    \begin{equation}    \label{eq:temp:lem:Taylor expansion of collision angle:H derivative}
        \partial_{\delta} H(\xi_1, \xi_2, \delta) = -a \cos\delta - \xi_1 \sin\varphi_1 - \xi_2 \sin\varphi_2 .
    \end{equation}
    Since $\xi_i \in (0, 1]$ and $\abs{\delta} < \frac{\pi}{2}$, we have $\varphi_i = \xi_i \parentheses[\big]{ \frac{\pi}{2} \pm \delta } \in (0, \pi)$, so both sines are positive and
    \begin{equation}    \label{eq:temp:lem:Taylor expansion of collision angle:H monotone}
        \partial_{\delta} H(\xi_1, \xi_2, \delta) < -a \cos\delta < 0
        \qquad \text{for } \abs{\delta} < \frac{\pi}{2},
    \end{equation}
    with no restriction on $a > 0$ and none on $\xi_1$, $\xi_2$ beyond $\xi_i \in (0, 1]$.
    Hence $H$, and with it $F$, has at most one zero in $\parentheses[\big]{ -\frac{\pi}{2}, \frac{\pi}{2} }$, showing that the pair $\juxtapose{p_1}{p_2}$ determines $\Gamma^{p_1, p_2}$.

    Since $F(0, 0, 0) = 0$ and $\partial_{\delta} F(0, 0, 0) = - a < 0$, it follows from the implicit function theorem that there exists a neighborhood $\mathcal{U}$ of $(0, 0, 0)$ where the equation $F(\xi_1, \xi_2, \delta) = 0$ has a unique solution $\delta = \delta(\xi_1, \xi_2)$.
    In particular, by the continuity of $\partial_{\delta} F$, there exist $\delta_{0} > 0$ and integers $\juxtapose{p_1^0}{p_2^0} \in \n$ such that $\mathcal{U}_{0} \define (- 1 / p_1^0, 1 / p_1^0)\times (- 1 / p_2^0, 1 / p_2^0)\times (-\delta_0, \delta_0) \subseteq \mathcal{U}$.
    Here the constants $p_1^0$, $p_2^0$, and $\delta_{0}$ may depend on $a$.

    \smallskip

    We claim that for each $m \in \n$, there exists an integer $q_{\ast} = q_{\ast}(m, a) \in \n$ such that for each integer $q \geqslant q_{\ast}$, the $(q, mq)$-type hybrid periodic orbit $\Gamma^{q, mq}$ exists.

    \smallskip

    Fix $m \in \n$.
    For $\xi > 0$, set $(\xi_1, \xi_2) = (\xi, \xi/m)$.
    On the one hand, note that $\xi_{1} \geqslant \xi_{2}$ implies that $F(\xi_{1}, \xi_{2}, 0) = \sin\theta_1 - \sin\theta_2 \leqslant 0$.
    On the other hand, at $\delta = - \frac{\xi_{2}}{1 - \xi_{2}} \frac{\pi}{2}$ we have $\theta_{2} = \frac{\pi}{2}$, hence $\sin\theta_{2} = 1$ and $\cos\theta_{2} = 0$, so
    \[
        F\parentheses[\big]{ \xi_{1}, \xi_{2}, - \frac{\xi_{2}}{1 - \xi_{2}} \frac{\pi}{2} }
        = \sin\theta_{1} - 1 + \tan\parentheses[\Big]{ \frac{\xi_{2}}{1 - \xi_{2}} \frac{\pi}{2} } \parentheses[\big]{ a + \cos\theta_{1} }.
    \]
    As $\xi \to 0^{+}$ with $(\xi_1, \xi_2) = (\xi, \xi/m)$, Taylor expansion gives
    \[
        \tan\parentheses[\Big]{ \frac{\xi_2}{1 - \xi_2} \frac{\pi}{2} } \parentheses[\big]{ a + \cos\theta_1 }
        = \frac{\pi a}{2m}\, \xi + \mathcal{O}(\xi^2),
        \qquad
        \sin\theta_1 - 1 = \mathcal{O}(\xi^2).
    \]
    Therefore
    \[
        F\parentheses[\Big]{ \xi_1, \xi_2, -\frac{\xi_2}{1 - \xi_2}\frac{\pi}{2} }
        = \frac{\pi a}{2m}\, \xi + \mathcal{O}(\xi^2) > 0,
    \]
    for all sufficiently small $\xi > 0$.
    Since $\frac{\xi/m}{1 - \xi/m}\frac{\pi}{2} \to 0$ as $\xi \to 0^{+}$, we may choose $\xi_0 = \xi_0(m, a) > 0$ sufficiently small so that
    \[
        \xi_0 < \min\set[\big]{1/p_1^0, \, 1/p_2^0},
        \qquad
        \frac{\xi_0/m}{1 - \xi_0/m}\frac{\pi}{2} < \delta_0,
    \]
    and
    \[
        F\parentheses[\Big]{ \xi, \frac{\xi}{m}, -\frac{\xi/m}{1 - \xi/m}\frac{\pi}{2} } > 0,
    \]
    for each $\xi \in (0, \xi_0)$.
    Set $V_m(\xi_0) \define \set[\big]{ (\xi, \, \xi/m) \describe \xi \in (0, \xi_0) }$.
    For each $(\xi_1, \xi_2) \in V_m(\xi_0)$, we therefore have both $F(\xi_1, \xi_2, 0) \leqslant 0$ and $F\parentheses[\big]{\xi_1, \xi_2, -\frac{\xi_2}{1 - \xi_2}\frac{\pi}{2}} > 0$.
    The intermediate value theorem gives a zero in $[-\frac{\xi_2}{1 - \xi_2}\frac{\pi}{2}, 0]$.
    Since this interval lies in $(-\delta_0, \delta_0)$, uniqueness in $\mathcal{U}_0$ shows that this zero is $\delta(\xi_1, \xi_2)$; hence
    \[
        -\frac{\xi_2}{1 - \xi_2}\frac{\pi}{2}
        \leqslant \delta(\xi_1, \xi_2)
        \leqslant 0
        \leqslant \frac{\xi_1}{1 - \xi_1}\frac{\pi}{2}.
    \]
    Thus $\delta(\xi_1, \xi_2)$ satisfies the boundary condition~\eqref{eq:temp:lem:Taylor expansion of collision angle:boundary condition}.
    The left-hand inequality is in fact strict, and this is what makes the resulting orbit regular.
    Indeed, $F\parentheses[\big]{\xi_1, \xi_2, -\frac{\xi_2}{1 - \xi_2}\frac{\pi}{2}} > 0$ while $F$ vanishes at $\delta(\xi_1, \xi_2)$, so $\delta(\xi_1, \xi_2) > -\frac{\xi_2}{1 - \xi_2}\frac{\pi}{2}$, whence $\theta_2 < \frac{\pi}{2}$ by~\eqref{eq:temp:lem:Taylor expansion of collision angle:relation of angle}; and $\delta(\xi_1, \xi_2) \leqslant 0$ gives $\theta_1 \leqslant \frac{\pi}{2}(1 - \xi_1) < \frac{\pi}{2}$.
    The collision points on the $i$-th arc lie within angular distance $\theta_i$ of the midpoint of that arc, so none of them is a junction point.
    Choose an integer $q_{\ast} = q_{\ast}(m, a) \in \n$ such that $q_{\ast} > 1/\xi_0$.
    For each integer $q \geqslant q_{\ast}$, we have $(1/q, 1/(mq)) \in V_m(\xi_0)$, and the equivalence established above gives the existence of the hybrid periodic orbit $\Gamma^{q, mq}$, which is regular by the previous paragraph.
    This proves the claim.

    \smallskip

    Fix an integer $P_0 = P_0(a) \in \n$ with $P_0 > \max\set[\big]{ p_1^0, \, p_2^0 }$ and $\frac{1/P_0}{1 - 1/P_0}\frac{\pi}{2} < \delta_0$, and suppose that the hybrid periodic orbit $\Gamma^{p_{1}, p_{2}}$ exists with $\juxtapose{p_1}{p_2} \geqslant P_0$.
    The admissibility condition~\eqref{eq:temp:lem:Taylor expansion of collision angle:boundary condition} gives
    \[
        \abs{\delta}
        \leqslant \frac{\pi}{2} \max\set[\Big]{ \frac{\xi_1}{1 - \xi_1}, \, \frac{\xi_2}{1 - \xi_2} }
        \leqslant \frac{1/P_0}{1 - 1/P_0}\frac{\pi}{2}
        < \delta_0,
    \]
    so the triple $(\xi_1, \xi_2, \delta)$ lies in $\mathcal{U}_0$, and $\delta$ agrees with the local solution $\delta(\xi_1, \xi_2)$.
    Since $F$ is real-analytic, the implicit function theorem further implies that the function $\delta=\delta(\xi_1, \xi_2)$ is real-analytic.
    Applying the linear transformation 
    $\begin{cases}
        \xi_1 = u + v, \\
        \xi_2 = u - v,
    \end{cases}$
    we consider $\delta$ as a real-analytic function of $(u, v)$ with $\lim_{(u, v) \to (0, 0)} \delta(u, v) = 0$.
    It follows from \eqref{eq:temp:lem:Taylor expansion of collision angle:relation of angle} that
    \[
        \begin{split}
            \sin\theta_1 - \sin\theta_2
            &= 2\cos\frac{\theta_1 + \theta_2}{2} \sin\frac{\theta_1 - \theta_2}{2}
            = 2\sin\parentheses[\Big]{ \frac{\pi}{2} u + \delta v } \sin\parentheses[\Big]{ -\frac{\pi}{2} v + \delta - u\delta }, \\
            \cos\theta_1+\cos\theta_2 
            &= 2\cos\frac{\theta_1 + \theta_2}{2} \cos\frac{\theta_1 - \theta_2}{2}
            = 2\sin\parentheses[\Big]{ \frac{\pi}{2} u + \delta v } \cos\parentheses[\Big]{ -\frac{\pi}{2} v + \delta - u\delta } .
        \end{split}
    \]
    Setting $A \define \frac{\pi}{2} u + \delta v$ and $B \define -\frac{\pi}{2} v + \delta(1 - u)$, the relation $F(\xi_1, \xi_2, \delta) = 0$ takes the form
    \begin{equation} \label{eq:temp:lem:Taylor expansion of collision angle:F=0}
        2\sin A \sin B - \tan\delta \parentheses[\big]{ a + 2 \sin A \cos B } = 0.
    \end{equation}
    To determine the asymptotic expansion of the real-analytic function $\delta(u, v)$ near the origin, set $\rho \define \sqrt{u^2 + v^2}$.
    Since $\delta(u, v) \to 0$ as $\rho \to 0$, expanding the factors in~\eqref{eq:temp:lem:Taylor expansion of collision angle:F=0} to second order in $\rho$ yields
    \[
        2\parentheses[\Big]{ \frac{\pi}{2} u }\parentheses[\Big]{ -\frac{\pi}{2} v + \delta } - \delta (a + \pi u) = \mathcal{O}(\rho^3).
    \]
    The mixed terms $\pi u \delta$ cancel, leaving $-\frac{\pi^2}{2} uv - a\delta = \mathcal{O}(\rho^3)$, or equivalently,
    \begin{equation} \label{eq:temp:lem:Taylor expansion of collision angle:delta exp 2}
        \delta = -\frac{\pi^2}{2a}\, uv + \mathcal{O}(\rho^3).
    \end{equation}
    In particular, $\delta = \mathcal{O}(\rho^2)$.

    To determine the next non-vanishing term in the expansion, we expand~\eqref{eq:temp:lem:Taylor expansion of collision angle:F=0} to fourth order in $\rho$.
    Since $\delta = \mathcal{O}(\rho^2)$, we have $\tan\delta = \delta + \mathcal{O}(\rho^6)$, while
    \[
        2\sin A \sin B = 2AB - \frac{1}{3} AB \bigl(A^2 + B^2\bigr) + \mathcal{O}(\rho^5)
    \]
    and
    \[
        \tan\delta \bigl(a + 2\sin A \cos B\bigr) = \delta (a + 2A) + \mathcal{O}(\rho^5) = a\delta + \pi u \delta + \mathcal{O}(\rho^5).
    \]
    Substituting $A = \frac{\pi}{2} u + \delta v$ and $B = -\frac{\pi}{2} v + \delta(1 - u)$, we find
    \[
        2AB = -\frac{\pi^2}{2} uv + \pi u \delta - \pi (u^2 + v^2) \delta + \mathcal{O}(\rho^5),
        \qquad
        -\frac{1}{3} AB \bigl(A^2 + B^2\bigr) = \frac{\pi^4}{48}\, uv\, (u^2 + v^2) + \mathcal{O}(\rho^5).
    \]
    The terms $\pi u \delta$ cancel once again, and~\eqref{eq:temp:lem:Taylor expansion of collision angle:F=0} simplifies to
    \[
        -\frac{\pi^2}{2} uv - a\delta + \parentheses[\Big]{ \frac{\pi^4}{48} uv - \pi \delta } (u^2 + v^2) = \mathcal{O}(\rho^5).
    \]
    Inserting the leading-order expansion~\eqref{eq:temp:lem:Taylor expansion of collision angle:delta exp 2} into the bracket gives
    \[
        a\delta + \frac{\pi^2}{2} uv = \parentheses[\Big]{ \frac{\pi^3}{2a} + \frac{\pi^4}{48} }\, uv\, (u^2 + v^2) + \mathcal{O}(\rho^5),
    \]
    which yields
    \begin{equation} \label{eq:temp:lem:Taylor expansion of collision angle:delta exp 4}
        \delta = -4 c_3 uv + 8 c_5 uv (u^2 + v^2) + \mathcal{O}\bigl((u^2 + v^2)^{5/2}\bigr),
    \end{equation}
    where $c_3 \define \frac{\pi^2}{8a}$ and $c_5 \define \frac{\pi^3}{16} \bigl( \frac{1}{a^2} + \frac{\pi}{24a} \bigr)$.
    Recall that 
    \[
        \begin{cases}
            u = (\xi_1 + \xi_2) / 2, \\
            v = (\xi_1 - \xi_2) / 2.
        \end{cases}
    \]
    Then  
    \begin{equation}    \label{eq:temp:lem:Taylor expansion of collision angle:uv xi}
        uv = \frac{\xi_1^2 - \xi_2^2}{4} = \frac{1/p_1^2 - 1/p_2^2}{4} \quad \text{and} \quad
        uv(u^2 + v^2) = \frac{\xi_1^4-\xi_2^4}{8}=\frac{1/p_1^4-1/p_2^4}{8}.
    \end{equation}
    Then the expansion~\eqref{eq:Taylor expansion of collision angle} follows from~\eqref{eq:temp:lem:Taylor expansion of collision angle:relation of angle}, \eqref{eq:temp:lem:Taylor expansion of collision angle:delta exp 4} and
    \eqref{eq:temp:lem:Taylor expansion of collision angle:uv xi}.

    Finally, let $p_1 = p_2 = q$, so that $\xi_1 = \xi_2$.
    The map $\delta \mapsto -\delta$ then exchanges $\theta_1$ and $\theta_2$ by~\eqref{eq:temp:lem:Taylor expansion of collision angle:relation of angle}, so $F(\xi_1, \xi_2, \cdot)$ is odd; in particular $F(\xi_1, \xi_2, 0) = 0$, and $\delta = 0$ is its only zero by the uniqueness proved above.
    The orbit reduces to two regular $q$-gon inscriptions connected by a horizontal segment;
    the boundary condition $\theta_i = (q - 1) \frac{\pi}{2q} < \frac{\pi}{2}$ is satisfied for each $q \in \n$, and is strict, so no collision point is a junction point.
    Therefore the hybrid periodic orbit $\Gamma^{q, q}$ exists and is regular for each $q \in \n$, and $\varphi_1^{q, q} = \varphi_2^{q, q} = \frac{\pi}{2q}$.
\end{proof}

The following corollary records the Taylor expansions of the collision angles of $\Gamma^{q, mq}$.

\begin{corollary}  \label{cor:Taylor expansion of collision angle}
    For each $m \in \n$, as $q \to +\infty$,
    \begin{equation} \label{eq:cor:Taylor expansion of collision angle}
        \begin{split}
            \varphi_1^{q, mq} &= \frac{\pi}{2 q} - \frac{ c_{3} \parentheses[\big]{ 1 - m^{-2} } }{q^{3}} + \mathcal{O}\parentheses[\big]{ q^{-5} }
            \quad \text{and}   \\
            \varphi_2^{q, mq} &= \frac{\pi}{2 mq} + \frac{ c_{3} \parentheses[\big]{ 1 - m^{-2} } }{ m q^{3} } + \mathcal{O}\parentheses[\big]{ q^{-5} },
        \end{split}
    \end{equation}
    where $c_3 \define \frac{\pi^{2}}{8a}$.
\end{corollary}
\begin{proof}
    The expansion~\eqref{eq:cor:Taylor expansion of collision angle} immediately follows from~\eqref{eq:Taylor expansion of collision angle} in Lemma~\ref{lem:Taylor expansion of collision angle}.
\end{proof}

%% file: section/lemmas_orbits_existence.tex
\subsection{Existence of hybrid orbits for long flat edges}
\label{sub:existence for long flat edges}

Lemma~\ref{lem:Taylor expansion of collision angle} shows that the orbits $\Gamma^{q, mq}$ exist for each fixed $m$ and all sufficiently large $q$, with a threshold depending on $a$.
In this subsection we reverse the quantifiers and prove that, once $a \geqslant 12$, the orbits $\Gamma^{q, mq}$ with $m \in \set{2, \, 3, \, 4}$ exist for every $q \geqslant 1$, with collision-angle estimates uniform in $q$.

Given integers $q \geqslant 1$ and $m \in \set{2, \, 3, \, 4}$, set $\xi_1 \define 1/q$ and $\xi_2 \define 1/(mq)$, and recall from the proof of Lemma~\ref{lem:Taylor expansion of collision angle} the function
\begin{equation}\label{eq:orbits:F}
	F(\xi_1, \xi_2, \delta)
	= \parentheses[\big]{ \sin\theta_1 - \sin\theta_2 }
	- \tan\delta \parentheses[\big]{ a + \cos\theta_1 + \cos\theta_2 },
\end{equation}
where $\theta_1 = \parentheses[\big]{ \frac{\pi}{2} + \delta }(1 - \xi_1)$ and $\theta_2 = \parentheses[\big]{ \frac{\pi}{2} - \delta }(1 - \xi_2)$ by the angle relations~\eqref{eq:temp:lem:Taylor expansion of collision angle:relation of angle}.
The orbit $\Gamma^{q, mq}$ exists if and only if $F(\xi_1, \xi_2, \cdot)$ vanishes at some $\delta$ satisfying the admissibility condition~\eqref{eq:temp:lem:Taylor expansion of collision angle:boundary condition}, in which case the collision angles are
\begin{equation}\label{eq:orbits:angle corrections}
	\varphi_1^{q, mq} = \frac{\pi}{2q} + \varepsilon_1,
	\qquad
	\varphi_2^{q, mq} = \frac{\pi}{2mq} + \varepsilon_2,
	\qquad\text{with}\quad
	\varepsilon_1 = \frac{\delta}{q},
	\quad
	\varepsilon_2 = -\frac{\delta}{mq}.
\end{equation}
For admissible $\delta$, both angles lie in $\bigl[0, \frac{\pi}{2}\bigr]$; since $m \geqslant 2$ forces $\xi_1 > \xi_2$ and hence $\theta_1 < \theta_2$ at $\delta = 0$, we obtain
\begin{equation}\label{eq:orbits:sign at zero}
	F(\xi_1, \xi_2, 0)
	= \parentheses[\big]{ \sin\theta_1 - \sin\theta_2 } \Big|_{\delta = 0}
	< 0.
\end{equation}
To establish existence and quantitative control, recall from the proof of Lemma~\ref{lem:Taylor expansion of collision angle} that the auxiliary function
\begin{equation}\label{eq:orbits:H}
	H(\xi_1, \xi_2, \delta) \define F(\xi_1, \xi_2, \delta) \cos\delta
\end{equation}
shares the zeros of $F$ on $(-\pi/2, \pi/2)$, and its derivative satisfies the identity
\begin{equation}\label{eq:orbits:dH}
	\partial_\delta H(\xi_1, \xi_2, \delta) = -a \cos\delta - \xi_1 \sin\varphi_1 - \xi_2 \sin\varphi_2.
\end{equation}
Since the collision angles lie in $(0, \pi)$, the two sine terms are strictly positive; discarding them yields the global monotonicity bound
\begin{equation}\label{eq:orbits:monotonicity}
	\partial_\delta H(\xi_1, \xi_2, \delta) \leqslant -a \cos\delta
	\qquad \text{for } \abs{\delta} < \frac{\pi}{2},
\end{equation}
which, unlike its counterpart for $F$, holds on the whole interval and for every $a > 0$.
Integrating~\eqref{eq:orbits:monotonicity} yields
\begin{equation}\label{eq:orbits:integrated}
	H(\xi_1, \xi_2, \delta) - H(\xi_1, \xi_2, \delta')
	\geqslant a \parentheses[\big]{ \sin\delta' - \sin\delta }
	\qquad \parentheses[\Big]{ -\frac{\pi}{2} < \delta \leqslant \delta' < \frac{\pi}{2} },
\end{equation}
and, taking one of $\delta$, $\delta'$ to be a zero of $F$ and the other to be $0$, where $H(\xi_1, \xi_2, 0) = F(\xi_1, \xi_2, 0)$,
\begin{equation}\label{eq:orbits:sine bound}
	a \abs{\sin\delta} \leqslant \abs[\big]{ F(\xi_1, \xi_2, 0) }
	\qquad \text{at every zero } \delta \text{ of } F \text{ in } \parentheses[\Big]{ -\frac{\pi}{2}, \frac{\pi}{2} }.
\end{equation}
Since $\abs{\delta} \leqslant \frac{\pi}{2} \abs{\sin\delta}$ on $\bigl[-\frac{\pi}{2}, \frac{\pi}{2}\bigr]$, the bound~\eqref{eq:orbits:sine bound} gives $\abs{\delta} \leqslant \frac{2}{a} \abs[\big]{ F(\xi_1, \xi_2, 0) }$.

\begin{lemma}
\label{lem:short orbit control}
	For each $a \geqslant 12$ and each $m \in \set{2, \, 3, \, 4}$, the orbit $\Gamma^{1, m}$ exists and is regular, and its angle corrections in~\eqref{eq:orbits:angle corrections} satisfy $\abs{\varepsilon_1} + \abs{\varepsilon_2} \leqslant 3 c_3$.
\end{lemma}
\begin{proof}
	For $q = 1$ we have $\xi_1 = 1$ and $\theta_1 = 0$, so~\eqref{eq:orbits:sign at zero} reads
	\[
		F\parentheses[\Big]{1, \frac{1}{m}, 0}
		= -\sin\parentheses[\bigg]{ \frac{\pi}{2} \parentheses[\Big]{ 1 - \frac{1}{m} } }
		\in (-1, 0).
	\]
	The interval $\bigl[-\frac{\pi}{6}, 0\bigr]$ is admissible for each $m \in \set{2, \, 3, \, 4}$.
	Indeed, the angle $\theta_1$ vanishes identically, so admissibility reduces to the lower constraint in~\eqref{eq:temp:lem:Taylor expansion of collision angle:boundary condition}, whose endpoint equals $-\frac{\pi}{2(m - 1)} \leqslant -\frac{\pi}{6}$.
	The integrated bound~\eqref{eq:orbits:integrated}, applied with $\delta = -\frac{\pi}{6}$ and $\delta' = 0$, therefore gives
	\[
		H\parentheses[\Big]{1, \frac{1}{m}, -\frac{\pi}{6}}
		\geqslant F\parentheses[\Big]{1, \frac{1}{m}, 0} + \frac{a}{2}
		\geqslant -1 + 6
		> 0,
	\]
	and hence $F\parentheses[\big]{1, \frac{1}{m}, -\frac{\pi}{6}} > 0$, since $H = F \cos\delta$ and $\cos\frac{\pi}{6} > 0$.
	The intermediate value theorem now gives an admissible zero $\delta \in \parentheses[\big]{-\frac{\pi}{6}, 0}$ of $F$, unique by Lemma~\ref{lem:Taylor expansion of collision angle}.
	Thus $\Gamma^{1, m}$ exists, and it is regular: $\theta_1$ vanishes identically, while $\delta > -\frac{\pi}{6} \geqslant -\frac{\pi}{2(m - 1)}$ gives $\theta_2 < \frac{\pi}{2}$.
	The bound~\eqref{eq:orbits:sine bound}, at this zero, gives $\abs{\delta} \leqslant \frac{2}{a} \abs[\big]{F(1, 1/m, 0)} \leqslant \frac{2}{a}$, whence, by~\eqref{eq:orbits:angle corrections},
	\[
		\abs{\varepsilon_1} + \abs{\varepsilon_2}
		\leqslant \frac{3}{2} \abs{\delta}
		\leqslant \frac{3}{a}
		= \frac{24}{\pi^2}\, c_3
		\leqslant 3 c_3.
		\qedhere
	\]
\end{proof}
Only $m \in \set{2, \, 3}$ is used below; the value $m = 4$ is included for uniformity with Lemma~\ref{lem:tail orbit asymptotics}.

\begin{lemma}	\label{lem:tail orbit asymptotics}
	For each $a \geqslant 12$, each integer $q \geqslant 2$, and each $m \in \set{2, \, 3, \, 4}$, the orbit $\Gamma^{q, mq}$ exists and is regular, and its angle corrections in~\eqref{eq:orbits:angle corrections} satisfy
	\[
		\varepsilon_1 = -c_3\, \bigl(1 - m^{-2}\bigr)\, q^{-3} + \mathcal{O}\bigl(c_3\, q^{-5}\bigr),
		\qquad
		\varepsilon_2 = c_3\, m^{-1} \bigl(1 - m^{-2}\bigr)\, q^{-3} + \mathcal{O}\bigl(c_3\, q^{-5}\bigr),
	\]
	with absolute implicit constants.
	In particular, we have $\abs{\varepsilon_1} + \abs{\varepsilon_2} \leqslant C_0\, c_3\, q^{-3}$ for an absolute constant $C_0 > 0$.
\end{lemma}
\begin{proof}
	At the lower endpoint $\delta_- \define -\frac{\xi_2}{1 - \xi_2} \frac{\pi}{2} = -\frac{\pi}{2(mq - 1)}$ of the admissible interval we have $\theta_2 = \frac{\pi}{2}$, so
	\[
		F(\xi_1, \xi_2, \delta_-)
		= \sin\theta_1 - 1 + \tan\varrho\, \parentheses[\big]{ a + \cos\theta_1 },
		\qquad
		\varrho \define -\delta_- = \frac{\pi}{2(mq - 1)}.
	\]
	Since $mq - 1 \geqslant q$, we have $\varrho \leqslant \frac{\pi}{2q}$, so the angle $\alpha \define \frac{\pi}{2} - \theta_1 = \frac{\pi}{2q} + \varrho \parentheses[\big]{1 - \frac{1}{q}}$ lies in $\bigl[0, \frac{\pi}{q}\bigr]$, whence
	\[
		1 - \sin\theta_1
		= 1 - \cos\alpha
		\leqslant \frac{\alpha^2}{2}
		\leqslant \frac{\pi^2}{2 q^2};
	\]
	moreover, we have $\tan\varrho > \varrho \geqslant \frac{\pi}{2mq} \geqslant \frac{\pi}{8q}$.
	Therefore
	\[
		F(\xi_1, \xi_2, \delta_-)
		> -\frac{\pi^2}{2 q^2} + \frac{a \pi}{8 q}
		\geqslant 0,
	\]
	because $a \geqslant 12 > 2\pi \geqslant \frac{4\pi}{q}$.
	Since $F(\xi_1, \xi_2, 0) < 0$ by~\eqref{eq:orbits:sign at zero}, the intermediate value theorem gives a zero $\delta \in (\delta_-, 0)$ of $F$; every point of this interval is admissible, so the orbit $\Gamma^{q, mq}$ exists.
	It is regular, because the interval is open: $\delta > \delta_-$ gives $\theta_2 < \frac{\pi}{2}$, and $\delta < 0$ gives $\theta_1 < \frac{\pi}{2}$.

	At $\delta = 0$ the angles equal $\theta_i^0 \define \frac{\pi}{2}(1 - \xi_i)$, so $0 < \theta_2^0 - \theta_1^0 = \frac{\pi}{2} \parentheses[\big]{\xi_1 - \xi_2} \leqslant \frac{\pi}{2q}$, while the bound $\cos t \leqslant \frac{\pi}{2} - t$ on $\bigl[0, \frac{\pi}{2}\bigr]$ gives $\cos\theta^* \leqslant \cos\theta_1^0 \leqslant \frac{\pi}{2q}$ for each $\theta^* \in [\theta_1^0, \theta_2^0]$.
	The mean value theorem applied to the sine thus gives
	\[
		\abs[\big]{F(\xi_1, \xi_2, 0)}
		= \sin\theta_2^0 - \sin\theta_1^0
		\leqslant \frac{\pi^2}{4 q^2}.
	\]
	The bound~\eqref{eq:orbits:sine bound}, together with $\abs{\delta} \leqslant \frac{\pi}{2} \abs{\sin\delta}$, gives
	\begin{equation}\label{eq:orbits:apriori delta}
		\abs{\delta}
		\leqslant \frac{2}{a} \abs[\big]{F(\xi_1, \xi_2, 0)}
		\leqslant \frac{\pi^2}{2 a q^2}
		= 4 c_3\, q^{-2}.
	\end{equation}

	Since $a \geqslant 12$, the parameter $c_3 = \frac{\pi^2}{8a}$ satisfies $c_3 \leqslant \frac{\pi^2}{96} < 1$.
	Using the definition~\eqref{eq:orbits:F} of $F$ and the Taylor expansion of the cosine, we obtain
	\[
		F(\xi_1, \xi_2, 0)
		= \cos\frac{\pi}{2q} - \cos\frac{\pi}{2mq}
		= -\frac{\pi^2}{8} \bigl(1 - m^{-2}\bigr)\, q^{-2} + \mathcal{O}(q^{-4}).
	\]
	Since $H(\xi_1, \xi_2, \delta) = 0$ and $H(\xi_1, \xi_2, 0) = F(\xi_1, \xi_2, 0)$, the mean value theorem applied to the function $\delta' \mapsto H(\xi_1, \xi_2, \delta')$ on $[\delta, 0]$ yields a point $\delta^{*} \in (\delta, 0)$ such that
	\[
		\delta = -\frac{F(\xi_1, \xi_2, 0)}{\partial_\delta H(\xi_1, \xi_2, \delta^{*})}.
	\]
	In particular, \eqref{eq:orbits:apriori delta} shows that $\abs{\delta^{*}} \leqslant 4 c_3\, q^{-2}$.
	To estimate the derivative in the denominator, recall from~\eqref{eq:orbits:dH} that $\partial_\delta H = -a \cos\delta - \xi_1 \sin\varphi_1 - \xi_2 \sin\varphi_2$.
	Here the bound on $\abs{\delta^{*}}$ yields $\cos\delta^{*} = 1 + \mathcal{O}\parentheses[\big]{(\delta^{*})^{2}} = 1 + \mathcal{O}\bigl(c_3^{2}\, q^{-4}\bigr)$.
	Moreover, since $\varphi_i \leqslant \pi \xi_i$, we have $0 \leqslant \xi_1 \sin\varphi_1 + \xi_2 \sin\varphi_2 \leqslant \pi \bigl( \xi_1^{2} + \xi_2^{2} \bigr) \leqslant 2 \pi\, q^{-2}$.
	Dividing both terms by $a$ and using $a^{-1} = \frac{8}{\pi^2}\, c_3$ together with $c_3 \leqslant 1$, we obtain
	\[
		\partial_\delta H(\xi_1, \xi_2, \delta^{*})
		= -a \parentheses[\big]{ 1 + \mathcal{O}(c_3\, q^{-2}) }.
	\]
	Substituting this estimate into the expression for $\delta$ and using the expansion of $F(\xi_1, \xi_2, 0)$, we deduce that
	\[
		\delta
		= a^{-1} F(\xi_1, \xi_2, 0) \parentheses[\big]{ 1 + \mathcal{O}(c_3\, q^{-2}) }
		= -c_3\, \bigl(1 - m^{-2}\bigr)\, q^{-2} + \mathcal{O}\bigl(c_3\, q^{-4}\bigr),
	\]
	where the cross term contributes $\mathcal{O}(c_3^2\, q^{-4}) = \mathcal{O}(c_3\, q^{-4})$ because $c_3 < 1$.
	In view of~\eqref{eq:orbits:angle corrections}, dividing this expansion by $q$ and by $-mq$ yields the desired formulas for $\varepsilon_1$ and $\varepsilon_2$.
	The bound on $\abs{\varepsilon_1} + \abs{\varepsilon_2}$ then follows immediately, because $\bigl(1 - m^{-2}\bigr) \bigl(1 + m^{-1}\bigr) < \frac{3}{2}$ and $q^{-5} \leqslant \frac{1}{4}\, q^{-3}$.
\end{proof}

%% file: section/Rigidity_Standard.tex
\section{The linearized isospectral operator and its kernel}
\label{sec:linearized isospectral operator}

\subsection{Function space and parametrization}
\label{sub:Function spaces for normal variation}

Recall from Subsection~\ref{sub:The construction of hybrid periodic orbit} that the disjoint union of the two circular arcs of the circular Bunimovich stadium $\Omega_{*}$ is identified with
\[
	S=S_1\cup S_2 \coloneqq
	\left\{s_1 \describe \ -\frac{\pi}{2}\leqslant s_1\leqslant \frac{\pi}{2} \right\}\cup 
	\left\{s_2 \describe \ -\frac{\pi}{2}\leqslant s_2\leqslant \frac{\pi}{2} \right\},
	\]
	so that $S_1 \cong S_2 \cong [-\tfrac{\pi}{2}, \tfrac{\pi}{2}]$.
	Throughout, for functions $\mathbf{u}_i$ on $S_i$ ($i \in \set{1, \, 2}$), we write $\mathbf{u} = \mathbf{u}_1 \cup \mathbf{u}_2$ for the function on $S$ whose restriction to $S_i$ equals $\mathbf{u}_i$; equivalently, we identify $\mathbf{u}$ with the pair $(\mathbf{u}_1, \, \mathbf{u}_2)$.

	We define
	\[
		L^1_{*}(S) \define \set[\bigg]{ \mathbf{u} = \mathbf{u}_1 \cup \mathbf{u}_2 \describe \mathbf{u}_i \in L^1(S_i), \, \mathbf{u}_i(-s_i) = \mathbf{u}_{i}(s_i), \, \int \mathbf{u}_1\, \mathrm{d}s_1 + \int \mathbf{u}_2\, \mathrm{d}s_2 = 0 }.
	\]
	Equivalently,
	\[
		L^1_{*}(S) = \set[\bigg]{ \mathbf{u} = \mathbf{u}_1 \cup \mathbf{u}_2 \describe \mathbf{u}_i \in L^1(S_i) \text{ even for } i \in \set{1, \, 2}, \, \widehat{u}_{1, 0} + \widehat{u}_{2, 0} = 0 },
	\]
	where $\widehat{u}_{i, j}$ denotes the $j$-th Fourier coefficient of $\mathbf{u}_i$, so that $\mathbf{u}_i$ has the Fourier series $\sum_{j = 0}^{+\infty} \widehat{u}_{i, j}\, \mathbf{e}_{i, j}$.
Here the Fourier basis elements $\mathbf{e}_{i, j}$ are defined by
\begin{equation} \label{eq:def of Fourier basis e_ij}
	\mathbf{e}_{i, j}(s_i) \define \cos(2j s_i) \qquad \text{for each } i \in \set{1, \, 2} \text{ and each } j \in \n_0,
\end{equation}
and the Fourier coefficients of $\mathbf{u}_i$ by
\[
	\widehat{u}_{i, 0} \define \frac{1}{\pi} \int_{S_i} \mathbf{u}_i\, \mathrm{d}s_i
	\qquad \text{and} \qquad
	\widehat{u}_{i, j} \define \frac{2}{\pi} \int_{S_i} \mathbf{u}_i\, \mathbf{e}_{i, j}\, \mathrm{d}s_i \quad \text{for } j \in \n.
\]
The factor $2$ in the Fourier basis $\cos(2j s_i)$ reflects that both $S_1$ and $S_2$ have length $\pi$.

\subsection{The class \texorpdfstring{$\mathcal{M}^{r}$}{M\string^r} and the deformation function}
\label{sub:class Mr and deformation function}

For $a > 0$, let $\Omega_{*} \subset \real^{2}$ be the circular Bunimovich stadium, the convex domain bounded by the two horizontal flat edges of length $a$
\[
	E_{\mathrm{top}} \define [-a/2, \, a/2] \times \set{+1}
	\qquad \text{and} \qquad
	E_{\mathrm{bot}} \define [-a/2, \, a/2] \times \set{-1},
\]
together with the left and right unit semicircles centred at $(\pm a/2, \, 0)$, joined to $E_{\mathrm{top}} \cup E_{\mathrm{bot}}$ with $C^{1}$ tangency at the four junction points $(\pm a/2, \, \pm 1)$.
Throughout this article, we fix $r = 4$.

\begin{definition}[The class $\mathcal{M}^{r}$]\label{def:Mr}
	A planar domain $\Omega$ belongs to $\mathcal{M}^{r}$ if
	\begin{enumerate}
		\smallskip
		\item[\textup{(M1)}] $\partial \Omega = E_{\mathrm{top}} \cup E_{\mathrm{bot}} \cup \Gamma_{1} \cup \Gamma_{2}$, where $\Gamma_{1}$ and $\Gamma_{2}$ are simple $C^{r}$ arcs of positive curvature joining, respectively, $(-a/2, +1)$ to $(-a/2, -1)$ on the left and $(+a/2, +1)$ to $(+a/2, -1)$ on the right, and the domain $\Omega$ is convex;
		\smallskip
		\item[\textup{(M2)}] at each of the four junction points, the arc $\Gamma_{i}$ meets the adjacent flat edge with $C^{1}$ tangency;
		\smallskip
		\item[\textup{(M3)}] $\Omega$ is symmetric under reflection across the horizontal axis $\set{y = 0}$.
	\end{enumerate}
\end{definition}

Both $\Omega_{*}$ and $\mathcal{M}^{r}$ depend on $a$, which we suppress from the notation.
By conditions~(M1) and~(M3), every domain in $\mathcal{M}^{r}$ shares the same junction points and symmetry axis. We fix their common position as a normalization, free to do so because the length spectrum is invariant under rigid motions.

Consider a $C^{1}$ one-parameter family $\set{\Omega_{\tau}}_{\abs{\tau} \leqslant 1}$ in $\mathcal{M}^{r}$ with $\Omega_{0} = \Omega_{*}$.
For each $\tau$, we define the deformation function $\mathbf{n}(\tau, \cdot \,) \colon S \to \real$ of this family at $\Omega_{\tau}$ by
\begin{equation}\label{eq:def:deformation function}
	\mathbf{n}(\tau, s) \define
	\begin{cases}
		\left\langle \partial_{\tau}\gamma_{1, \tau}(s_1), \, N_{1, \tau}(s_1)\right\rangle, & s = s_1 \in S_1,\\
		\left\langle \partial_{\tau}\gamma_{2, \tau}(s_2), \, N_{2, \tau}(s_2)\right\rangle, & s = s_2 \in S_2,
	\end{cases}
\end{equation}
where, for each $i \in \set{1, \, 2}$, $\gamma_{i, \tau} \colon S_i \to \partial\Omega_{\tau}$ is the constant-speed parametrization of the $i$-th arc of $\Omega_{\tau}$, oriented so that $\gamma_{i, \tau}(-\pi/2)$ and $\gamma_{i, \tau}(\pi/2)$ are the junction points on $E_{\mathrm{bot}}$ and $E_{\mathrm{top}}$, respectively, and $N_{i, \tau}$ denotes the inward unit normal along $\gamma_{i, \tau}$.
The speed equals $\ell(\Gamma_{i, \tau})/\pi$; at $\tau = 0$ it equals $1$, so $\gamma_{i, 0}$ parametrizes the unit semicircle by arclength.
That the family is $C^{1}$ means that $\tau \mapsto \gamma_{i, \tau}$ is $C^{1}$ with values in $C^{r}(S_i, \real^{2})$ for each $i \in \set{1, \, 2}$.

Where no parameter is named, the \emph{deformation function} is the one at $\Omega_{*}$, namely $\mathbf{n} = \mathbf{n}(0, \cdot \,) \colon S \to \real$, which we write as
\[
	\mathbf{n}(s) =
	\begin{cases}
		\mathbf{n}_1(s_1),  &  s = s_1 \in S_1; \\
		\mathbf{n}_2(s_2),  &  s = s_2 \in S_2,
	\end{cases}
\]
where $\mathbf{n}_1(s_1)$ and $\mathbf{n}_2(s_2)$ represent the normal variation on the left and right arcs, respectively.

For every such family, $\mathbf{n}_1$ and $\mathbf{n}_2$ satisfy Condition~(H1) below; if the family is moreover dynamically isospectral, they satisfy Condition~(H2) as well:
\begin{enumerate}
	\smallskip

	\item[(H1)]
	for each $i \in \set{1, \, 2}$, identifying $S_i \cong [-\tfrac{\pi}{2}, \tfrac{\pi}{2}]$ via arclength, the function $\mathbf{n}_i$ is a $C^r$ even function on $[-\tfrac{\pi}{2}, \tfrac{\pi}{2}]$ satisfying
	\[
		\mathbf{n}_i \parentheses[\Big]{ \pm \frac{\pi}{2} } = 0
		\qquad \text{and} \qquad
		\mathbf{n}_i' \parentheses[\Big]{ \pm \frac{\pi}{2} } = 0;
	\]

	\smallskip

	\item[(H2)]
	$\int \mathbf{n}_1(s_1)\, \mathrm{d}s_1 + \int \mathbf{n}_2(s_2)\, \mathrm{d}s_2 = 0$.
\end{enumerate}

Indeed, by~(M3), the reflection $R \colon (x, y) \mapsto (x, -y)$ maps the arc $\Gamma_{i, \tau}$ to itself and exchanges its two endpoints, so $s_i \mapsto R\, \gamma_{i, \tau}(-s_i)$ is again a constant-speed parametrization of $\Gamma_{i, \tau}$ with the same endpoints and orientation, and hence coincides with $\gamma_{i, \tau}$.
Since $R$ is a symmetry of $\Omega_{\tau}$, it carries inward normals to inward normals, so $N_{i, \tau}(-s_i) = R\, N_{i, \tau}(s_i)$; as $R$ is orthogonal, it follows that $\mathbf{n}(\tau, -s_i) = \mathbf{n}(\tau, s_i)$, and each $\mathbf{n}_i$ is even.
The junction points are fixed by~(M1), so $\partial_{\tau} \gamma_{i, \tau}(\pm \pi/2) = 0$, and $\mathbf{n}_{i}(\pm \pi/2) = 0$.
By~(M2), the tangent vector $\partial_{s_i} \gamma_{i, \tau}(\pm \pi/2)$ lies in the fixed direction of the adjacent flat edge for every $\tau$, so its $\tau$-derivative is orthogonal to $N_{i, 0}(\pm \pi/2)$; here $\partial_{\tau} \partial_{s_i} \gamma_{i, \tau} = \partial_{s_i} \partial_{\tau} \gamma_{i, \tau}$, as $\tau \mapsto \gamma_{i, \tau}$ is $C^{1}$ with values in $C^{r}$.
Hence, differentiating~\eqref{eq:def:deformation function} in $s_i$ and using $\partial_{\tau} \gamma_{i, \tau}(\pm \pi/2) = 0$, we obtain $\mathbf{n}_{i}'(\pm \pi/2) = \left\langle \partial_{s_i} \partial_{\tau} \gamma_{i, 0}(\pm \pi/2), \, N_{i, 0}(\pm \pi/2) \right\rangle = 0$.
Finally, $\mathbf{n}_{i} \in C^{r}(S_i)$, since $\partial_{\tau} \gamma_{i, \tau}\big|_{\tau = 0} \in C^{r}(S_i, \real^{2})$ and $N_{i, 0}$ is smooth.
This proves~(H1) for every $C^{1}$ family in $\mathcal{M}^{r}$.
Condition~(H2) is derived in Lemma~\ref{lem:standard stadium circular orbits} from the vanishing of the linearized lengths of the circular orbits $\Gamma^{q, q}$.

We denote by $C^r_*(S)$ the space of all functions $\mathbf{n} = \mathbf{n}_1\cup \mathbf{n}_2$ such that $\mathbf{n}_1$ and $\mathbf{n}_2$ satisfy Conditions~(H1) and (H2).
Note that $C^r_{*}(S) \subseteq L^1_{*}(S)$.

\subsection{Linear functionals associated with hybrid periodic orbits}
\label{sub:Linear functionals and linearized isospectral operators associated with hybrid periodic orbits}

For each continuous function $\mathbf{u} = \mathbf{u}_1 \cup \mathbf{u}_2$ on $S$, we define
\begin{equation}    \label{eq:linear functionals associated with hybrid periodic orbits}
	\ell^{p_1, p_2}(\mathbf{u}) \define \sum_{(s, \varphi)\in \Gamma^{p_1, p_2}} \mathbf{u}(s)\sin \varphi.
\end{equation}
Recall that the coordinates of $\Gamma^{p_1, p_2}$ are $\parentheses[\big]{ s_i^{p_1, p_2}(k), \, \varphi_i^{p_1, p_2}(k) }$ for $1 \leqslant k \leqslant p_i$, where $i = 1$ and $i = 2$ index the left and right arcs, respectively.
In this notation, \eqref{eq:linear functionals associated with hybrid periodic orbits} can be rewritten as $\ell^{p_1, p_2}(\mathbf{u}) = \ell_1^{p_1, p_2}(\mathbf{u}_1) + \ell_2^{p_1, p_2}(\mathbf{u}_2)$, where
\begin{equation}    \label{eq:linear functionals associated with hybrid periodic orbits:each side}
	\ell_i^{p_1, p_2}(\mathbf{u}_i) \define \sum_{k = 1}^{p_i} \mathbf{u}_i(s_i^{p_1, p_2}(k)) \sin (\varphi_i^{p_1, p_2}(k)).
\end{equation}
Subsection~\ref{sub:The construction of hybrid periodic orbit} normalizes $s_1$ and $s_2$ to vanish at the midpoints of the two arcs, as $\gamma_{1, 0}$ and $\gamma_{2, 0}$ do, but leaves their directions unprescribed.
The parameters $s_i^{p_1, p_2}(k)$ are symmetric about the origin, and the angles $\varphi_i^{p_1, p_2}(k)$ do not depend on $k$; reversing the direction of $s_i$ therefore leaves $\ell_i^{p_1, p_2}$ unchanged.

\begin{lemma}	\label{lem:length spectrum null}
	The length spectrum $\mathcal{L}(\Omega_{*})$ has zero Lebesgue measure.
\end{lemma}

\begin{proof}
	The boundary $\partial\Omega_{*}$ is the disjoint union of its four junction points and the four relatively open pieces cut out by them.
	Let $\kappa_1, \ldots, \kappa_4$ be real-analytic parametrizations of these open pieces, defined on open intervals, and let $\kappa_5, \ldots, \kappa_8$ be the four junction points.
	Fix $N \in \n$ together with a map $\iota \colon \z/N\z \to \set{1, \, \ldots, \, 8}$ prescribing a piece for each of $N$ successive collisions, and consider
	\[
		F_{\iota} \colon (s_k)_{\iota(k) \leqslant 4} \longmapsto \sum_{k \in \z/N\z} \abs[\big]{ \kappa_{\iota(k + 1)}(s_{k + 1}) - \kappa_{\iota(k)}(s_k) },
	\]
	where $\kappa_{\iota(k)}(s_k)$ is to be read as the point $\kappa_{\iota(k)}$ when $\iota(k) \geqslant 5$.
	Thus $F_{\iota}$ has one variable for each $k$ with $\iota(k) \leqslant 4$, and it is real-analytic on the open set where consecutive points are distinct.
	A closed billiard trajectory with $N$ collisions distributed according to $\iota$ obeys the reflection law at each collision lying on one of the four open pieces, and its length is therefore a critical value of $F_{\iota}$.
	Sard's theorem requires class $C^{n}$ of a real-valued function of $n$ variables, and $F_{\iota}$ is real-analytic, so its critical values form a null set however large $N$ is.
	If $\iota$ takes values in $\set{5, \, \ldots, \, 8}$ only, then $F_{\iota}$ has no variables at all, and the corresponding set of lengths is a single point.
	Since there are countably many pairs $\parentheses[\big]{ \juxtapose{N}{\iota} }$, the set $\mathcal{L}(\Omega_{*})$ is contained in a countable union of null sets.
\end{proof}

The convention adopted here slightly differs from that of~\cite{MR3665005}, where the length spectrum is enlarged by the integer multiples of the lengths of the closed billiard trajectories and of $\abs{\partial\Omega}$.
A countable union of dilates of a null set is null, so Lemma~\ref{lem:length spectrum null} holds for that larger set as well, and nothing below depends on the choice.

In the following, we fix a pair $\juxtapose{p_1}{p_2} \in \n$ such that the hybrid periodic orbit $\Gamma^{p_1, p_2}$ exists and is regular, and we set $N \define p_1 + p_2$.
We describe $\Gamma^{p_1, p_2}$ as a non-degenerate critical point of a length function; by the implicit function theorem, such a critical point persists under small deformations of $\Omega_{*}$, and Lemma~\ref{lem:isospectral orbit functionals} below rests on this persistence.

Let $x_1, \ldots, x_N$ be the collision points of $\Gamma^{p_1, p_2}$ in the order visited, so that $x_1, \ldots, x_{p_1}$ lie on the left arc and $x_{p_1 + 1}, \ldots, x_N$ on the right arc; all indices are read cyclically modulo $N$.
In the orientation fixed in Subsection~\ref{sub:class Mr and deformation function}, the up--down symmetry makes the left cycle run in the direction of increasing $s_1$ and the right cycle in the direction of decreasing $s_2$, so that the parameters of $x_{p_1 + 1}, \ldots, x_N$ are the numbers in~\eqref{eq: Gamma qi s varphi-} in reverse order.
Nothing below uses more about the ordering than which arc carries each $x_k$.
For each $k$, let $i(k) \in \set{1, \, 2}$ be the index of the arc carrying $x_k$, let $\varphi_k \define \varphi_{i(k)}^{p_1, p_2}$ be the collision angle at $x_k$, and let $t_k \define \abs{x_{k + 1} - x_k}$ and $\nu_k \define (x_{k + 1} - x_k) / t_k$ be the length of the $k$-th side of $\Gamma^{p_1, p_2}$ and the unit vector along it.
Since the arcs have unit radius, we have $t_k = 2 \sin\varphi_k$ whenever $x_k$ and $x_{k + 1}$ lie on the same arc.
Since $\Gamma^{p_1, p_2}$ is regular, each $x_k$ lies in the interior of an arc, and so $x_k = \gamma_{i(k), 0}(s_k^{*})$ for a unique $s_k^{*} \in (-\pi/2, \pi/2)$; we set $s^{*} \define (s_1^{*}, \ldots, s_N^{*})$.
Let $T_k \define \partial_{s} \gamma_{i(k), 0}(s_k^{*})$ and $N_k \define N_{i(k), 0}(s_k^{*})$ be the unit tangent vector and the inward unit normal at $x_k$.
By the reflection law at $x_k$, the vector $\nu_k$ is the mirror image of $\nu_{k - 1}$ in the tangent line at $x_k$, so that $\nu_k = \nu_{k - 1} - 2 \langle \nu_{k - 1}, N_k \rangle N_k$; moreover, it points into $\Omega_{*}$ and makes the angle $\varphi_k$ with this line.
Hence
\begin{equation}\label{eq:lem:isospectral:reflection}
	\big\langle \nu_k, \, N_k \big\rangle = \sin\varphi_k = - \big\langle \nu_{k - 1}, \, N_k \big\rangle
	\qquad \text{and} \qquad
	\nu_{k - 1} - \nu_k = -2 \sin\varphi_k \, N_k .
\end{equation}
Since $\sin\varphi_k > 0$, it follows that $\nu_{k - 1} \neq \nu_k$ for each $k$.

Let $U$ be the set of all $s = (s_1, \ldots, s_N) \in (-\pi/2, \pi/2)^{N}$ with $s_k \neq s_{k + 1}$ whenever $i(k) = i(k + 1)$.
Then $U$ is open and contains $s^{*}$.
For $s \in U$, the points $\gamma_{i(k), 0}(s_k)$ and $\gamma_{i(k + 1), 0}(s_{k + 1})$ are distinct for each $k$, since the two arcs are disjoint and each $\gamma_{i, 0}$ is injective; hence the length
\begin{equation}\label{eq:lem:isospectral:length function}
	L(s) \define \sum_{k = 1}^{N} \abs[\big]{ \gamma_{i(k + 1), 0}(s_{k + 1}) - \gamma_{i(k), 0}(s_{k}) }
	\qquad (s \in U)
\end{equation}
of the closed polygon with vertices $\gamma_{i(k), 0}(s_k)$ is a smooth function on $U$.
Let $\nu_k(s)$ be the unit vector from the $k$-th to the $(k + 1)$-st vertex of this polygon, so that $\nu_k(s^{*}) = \nu_k$.
Since the $k$-th vertex is the initial point of the $k$-th side and the terminal point of the $(k - 1)$-st side, differentiating~\eqref{eq:lem:isospectral:length function} gives
\begin{equation}\label{eq:lem:isospectral:first derivative}
	\frac{\partial L}{\partial s_k}(s)
	= \big\langle \partial_{s} \gamma_{i(k), 0}(s_k), \, \nu_{k - 1}(s) - \nu_k(s) \big\rangle
	\qquad (s \in U).
\end{equation}
By~\eqref{eq:lem:isospectral:reflection}, the right-hand side vanishes at $s^{*}$; thus $s^{*}$ is a critical point of $L$.

\begin{lemma}\label{lem:hybrid orbit nondegenerate}
	The critical point $s^{*}$ of the length function $L$ is non-degenerate, that is, the Hessian $\nabla^{2} L(s^{*})$ is non-singular.
\end{lemma}

\begin{proof}
	For $v \in \real^{2}$, let $v^{\perp}$ be the image of $v$ under the rotation of $\real^{2}$ through the angle $\pi/2$; then $\langle u, v^{\perp} \rangle = - \langle u^{\perp}, v \rangle$ for $u, v \in \real^{2}$.
	Let $\chi_k \in \set{-1, \, 1}$ be defined by $N_k = \chi_k T_k^{\perp}$.
	Then~\eqref{eq:lem:isospectral:reflection} gives
	\begin{equation}\label{eq:lem:isospectral:frame}
		\big\langle T_k, \, \nu_k^{\perp} \big\rangle = - \chi_k \big\langle N_k, \, \nu_k \big\rangle = - \chi_k \sin\varphi_k,
		\qquad
		\big\langle T_k, \, \nu_{k - 1}^{\perp} \big\rangle = \chi_k \sin\varphi_k .
	\end{equation}

	The function $L$ is the sum of the side lengths $h_k \define \abs{w_k}$, where
	\[
		w_k \define \gamma_{i(k + 1), 0}(s_{k + 1}) - \gamma_{i(k), 0}(s_k)
		\qquad \text{and} \qquad
		\nu_k(s) \define w_k / h_k.
	\]
	Hence $\partial_{s_k} h_k = - \langle \partial_{s} \gamma_{i(k), 0}(s_k), \, \nu_k(s) \rangle$ and $\partial_{s_{k + 1}} h_k = \langle \partial_{s} \gamma_{i(k + 1), 0}(s_{k + 1}), \, \nu_k(s) \rangle$.
	The derivative $\partial \nu_k(s)$, being orthogonal to $\nu_k(s)$, equals $h_k^{-1} \langle \partial w_k, \, \nu_k(s)^{\perp} \rangle \nu_k(s)^{\perp}$ for each variable.
	At $s^{*}$ we have $h_k = t_k$ and $\partial_{s}^{2} \gamma_{i(k), 0}(s_k^{*}) = N_k$, the arcs having unit curvature.
	Differentiating once more and using~\eqref{eq:lem:isospectral:reflection} and~\eqref{eq:lem:isospectral:frame} at the $k$-th and at the $(k + 1)$-st vertex, we obtain at $s^{*}$
	\[
		\begin{gathered}
			\partial_{s_k}^{2} h_k = \frac{\sin^{2}\varphi_k}{t_k} - \sin\varphi_k,
			\qquad
			\partial_{s_{k + 1}}^{2} h_k = \frac{\sin^{2}\varphi_{k + 1}}{t_k} - \sin\varphi_{k + 1}, \\
			\partial_{s_k} \partial_{s_{k + 1}} h_k = \chi_k \chi_{k + 1}\, \frac{\sin\varphi_k \sin\varphi_{k + 1}}{t_k} .
		\end{gathered}
	\]
	For $y \in \real^{N}$ set $Y_k \define \chi_k \sin\varphi_k \, y_k$.
	Since $\chi_k^{2} = 1$, the three formulas combine to
	\[
		\nabla^{2} h_k(s^{*})[y, \, y]
		= \frac{(Y_k + Y_{k + 1})^{2}}{t_k} - \frac{Y_k^{2}}{\sin\varphi_k} - \frac{Y_{k + 1}^{2}}{\sin\varphi_{k + 1}},
	\]
	and summing over $k$, each vertex being an endpoint of two sides, we obtain
	\begin{equation}\label{eq:lem:isospectral:hessian}
		\nabla^{2} L(s^{*})[y, \, y]
		= \sum_{k = 1}^{N} \frac{(Y_k + Y_{k + 1})^{2}}{t_k} - \sum_{k = 1}^{N} \frac{2 Y_k^{2}}{\sin\varphi_k} .
	\end{equation}

	Since $y \mapsto Y$ is a linear isomorphism of $\real^{N}$, the Hessian $\nabla^{2} L(s^{*})$ is singular if and only if the quadratic form in $Y$ on the right-hand side of~\eqref{eq:lem:isospectral:hessian} is degenerate, that is, if and only if its partial derivatives with respect to $Y_1, \ldots, Y_N$ have a common zero $Y \neq 0$.
	Computing these derivatives, we see that this is the case if and only if the linear system
	\begin{equation}\label{eq:lem:isospectral:jacobi}
		\frac{Y_{k - 1} + Y_k}{t_{k - 1}} + \frac{Y_k + Y_{k + 1}}{t_k} = \frac{2 Y_k}{\sin\varphi_k}
		\qquad (1 \leqslant k \leqslant N)
	\end{equation}
	has a non-trivial solution.
	Extend $t_k$ and $\varphi_k$ to all $k \in \z$ by $N$-periodicity, and consider~\eqref{eq:lem:isospectral:jacobi} as a recursion for sequences $(Y_k)_{k \in \z}$, imposed for all $k \in \z$; the solutions of the original system are then precisely the $N$-periodic solutions of this recursion.
	For a sequence $(Y_k)_{k \in \z}$ set $\eta_k \define (-1)^{k} Y_k$ and $\eta_k' \define (-1)^{k + 1} (Y_k + Y_{k + 1}) / t_k$.
	Then $\eta_{k + 1} = \eta_k + t_k \eta_k'$ by the definition of $\eta_k'$, and~\eqref{eq:lem:isospectral:jacobi} with $k + 1$ in place of $k$ is equivalent to $\eta_{k + 1}' = \eta_k' - 2 \eta_{k + 1} / \sin\varphi_{k + 1}$.
	Hence $(Y_k)_{k \in \z}$ solves the recursion if and only if
	\[
		\begin{pmatrix} \eta_{k + 1} \\ \eta_{k + 1}' \end{pmatrix}
		= B_k \begin{pmatrix} \eta_{k} \\ \eta_{k}' \end{pmatrix}
		\quad \text{for all } k \in \z,
		\qquad \text{where} \quad
		B_k \define
		\begin{pmatrix} 1 & 0 \\ -2/\sin\varphi_{k + 1} & 1 \end{pmatrix}
		\begin{pmatrix} 1 & t_k \\ 0 & 1 \end{pmatrix} .
	\]
	Since each $B_k$ is invertible, the map $(Y_k)_{k \in \z} \mapsto (\eta_1, \eta_1')$ is a linear isomorphism from the space of solutions of the recursion onto $\real^{2}$.
	A solution is $N$-periodic if and only if $(\eta_{k + N}, \eta_{k + N}') = (-1)^{N} (\eta_k, \eta_k')$ for all $k \in \z$, and since $B_{k + N} = B_k$, it suffices to require this for $k = 1$.
	As $(\eta_{N + 1}, \eta_{N + 1}')$ is the image of $(\eta_1, \eta_1')$ under the monodromy matrix $M \define B_N B_{N - 1} \cdots B_1$, we conclude that $\nabla^{2} L(s^{*})$ is singular if and only if $(-1)^{N}$ is an eigenvalue of $M$.
	Each factor of $M$ has determinant one, and so $\det M = 1$.
	If $\abs{\operatorname{tr} M} > 2$, then the eigenvalues of $M$ are real and of modulus different from one, and $(-1)^{N}$ is not among them.
	It therefore suffices to show that $\abs{\operatorname{tr} M} > 2$.

	Set $\sigma_i \define \sin \varphi_{i}^{p_1, p_2}$ and $m_i \define p_i - 1$ for $i \in \set{1, \, 2}$, and let $d$ be the common length of the two crossing chords of $\Gamma^{p_1, p_2}$, which are mirror images of each other across the symmetry axis.
	If $1 \leqslant k \leqslant p_1 - 1$, then $x_k$ and $x_{k + 1}$ lie on the left arc, so that $t_k = 2 \sigma_1$ and $\sin\varphi_{k + 1} = \sigma_1$; if $p_1 + 1 \leqslant k \leqslant N - 1$, then both lie on the right arc, so that $t_k = 2 \sigma_2$ and $\sin\varphi_{k + 1} = \sigma_2$.
	The remaining two sides, the $p_1$-th and the $N$-th, are the crossing chords, so that $t_{p_1} = t_N = d$, $\sin\varphi_{p_1 + 1} = \sigma_2$, and $\sin\varphi_{N + 1} = \sigma_1$.
	Grouping the factors of $M$ accordingly, we obtain
	\[
		M = X_1 \Pi_2^{\, m_2} X_2 \Pi_1^{\, m_1},
		\qquad
		X_i \define
		\begin{pmatrix} 1 & 0 \\ -2/\sigma_i & 1 \end{pmatrix}
		\begin{pmatrix} 1 & d \\ 0 & 1 \end{pmatrix},
		\qquad
		\Pi_i \define
		\begin{pmatrix} 1 & 0 \\ -2/\sigma_i & 1 \end{pmatrix}
		\begin{pmatrix} 1 & 2\sigma_i \\ 0 & 1 \end{pmatrix} .
	\]
	Set $u_i \define (1, \, -1/\sigma_i)^{\mathsf{T}}$ and $v_i \define (1, \, \sigma_i)^{\mathsf{T}}$.
	Then $\Pi_i = -I + 2 u_i v_i^{\mathsf{T}}$ and $v_i^{\mathsf{T}} u_i = 0$, so that $\Pi_i^{\, m} = (-1)^{m} \parentheses[\big]{ I - 2 m\, u_i v_i^{\mathsf{T}} }$ for each $m \in \n_0$.
	Since $(-1)^{m_1 + m_2} = (-1)^{N}$, expanding the product gives
	\[
		(-1)^{N} \operatorname{tr} M
		= \operatorname{tr}(X_1 X_2)
		- 2 m_1\, v_1^{\mathsf{T}} X_1 X_2 u_1
		- 2 m_2\, v_2^{\mathsf{T}} X_2 X_1 u_2
		+ 4 m_1 m_2 \parentheses[\big]{ v_2^{\mathsf{T}} X_2 u_1 } \parentheses[\big]{ v_1^{\mathsf{T}} X_1 u_2 } .
	\]
	A direct computation gives
	\[
		\operatorname{tr}(X_1 X_2) = 2 + \frac{4 d\, (d - \sigma_1 - \sigma_2)}{\sigma_1 \sigma_2}
	\]
	and, for $i \in \set{1, \, 2}$,
	\[
		v_i^{\mathsf{T}} X_i X_{3 - i} u_i = - \frac{2 (d - \sigma_i)(d - \sigma_1 - \sigma_2)}{\sigma_1 \sigma_2},
		\qquad
		v_i^{\mathsf{T}} X_i u_{3 - i} = \frac{d - \sigma_1 - \sigma_2}{\sigma_{3 - i}} .
	\]
	Hence
	\begin{equation}\label{eq:lem:isospectral:trace}
		(-1)^{N} \operatorname{tr} M
		= 2 + \frac{4\, (d - \sigma_1 - \sigma_2)}{\sigma_1 \sigma_2}
		\parentheses[\big]{ d + m_1 (d - \sigma_1) + m_2 (d - \sigma_2) + m_1 m_2 (d - \sigma_1 - \sigma_2) } .
	\end{equation}

	We claim that $d - \sigma_1 - \sigma_2 > 0$.
	Let $O_1$ and $O_2$ be the centers of the left and right arcs, so that $O_2 - O_1 = (a, \, 0)$.
	Since the arcs have unit radius, we have $O_1 = x_{p_1} + N_{p_1}$ and $O_2 = x_{p_1 + 1} + N_{p_1 + 1}$; as $x_{p_1 + 1} - x_{p_1} = d\, \nu_{p_1}$, it follows that
	\[
		O_2 - O_1 = d\, \nu_{p_1} + N_{p_1 + 1} - N_{p_1} .
	\]
	Taking the inner product with $\nu_{p_1}$ and using~\eqref{eq:lem:isospectral:reflection} at $x_{p_1}$ and at $x_{p_1 + 1}$, we obtain
	\begin{equation}\label{eq:lem:isospectral:defocusing}
		d - \sigma_1 - \sigma_2 = \big\langle O_2 - O_1, \, \nu_{p_1} \big\rangle ,
	\end{equation}
	which is $a$ times the first coordinate of $\nu_{p_1}$.
	The first coordinate of $x_{p_1}$ is at most $-a/2$ and that of $x_{p_1 + 1}$ is at least $a/2$, so the first coordinate of $d\, \nu_{p_1} = x_{p_1 + 1} - x_{p_1}$ is at least $a$, and
	\[
		d - \sigma_1 - \sigma_2 \geqslant \frac{a^{2}}{d} > 0 .
	\]
	This proves the claim, and it follows that $d - \sigma_i = \sigma_{3 - i} + (d - \sigma_1 - \sigma_2) > 0$ for $i \in \set{1, \, 2}$.
	Every term inside the parentheses in~\eqref{eq:lem:isospectral:trace} is therefore non-negative, the first being positive.
	Dropping the three terms that carry a factor $m_1$ or $m_2$ and using $\sigma_1 \sigma_2 \leqslant 1$, we obtain
	\[
		(-1)^{N} \operatorname{tr} M
		\geqslant 2 + \frac{4 d\, (d - \sigma_1 - \sigma_2)}{\sigma_1 \sigma_2}
		\geqslant 2 + 4 a^{2}
		> 2 .
	\]
	Hence $\abs{\operatorname{tr} M} > 2$, and the lemma follows.
\end{proof}

\begin{remark}\label{rem:period two orbit}
	The case $p_1 = p_2 = 1$ is included, and is used below.
	There $N = 2$, both collision angles equal $\frac{\pi}{2}$, and the two crossing chords are the same segment of the symmetry axis traversed in opposite directions, so that $\Gamma^{1, 1}$ is the maximal period-two orbit, of length $2 (a + 2)$.
	With $\sigma_1 = \sigma_2 = 1$, $m_1 = m_2 = 0$ and $d = a + 2$, the trace formula~\eqref{eq:lem:isospectral:trace} reads $(-1)^{N} \operatorname{tr} M = 2 + 4 a (a + 2)$.
	This orbit cannot be dispensed with: among the circular orbits $\Gamma^{q, q}$ it is the only one whose linear functional does not annihilate the Fourier mode $\mathbf{e}_{i, 1}$.
\end{remark}

Dynamical isospectrality enters the argument only through the following lemma, which holds for every $a > 0$, independently of the condition $a \geqslant 12$ in Subsection~\ref{sub:existence for long flat edges}.

\begin{lemma}	\label{lem:isospectral orbit functionals}
	Let $\set{ \Omega_{\tau} }_{\abs{\tau} \leqslant 1}$ be a dynamically isospectral $C^{1}$ family in $\mathcal{M}^{r}$ with $\Omega_{0} = \Omega_{*}$ and deformation function $\mathbf{n}$.
	Then $\ell^{p_1, p_2}(\mathbf{n}) = 0$ for each pair $\juxtapose{p_1}{p_2} \in \n$ such that the hybrid periodic orbit $\Gamma^{p_1, p_2}$ exists and is regular.
\end{lemma}

\begin{proof}
	Let $\juxtapose{p_1}{p_2}$ be as in the statement.
	We use the notation introduced before Lemma~\ref{lem:hybrid orbit nondegenerate}.
	For $\abs{\tau} \leqslant 1$ and $s \in U$, let
	\[
		L_{\tau}(s) \define \sum_{k = 1}^{N} \abs[\big]{ \gamma_{i(k + 1), \tau}(s_{k + 1}) - \gamma_{i(k), \tau}(s_{k}) }
	\]
	be the length of the closed polygon with vertices $\gamma_{i(k), \tau}(s_k)$, consecutive vertices being distinct for the reason given for $\tau = 0$, and let $\nu_k(\tau, s)$ be the unit vector from the $k$-th to the $(k + 1)$-st vertex of this polygon.
	Then $L_0 = L$ and $\nu_k(0, s) = \nu_k(s)$, and as in~\eqref{eq:lem:isospectral:first derivative},
	\[
		\frac{\partial L_{\tau}}{\partial s_k}(s)
		= \big\langle \partial_{s} \gamma_{i(k), \tau}(s_k), \, \nu_{k - 1}(\tau, s) - \nu_k(\tau, s) \big\rangle .
	\]
	We claim that if $s \in U$ is a critical point of $L_{\tau}$ with $\nu_{k - 1}(\tau, s) \neq \nu_k(\tau, s)$ for each $k$, then the polygon determined by $s$ is a closed billiard trajectory in $\Omega_{\tau}$.
	To see this, note first that at each vertex the difference $\nu_{k - 1}(\tau, s) - \nu_k(\tau, s)$ is normal to $\partial\Omega_{\tau}$, by the formula for $\partial L_{\tau} / \partial s_k$, and non-zero, by assumption; as $\nu_{k - 1}(\tau, s)$ and $\nu_k(\tau, s)$ are unit vectors, this means that $\nu_k(\tau, s)$ is the mirror image of $\nu_{k - 1}(\tau, s)$ in the tangent line, which is the reflection law.
	Moreover, the sides of the polygon lie in the closure of $\Omega_{\tau}$, as $\Omega_{\tau}$ is convex.
	If the interior of the $k$-th side met $\partial\Omega_{\tau}$, then convexity would place the whole side in $\partial\Omega_{\tau}$, so that $\nu_k(\tau, s)$ would be tangent to $\partial\Omega_{\tau}$ at the $k$-th vertex; as $\nu_{k - 1}(\tau, s) - \nu_k(\tau, s)$ is normal there and both vectors have unit length, this would force $\nu_{k - 1}(\tau, s) = \nu_k(\tau, s)$.
	Hence the sides meet $\partial\Omega_{\tau}$ only at the vertices, and the claim follows.

	By Lemma~\ref{lem:hybrid orbit nondegenerate}, the point $s^{*}$ is a non-degenerate critical point of $L_0$.
	Since $\tau \mapsto \gamma_{i, \tau}$ is $C^{1}$ with values in $C^{r}(S_i, \real^{2})$ and $r \geqslant 2$, the map $(\tau, s) \mapsto \nabla_{s} L_{\tau}(s)$ is $C^{1}$.
	The implicit function theorem therefore gives a number $\tau_0 \in (0, 1]$, depending on $\juxtapose{p_1}{p_2}$, and a $C^{1}$ map $\tau \mapsto s(\tau)$ from $(-\tau_0, \tau_0)$ into $U$ with $s(0) = s^{*}$ and $\nabla_{s} L_{\tau}(s(\tau)) = 0$.
	We have $\nu_{k - 1} \neq \nu_k$ for each $k$ by~\eqref{eq:lem:isospectral:reflection}, and $\nu_k(\tau, s)$ depends continuously on $(\tau, s)$.
	Decreasing $\tau_0$ if necessary, we may therefore assume that $\nu_{k - 1}(\tau, s(\tau)) \neq \nu_k(\tau, s(\tau))$ for each $k$ and each $\abs{\tau} < \tau_0$.
	By the claim, the polygon determined by $s(\tau)$ is a closed billiard trajectory in $\Omega_{\tau}$, and so its length $\Lambda(\tau) \define L_{\tau}(s(\tau))$ belongs to $\mathcal{L}(\Omega_{\tau})$.

	The function $\Lambda$ is $C^{1}$ on $(-\tau_0, \tau_0)$, and $\mathcal{L}(\Omega_{\tau}) = \mathcal{L}(\Omega_{*})$ by dynamical isospectrality.
	Thus the image of $\Lambda$ is an interval contained in $\mathcal{L}(\Omega_{*})$, a set of measure zero by Lemma~\ref{lem:length spectrum null}; hence $\Lambda$ is constant, and $\Lambda'(0) = 0$.
	On the other hand, since $\nabla_{s} L_{0}(s^{*}) = 0$, the chain rule gives $\Lambda'(0) = \partial_{\tau} L_{\tau}(s^{*}) \big|_{\tau = 0}$.
	Differentiating the definition of $L_{\tau}$ in $\tau$ and grouping the terms by vertices, we obtain
	\[
		\Lambda'(0)
		= \sum_{k = 1}^{N} \big\langle \partial_{\tau} \gamma_{i(k), \tau}(s_k^{*}) \big|_{\tau = 0}, \, \nu_{k - 1} - \nu_k \big\rangle
		= -2 \sum_{k = 1}^{N} \mathbf{n}(s_k^{*}) \sin\varphi_k
		= -2\, \ell^{p_1, p_2}(\mathbf{n}) ,
	\]
	where the second equality follows from~\eqref{eq:lem:isospectral:reflection} and the definition~\eqref{eq:def:deformation function} of the deformation function.
	Hence $\ell^{p_1, p_2}(\mathbf{n}) = 0$.
\end{proof}

Regularity is needed because $\partial\Omega_{*}$ is only $C^{1}$ at the junction points; at a configuration with a collision there, the function $L$ is not twice differentiable, and the Hessian~\eqref{eq:lem:isospectral:hessian} is unavailable.

Every hybrid periodic orbit used below is regular: the circular orbits $\Gamma^{q, q}$ and the orbits $\Gamma^{q, mq}$ with $q$ large by Lemma~\ref{lem:Taylor expansion of collision angle}, and, once $a \geqslant 12$, the orbits $\Gamma^{1, m}$ and $\Gamma^{q, mq}$ with $q \geqslant 2$ by Lemmas~\ref{lem:short orbit control} and~\ref{lem:tail orbit asymptotics}.

For $\mathbf{n} = \mathbf{n}_1 \cup \mathbf{n}_2$, we write the Fourier expansion $\mathbf{n}_i = \sum_{j = 0}^{+\infty} \widehat{n}_{i, j}\, \mathbf{e}_{i, j}$ for each $i \in \set{1, \, 2}$, with $\mathbf{e}_{i, j}$ the Fourier basis defined in~\eqref{eq:def of Fourier basis e_ij}.
Then
\begin{equation}    \label{eq:isospectral condition}
	\sum_{j = 0}^{+\infty} \widehat{n}_{1, j} \ell_1^{p_1, p_2}(\mathbf{e}_{1, j}) + \sum_{j = 0}^{+\infty} \widehat{n}_{2, j} \ell_2^{p_1, p_2}(\mathbf{e}_{2, j}) = 0.
\end{equation}
It follows from \eqref{eq: Gamma qi s varphi}, \eqref{eq: Gamma qi s varphi-}, and \eqref{eq:linear functionals associated with hybrid periodic orbits:each side} that, for each $i \in \set{1, \, 2}$ and each $j \in \n_0$,
\begin{equation} \label{eq:expression of linear functional}
	\begin{split}
		\ell_i^{p_1, p_2}(\mathbf{e}_{i, j})
		&= \sin\varphi_i^{p_1, p_2} \sum_{k = 1}^{p_i} \cos{\parentheses[\big]{ 2 j s_i^{p_1, p_2}(k) } } \\
		&= \sin\varphi_i^{p_1, p_2} \sum_{k = 1}^{p_i} \cos{\parentheses[\big]{ ( 2k-p_i-1)  2j \varphi_i^{p_1, p_2} } } \\
		&= \sin \varphi_i^{p_1, p_2}  \frac{\sin{\parentheses[\big]{ p_i \parentheses[\big]{ 2j \varphi_i^{p_1, p_2} } } }}{\sin{\parentheses[\big]{ 2 j \varphi_i^{p_1, p_2} } }}.
	\end{split}
\end{equation}
Recall from Lemma~\ref{lem:Taylor expansion of collision angle} that $\varphi_i^{p_1, p_2} = \frac{\pi}{2p_i} + \varepsilon_i$, where $\varepsilon_i = \mathcal{O}(p_i^{-3})$ for the proportional families $\Gamma^{q, mq}$ ($m$ fixed) considered in this article.
We are thus led to consider the function
\begin{equation}    \label{eq:def:Phi}
	\begin{split}
		\Phi^{j, p}(\varepsilon) &\define \sin \parentheses[\Big]{ \frac{\pi}{2p} +\varepsilon } \frac{\sin \parentheses[\big]{ 2pj \parentheses[\big]{ \frac{\pi}{2p} +\varepsilon } } }{\sin \parentheses[\big]{ 2j \parentheses[\big]{ \frac{\pi}{2p} +\varepsilon } } } \\
		&= \sin \parentheses[\Big]{ \frac{\pi}{2p} +\varepsilon } \frac{\sin \parentheses[\big]{ p \parentheses[\big]{ \frac{j\pi}{p} + 2j \varepsilon } } }{\sin \parentheses[\big]{ \frac{j\pi}{p} + 2j \varepsilon } }.
	\end{split}
\end{equation}

\begin{lemma} \label{lem:Taylor expansion of auxiliary variable}
	If $\varepsilon = \mathcal{O}(p^{-3})$ as $p \to +\infty$, then the Taylor expansion of $\Phi^{j, p}(\varepsilon)$ at $\varepsilon = 0$ is
	\begin{equation}    \label{eq:lem:Taylor expansion of auxiliary variable}
		\Phi^{j, p}(\varepsilon) = \sum_{k = 0}^{1} D^{j, p}_k\, \varepsilon^k + \mathcal{O}(j^2\, p^{-4}) \quad \text{as } p \to +\infty, \text{ uniformly in } j \in \n,
	\end{equation}
	where $D_k^{j, p}$ is given by~\eqref{eq:lem:Taylor expansion of auxiliary variable:coefficients:p mid j} and~\eqref{eq:lem:Taylor expansion of auxiliary variable:coefficients:p not mid j}.
	In particular, we have $\abs[\big]{ D_k^{j, p} } = \mathcal{O}(j^k\, p^k)$ as $p \to +\infty$, uniformly in $j \in \n$, for each $k \in \set{0, \, 1}$.
\end{lemma}
\begin{proof}
	By Taylor expansion,
	\begin{equation}    \label{eq:temp:lem:Taylor expansion of auxiliary variable:taylor expansion of sin}
		\sin\parentheses[\Big]{\frac{\pi}{2p} + \varepsilon}
		= \sin\frac{\pi}{2p} + \cos\parentheses[\Big]{\frac{\pi}{2p}} \cdot \varepsilon + \mathcal{O}(\varepsilon^2)
	\end{equation}
	as $\varepsilon \to 0$.
	Define $f(x) \define f_p(x) \define \frac{\sin(px)}{\sin x}$, extended by removable continuation at the zeros of the denominator, so that
	\[
		f(x) = \sum_{k = 1}^{p} \cos\bigl((2k - p - 1)\, x\bigr)
	\]
	for all $x \in \real$.
	Taylor expansion of $f$ near $x = j\pi/p$ gives
	\begin{equation*}
		\frac{\sin\bigl(p(\frac{j\pi}{p} + 2j\varepsilon)\bigr)}{\sin\bigl(\frac{j\pi}{p} + 2j\varepsilon\bigr)}
		= f\parentheses[\Big]{\frac{j\pi}{p} + 2j\varepsilon}
		= f\parentheses[\Big]{\frac{j\pi}{p}} + f'\parentheses[\Big]{\frac{j\pi}{p}} \cdot 2j\, \varepsilon + \frac{1}{2}\, f''\parentheses[\Big]{\frac{j\pi}{p} + \vartheta \cdot 2j\varepsilon} \cdot 4 j^2\, \varepsilon^2,
	\end{equation*}
	for some $\vartheta \in [0, 1]$, where $f'(x) = \frac{p \cos(px)}{\sin x} - \frac{\sin(px)\, \cos x}{\sin^2 x}$.

	We distinguish the cases $p \mid j$ and $p \nmid j$.

	\smallskip

	\emph{Case~1.} $p \mid j$.

	\smallskip

	In this case, we have\[
		f \parentheses[\Big]{ \frac{j \pi}{p} } = (-1)^{j + \frac{j}{p}} p, \quad 
		f' \parentheses[\Big]{ \frac{j \pi}{p} } = 0.
	\]
	Thus, we obtain
	\begin{equation}     \label{eq:lem:Taylor expansion of auxiliary variable:coefficients:p mid j}
		\begin{split}
			D_0^{j,p} &= (-1)^{j + \frac{j}{p} } \, p \sin{\frac{\pi}{2p}},\\
			D_1^{j,p} &= (-1)^{j + \frac{j}{p} } \, p \cos{\frac{\pi}{2p}}.
		\end{split}
	\end{equation}

	\smallskip

	\emph{Case~2.}  $p \nmid j$.

	\smallskip

	In this case, we have\[
		\begin{split}
			f \parentheses[\Big]{ \frac{j \pi}{p} } &= 0, \quad 
			f' \parentheses[\Big]{ \frac{j \pi}{p} } = \frac{(-1)^{j} p}{\sin \frac{j \pi}{p} }.
		\end{split}
	\]
	Thus, we obtain
	\begin{equation}     \label{eq:lem:Taylor expansion of auxiliary variable:coefficients:p not mid j}
		\begin{split}
			D_0^{j,p} &=  0,\\
			D_1^{j,p} &= (-1)^j p j \frac{2 \sin\frac{\pi}{2p} }{\sin{\frac{j\pi}{p}}}.
		\end{split}
	\end{equation}

	From the definition of $f$,
	\[
		f^{(k)}(x) =
		\begin{cases}
			\sum_{n = 1}^{p} (-1)^{k/2} (2n - p - 1)^k \cos\parentheses[\big]{ (2n - p - 1)\, x }, & 2 \mid k, \\
			\sum_{n = 1}^{p} (-1)^{(k+1)/2} (2n - p - 1)^k \sin\parentheses[\big]{ (2n - p - 1)\, x }, & 2 \nmid k,
		\end{cases}
	\]
	where $f^{(k)}(x)$ denotes the $k$-th derivative of $f$ at $x$.
	Hence
	\[
		\abs[\big]{ f^{(k)}(x) } \leqslant \sum_{n = 1}^{p} \abs{2n - p - 1}^k = \mathcal{O}\bigl(p^{k + 1}\bigr) \quad \text{as } p \to +\infty.
	\]
	By~\eqref{eq:lem:Taylor expansion of auxiliary variable:coefficients:p mid j} and~\eqref{eq:lem:Taylor expansion of auxiliary variable:coefficients:p not mid j},
	\[
		\abs[\big]{ D_k^{j,p} } = \mathcal{O}\bigl(j^k p^k\bigr) \quad \text{as } p \to +\infty, \text{ uniformly in } j \in \n,
	\]
	for each $k \in \set{0, \, 1}$.
	Combining these bounds with $\varepsilon = \mathcal{O}\bigl(p^{-3}\bigr)$ as $p \to +\infty$, and noting that the $j$-independent contribution $\mathcal{O}(\varepsilon^{2}) \cdot \abs[\big]{ f\parentheses[\big]{ \frac{j\pi}{p} + 2j\varepsilon } } = \mathcal{O}(p^{-5})$ is absorbed in $\mathcal{O}(j^{2}\, p^{-4})$ for $j \geqslant 1$, we obtain~\eqref{eq:lem:Taylor expansion of auxiliary variable}.
\end{proof}

\subsection{Linear functionals for the particular hybrid orbits}%
\label{sub:Linear functionals for the particular hybrid orbits}

We now specialize to the hybrid orbits $\Gamma^{q, mq}$, with $m \in \n$ fixed and $q \in \n$ sufficiently large.
Applying Lemma~\ref{lem:Taylor expansion of auxiliary variable}, we obtain the following expansions of the linear functionals on the Fourier basis.

\begin{lemma} \label{lem:estimate of linear functions}
	For each $m \in \n$, as $q \to +\infty$, uniformly in $j \in \n$,
	\[
		\ell_1^{q, mq}(\mathbf{e}_{1, j}) = \sum_{k = 0}^{2} \frac{a^{m, j, q}_k}{q^k} + \mathcal{O}(j^2\, q^{-4})
		\quad \text{and} \quad
		\ell_2^{q, mq}(\mathbf{e}_{2, j}) = \sum_{k = 0}^{2} \frac{b^{m, j, q}_k}{q^k} + \mathcal{O}(j^2\, q^{-4}),
	\]
	where the $a^{m, j, q}_k$ are given by~\eqref{eq:lem:estimate of linear functions:coefficients left:q mid j} and~\eqref{eq:lem:estimate of linear functions:coefficients left:q not mid j}, and the $b^{m, j, q}_k$ by~\eqref{eq:lem:estimate of linear functions:coefficients right:mq mid j} and~\eqref{eq:lem:estimate of linear functions:coefficients right:mq not mid j}.
\end{lemma}
\begin{proof}
	Fix $m \in \n$, and set $c_3 \define \pi^2/(8a)$.

	\emph{Left functional $\ell_1^{q, mq}(\mathbf{e}_{1, j})$.}
	By Corollary~\ref{cor:Taylor expansion of collision angle} and Lemma~\ref{lem:Taylor expansion of auxiliary variable}, we have
	\begin{equation}    \label{eq:temp:lem:estimate of linear functions:functional left}
		\ell_1^{q, mq} (\mathbf{e}_{1, j}) = \Phi^{j, q}(\varepsilon_1)
		= \sum_{k = 0}^{1} D^{j, q}_k \varepsilon_1^k + \mathcal{O} \parentheses[\big]{ j^{2} q^{-4} } \quad \text{as } q \to +\infty, \text{ uniformly in } j,
	\end{equation}
	where $\varepsilon_1 \define \varphi_1^{q, mq} - \frac{\pi}{2q} = -c_3\, (1 - m^{-2})\, q^{-3} + \mathcal{O}(q^{-5})$ as $q \to +\infty$, and $\abs[\big]{ D_k^{j, q} } = \mathcal{O}(j^k\, q^k)$ as $q \to +\infty$, uniformly in $j \in \n$, for $k \in \set{0, \, 1}$.
	Substituting the expansion of $\varepsilon_1$ into~\eqref{eq:temp:lem:estimate of linear functions:functional left} and collecting the powers of $q$ gives the asserted expansion of $\ell_1^{q, mq}(\mathbf{e}_{1, j})$; it remains to identify the coefficients $a^{m, j, q}_{k}$, which we do in the two cases $q \mid j$ and $q \nmid j$.

	\smallskip

	\emph{Case~1.} $q \mid j$.

	\smallskip

	Taylor-expanding the formulas in~\eqref{eq:lem:Taylor expansion of auxiliary variable:coefficients:p mid j} of Lemma~\ref{lem:Taylor expansion of auxiliary variable} gives, as $q \to +\infty$,
	\[
		\begin{split}
			D_0^{j, q} &= (-1)^{j + \frac{j}{q}} \parentheses[\Big]{ \frac{\pi}{2} - \frac{\pi^3}{48} q^{-2} } + \mathcal{O} \parentheses[\big]{ q^{-4} },   \\
			D_1^{j, q} &= (-1)^{j + \frac{j}{q}} \parentheses[\Big]{ q - \frac{\pi^2}{8} q^{-1} } + \mathcal{O} \parentheses[\big]{ q^{-3} },
		\end{split}
	\]
	and hence
	\begin{equation}    \label{eq:lem:estimate of linear functions:coefficients left:q mid j}
		\begin{split}
			a_{0}^{m, j, q} &= (-1)^{j + \frac{j}{q}} \frac{\pi}{2},\\
			a_{1}^{m, j, q} &=0,  \\
			a_{2}^{m, j, q} &= (-1)^{j + \frac{j}{q} + 1} \parentheses[\Big]{ \frac{\pi^3}{48} + c_3 \parentheses[\big]{ 1 - m^{-2} }  }.
		\end{split}
	\end{equation}

	\emph{Case~2.} $q \nmid j$.

	\smallskip

	Here~\eqref{eq:lem:Taylor expansion of auxiliary variable:coefficients:p not mid j} of Lemma~\ref{lem:Taylor expansion of auxiliary variable} gives
	\begin{equation}    \label{eq:lem:estimate of linear functions:coefficients left:q not mid j}
		\begin{split}
			a_{0}^{m, j, q} &= 0,\\
			a_{1}^{m, j, q} &= 0,  \\
			a_{2}^{m, j, q} &= (-1)^{j + 1}  2 c_3 \parentheses[\big]{ 1 - m^{-2} } j \frac{\sin\frac{\pi}{2q}}{\sin \frac{j \pi}{q}}.
		\end{split}
	\end{equation}

	\emph{Right functional $\ell_2^{q, mq}(\mathbf{e}_{2, j})$.}
	By Corollary~\ref{cor:Taylor expansion of collision angle} and Lemma~\ref{lem:Taylor expansion of auxiliary variable}, we have
	\begin{equation}    \label{eq:temp:lem:estimate of linear functions:functional right}
		\ell_2^{q, mq} (\mathbf{e}_{2, j}) = \Phi^{j, mq}(\varepsilon_2)
		= \sum_{k = 0}^{1} D^{j, mq}_k \varepsilon_2^k + \mathcal{O} \parentheses[\big]{ j^{2} q^{-4} } \quad \text{as } q \to +\infty, \text{ uniformly in } j,
	\end{equation}
	where $\varepsilon_2 \define \varphi_2^{q, mq} - \frac{\pi}{2 mq} = c_{3} m^{-1} \parentheses[\big]{ 1 - m^{-2} } q^{-3} + \mathcal{O}\parentheses[\big]{ q^{-5} }$ as $q \to +\infty$.
	As before, substituting the expansion of $\varepsilon_2$ into~\eqref{eq:temp:lem:estimate of linear functions:functional right} and collecting the powers of $q$ gives the asserted expansion of $\ell_2^{q, mq}(\mathbf{e}_{2, j})$, with coefficients determined in the two cases $m q \mid j$ and $m q \nmid j$.

	\smallskip

	\emph{Case~1.} $m q \mid j$.

	\smallskip

	Taylor-expanding the formulas in~\eqref{eq:lem:Taylor expansion of auxiliary variable:coefficients:p mid j} of Lemma~\ref{lem:Taylor expansion of auxiliary variable} gives, as $q \to +\infty$,
	\[
		\begin{split}
			D_0^{j, m q} &= (-1)^{j + \frac{j}{mq}} \parentheses[\Big]{ \frac{\pi}{2} - \frac{\pi^3}{48} m^{-2} q^{-2} } + \mathcal{O} \parentheses[\big]{ q^{-4} },   \\
			D_1^{j, m q} &= (-1)^{j + \frac{j}{mq}} \parentheses[\Big]{ m q - \frac{\pi^2}{8} m^{-1} q^{-1} } + \mathcal{O} \parentheses[\big]{ q^{-3} },
		\end{split}
	\]
	and hence
	\begin{equation}    \label{eq:lem:estimate of linear functions:coefficients right:mq mid j}
		\begin{split}
			b_{0}^{m, j, q} &= (-1)^{j + \frac{j}{mq}} \frac{\pi}{2},\\
			b_{1}^{m, j, q} &=0,  \\
			b_{2}^{m, j, q} &= (-1)^{j + \frac{j}{mq} + 1} \parentheses[\Big]{ \frac{\pi^3}{48} m^{-2} - c_3 \parentheses[\big]{ 1 - m^{-2} }  }.
		\end{split}
	\end{equation}

	\emph{Case~2.} $mq \nmid j$.

	\smallskip

	Here~\eqref{eq:lem:Taylor expansion of auxiliary variable:coefficients:p not mid j} of Lemma~\ref{lem:Taylor expansion of auxiliary variable} gives
	\begin{equation}    \label{eq:lem:estimate of linear functions:coefficients right:mq not mid j}
		\begin{split}
			b_{0}^{m, j, q} &= 0,\\
			b_{1}^{m, j, q} &= 0,  \\
			b_{2}^{m, j, q} &= (-1)^{j}  2 c_3 \parentheses[\big]{ 1 - m^{-2} } j \frac{\sin\frac{\pi}{2mq}}{\sin \frac{j \pi}{mq}}.
		\end{split}
	\end{equation}
\end{proof}

\begin{remark}\label{rem:estimate of linear functions:special cases}
	Two consequences of Lemma~\ref{lem:estimate of linear functions} are recorded below.

	\smallskip

	\emph{(i)} Fix $m \in \set{2, \, 3, \, 4}$. For all sufficiently large $q$ and all integers $s \geqslant 2$,
	\[
		\ell_1^{q, mq}(\mathbf{e}_{1, sq})
		= (-1)^{sq + s}\, \frac{\pi}{2} + \mathcal{O}\bigl(s^2\, q^{-2}\bigr),
	\]
	and
	\[
		\ell_2^{q, mq}(\mathbf{e}_{2, sq})
		=
		\begin{cases}
			(-1)^{sq + s/m}\, \frac{\pi}{2} + \mathcal{O}\bigl(s^2\, q^{-2}\bigr), & m \mid s, \\
			\mathcal{O}\bigl(s^2\, q^{-2}\bigr), & m \nmid s.
		\end{cases}
	\]
	Indeed, the cases $q \mid sq$ and $mq \mid sq$ follow from the coefficient formulas~\eqref{eq:lem:estimate of linear functions:coefficients left:q mid j} and~\eqref{eq:lem:estimate of linear functions:coefficients right:mq mid j} together with the remainder $\mathcal{O}\parentheses[\big]{ j^{2} q^{-4} }$ evaluated at $j = sq$.
	If $m \nmid s$, then $mq \nmid sq$, and by~\eqref{eq:lem:estimate of linear functions:coefficients right:mq not mid j},
	\[
		\frac{b_2^{m, sq, q}}{q^2}
		= (-1)^{sq}\, 2\, c_3\, \bigl(1 - m^{-2}\bigr)\, \frac{s}{q} \cdot \frac{\sin\frac{\pi}{2mq}}{\sin\frac{s\pi}{m}}.
	\]
	Since $m \in \set{2, \, 3, \, 4}$, the quantity $\abs{\sin\frac{s\pi}{m}}$ is uniformly bounded below whenever non-zero, so $b_2^{m, sq, q}/q^2 = \mathcal{O}(s\, q^{-2})$, which is absorbed by $\mathcal{O}(s^2\, q^{-2})$ for $s \geqslant 2$.

	\smallskip

	\emph{(ii)} Fix $j \geqslant 1$. Suppose that, for each $q$, the coefficients $y_m^{(q)}$ ($m \in \set{2, \, 3, \, 4}$) are bounded and satisfy
	\[
		\sum_{m \in \set{2, \, 3, \, 4}} y_m^{(q)}\, (1 - m^{-2}) = 0.
	\]
	Then
	\[
		\sum_{m \in \set{2, \, 3, \, 4}} y_m^{(q)}
		\bigl( \ell_1^{q, mq}(\mathbf{e}_{1, j}) - \ell_2^{q, mq}(\mathbf{e}_{2, j}) \bigr)
		= \mathcal{O}_j\bigl(q^{-4}\bigr)
		\qquad \text{as } q \to +\infty.
	\]
	Indeed, for $q > j$ we have $q \nmid j$ and $mq \nmid j$ for each $m \in \set{2, \, 3, \, 4}$, so the explicit formulas~\eqref{eq:lem:estimate of linear functions:coefficients left:q not mid j} and~\eqref{eq:lem:estimate of linear functions:coefficients right:mq not mid j} give
	\[
		\ell_1^{q, mq}(\mathbf{e}_{1, j}) - \ell_2^{q, mq}(\mathbf{e}_{2, j})
		= (-1)^{j+1}\, \frac{2\, c_3\, j}{q^2}\, (1 - m^{-2})
		\parentheses[\Bigg]{
			\frac{\sin\frac{\pi}{2q}}{\sin\frac{j\pi}{q}}
			+ \frac{\sin\frac{\pi}{2mq}}{\sin\frac{j\pi}{mq}}
		}
		+ \mathcal{O}_j\bigl(q^{-4}\bigr).
	\]
	For fixed $j$ and each $m \in \set{1, \, 2, \, 3, \, 4}$, Taylor-expanding $\sin u$ at $u = 0$ gives
	\[
		\frac{\sin\frac{\pi}{2mq}}{\sin\frac{j\pi}{mq}}
		= \frac{1}{2j} + \mathcal{O}_j\bigl(q^{-2}\bigr)
		\qquad \text{as } q \to +\infty.
	\]
	Hence the common $q^{-2}$ term is proportional to $(1 - m^{-2})$, and the assumed cancellation removes it.
\end{remark}

The following corollary records the zeroth Fourier mode asymptotics.

\begin{corollary}	\label{cor:zeroth mode difference}
	For each $m \geqslant 2$, the difference of the zeroth-mode linear functionals satisfies
	\begin{equation}\label{eq:cor:zeroth mode difference}
		\ell_1^{q,mq}(\mathbf{e}_{1,0}) - \ell_2^{q,mq}(\mathbf{e}_{2,0})
		=
		-\beta\,(1-m^{-2})\,q^{-2} + \mathcal{O}(q^{-4})
		\qquad \text{as } q \to +\infty,
	\end{equation}
	where $\beta \define \frac{\pi^{3}}{48} + 2c_3$ is independent of $m$.
\end{corollary}
\begin{proof}
	Since $\mathbf{e}_{i, 0} \equiv 1$, the trigonometric sum in~\eqref{eq:expression of linear functional} reduces to $p_i$, giving $\ell_i^{q,mq}(\mathbf{e}_{i, 0}) = p_i\sin\varphi_i^{q,mq}$.
	By Corollary~\ref{cor:Taylor expansion of collision angle}, $\varphi_i^{q,mq}=\frac{\pi}{2 p_i}+\varepsilon_i$ with $\varepsilon_1=-c_3(1-m^{-2})q^{-3}+\mathcal{O}(q^{-5})$ and $\varepsilon_2=c_3 m^{-1}(1-m^{-2})q^{-3}+\mathcal{O}(q^{-5})$.
	This gives
	\[
		\begin{split}
			p_1 \sin\varphi_1^{q,mq} &= \frac{\pi}{2} - \frac{\pi^{3}}{48\,q^{2}} - \frac{c_3(1-m^{-2})}{q^{2}} + \mathcal{O}(q^{-4}),\\
			p_2 \sin\varphi_2^{q,mq} &= \frac{\pi}{2} - \frac{\pi^{3}}{48\,m^{2}\,q^{2}} + \frac{c_3(1-m^{-2})}{q^{2}} + \mathcal{O}(q^{-4}).
		\end{split}
	\]
	In the difference, the geometric terms combine as $-\frac{\pi^{3}}{48}(1-m^{-2})q^{-2}$ and the $c_3$-corrections add, giving~\eqref{eq:cor:zeroth mode difference}.
\end{proof}

\subsection{The linearized isospectral operator \texorpdfstring{$\mathcal{J}$}{J}}
\label{sub:linearized isospectral operator J}

We define the \emph{linearized isospectral operator} $\mathcal{J}$ from $C^r_*(S)$ to $\ell^{\infty}$ by
\begin{equation}\label{eq:def linearized isospectral operator}
	\mathcal{J}(\mathbf{u})_{p_1, p_2} \define \ell^{p_1, p_2}(\mathbf{u}), \qquad (p_1, p_2) \in \n \times \n.
\end{equation}
If the orbit $\Gamma^{p_1, p_2}$ does not exist or is not regular, we set $\mathcal{J}(\mathbf{u})_{p_1, p_2} \define 0$.\footnote{
The boundary $\partial\Omega_{*}$ is $C^{1}$ at the junction points, so a closed polygon with a collision at a junction still obeys the reflection law, and its length belongs to $\mathcal{L}(\Omega_{*})$; the proof of Lemma~\ref{lem:length spectrum null} accounts for such trajectories. What fails at a junction is the second-order smoothness of the boundary, and with it the non-degeneracy argument of Lemma~\ref{lem:hybrid orbit nondegenerate}, so we discard these coordinates. The same trajectories are set aside for the same reason in~\cite[Lemma~5.3]{ChenKaloshinZhang2023}.
}
Since the $p_i$ collision points on the $i$-th arc are spaced $2 \varphi_i^{p_1, p_2}$ apart within an arc of length $\pi$ (see~\eqref{eq: Gamma qi s varphi} and~\eqref{eq: Gamma qi s varphi-}), we have $p_i \sin \varphi_i^{p_1, p_2} \leqslant \pi$; hence $\abs[\big]{ \ell^{p_1, p_2}(\mathbf{u}) } \leqslant 2 \pi\, \norm{\mathbf{u}}_{C^{0}(S)}$, and $\mathcal{J}$ takes values in $\ell^{\infty}$.

Throughout this article, we fix $\gamma \in (3, 4)$\footnote{
This range is necessary in the proof of Theorem~\ref{thm:standard stadium rigidity} to obtain convergent estimates. We set $\gamma = \tfrac{7}{2}$ wherever a numerical bound calls for an explicit value, and this value is not a free representative of the range: the bound certified in Lemma~\ref{lem:tail invertibility large a} exceeds $1$ for $\gamma \leqslant 3.13$.
}.
We define a subspace $X_{*, \gamma}$ of $L^1_{*}(S)$ as
\[
	X_{*, \gamma} \define \set[\Big]{ \mathbf{u} = \mathbf{u}_1 \cup \mathbf{u}_2 \in L^1_{*}(S) \describe \mathbf{u}_i = \sum_{j = 0}^{+\infty} \widehat{u}_{i, j} \mathbf{e}_{i, j}, \, \lim_{j \to +\infty} j^{\gamma} \abs[\big]{ \widehat{u}_{i, j} } = 0 \text{ for each } i \in \set{1, \, 2} },
\]
where $\mathbf{e}_{i, j}$ is the Fourier basis defined in~\eqref{eq:def of Fourier basis e_ij}.
We equip $X_{*, \gamma}$ with the norm $\matrixnorm{\cdot}[\gamma]$ defined by
\[
	\matrixnorm{\mathbf{u}}[\gamma] \define \max_{i \in \set{1, \, 2}} \max \parentheses[\Big]{ \abs{\widehat{u}_{i, 0}}, \, \sup_{j \in \n} j^{\gamma} \abs{ \widehat{u}_{i, j} } }.
\]
For each $q_0 \in \n$, we also define the \emph{tail sequence space}
\[
	h_{\gamma, q_0} \define \set[\Big]{ (u_j)_{j \geqslant q_0} \describe \lim_{j \to +\infty} j^{\gamma} \abs{u_j} = 0 },
\]
equipped with the norm $\matrixnorm{(u_j)}[\gamma, q_0] \define \sup_{j \geqslant q_0} j^{\gamma} \abs{u_j}$.
For a matrix $L = \parentheses{L_{q, j}}_{q, j \geqslant q_0}$ of complex numbers indexed by $q, j \geqslant q_0$, we also write
\[
	\matrixnorm{L}[\gamma, q_0] \define \sup_{q \geqslant q_0} q^{\gamma} \sum_{j \geqslant q_0} j^{-\gamma} \abs{L_{q, j}}
\]
for the associated matrix norm.
When $q_0 = 1$, we simply write $\matrixnorm{\cdot}[\gamma]$ for $\matrixnorm{\cdot}[\gamma, 1]$.
We record the invariance criterion used whenever such a matrix is inverted on $h_{\gamma, q_0}$.
A matrix $L$ whose columns satisfy $q^{\gamma} \abs{L_{q, j}} \to 0$ as $q \to +\infty$ for each fixed $j$ preserves $h_{\gamma, q_0}$ and acts on it with operator norm at most $\matrixnorm{L}[\gamma, q_0]$.
Indeed, $\abs{(L \mathbf{u})_q} \leqslant \sum_{j} \abs{L_{q, j}}\, j^{-\gamma}\, \matrixnorm{\mathbf{u}}[\gamma, q_0]$ gives the norm bound, and splitting this sum at $j = N_{\ast}$ verifies the decay condition, as the head tends to $0$ for each fixed $N_{\ast}$, while the tail is at most $\matrixnorm{L}[\gamma, q_0]\, \sup_{j > N_{\ast}} j^{\gamma} \abs{u_j}$, small for $N_{\ast}$ large.
We write $c_0$ for the Banach space of real sequences converging to zero, equipped with the supremum norm.
The weighting map $(u_j) \mapsto (j^{\gamma} u_j)$ identifies $h_{\gamma, q_0}$ isometrically with $c_0$, while $X_{*, \gamma}$ is a closed subspace of $\real^2 \oplus c_0 \oplus c_0$.
Hence both $(X_{*, \gamma}, \matrixnorm{\cdot}[\gamma])$ and $(h_{\gamma, q_0}, \matrixnorm{\cdot}[\gamma, q_0])$ are separable Banach spaces.

\begin{remark}\label{rem:operator well-defined on subspace}
	Since $\gamma \in (3, 4)$, we have $C^{4}_{*}(S) \subseteq X_{*, \gamma}$.
	Indeed, integrating
	\[
		\int_{-\pi/2}^{\pi/2} \mathbf{u}_i(s) \cos(2 j s)\, \mathrm{d}s
	\]
	by parts four times, the first and third boundary terms vanish because $\sin(j\pi) = 0$, and the second because $\mathbf{u}_i'$ is odd and $\mathbf{u}_i'(\pm \pi/2) = 0$.
	The fourth boundary term contributes $-(-1)^{j}\, \mathbf{u}_i'''(\pi/2) / (4 \pi j^{4})$ to $\widehat{u}_{i, j}$, and the remaining integral is $o(j^{-4})$ by the Riemann--Lebesgue lemma.
	Hence $\widehat{u}_{i, j} = \mathcal{O}(j^{-4}) = o(j^{-\gamma})$, the last step using $\gamma < 4$.
	Both exponents are sharp.
	No condition of the class forces $\mathbf{u}_i'''(\pm \pi/2) = 0$, so $j^{4}\, \widehat{u}_{i, j}$ need not tend to $0$ and the inclusion fails at $\gamma = 4$: for $\mathbf{u}_1 = -\mathbf{u}_2 = (\frac{\pi^{2}}{4} - s^{2})^{2} \in C^{4}_{*}(S)$ one computes $\widehat{u}_{1, j} = -3 (-1)^{j} j^{-4}$.
	In the other direction, $\mathbf{u}_i \in C^{3}$ would give only $\widehat{u}_{i, j} = o(j^{-3})$, which is not $o(j^{-\gamma})$ for $\gamma > 3$; so $r = 4$ cannot be lowered within this scheme.
	Conversely, each $\mathbf{u} \in X_{*, \gamma}$ is of class $C^{2}$, with $\mathbf{u}_i$ even and $\mathbf{u}_i'(\pm \pi/2) = 0$, since $\sum_{j} j^{2}\, \abs{\widehat{u}_{i, j}} < +\infty$ for $\gamma > 3$.

	For each $\mathbf{u} \in X_{*, \gamma}$, the decay $\widehat{u}_{i, j} = o(j^{-\gamma})$ gives $\sum_{j} \abs{\widehat{u}_{i, j}} < +\infty$, so the Fourier series $\sum_{j = 0}^{+\infty} \widehat{u}_{i, j}\, \mathbf{e}_{i, j}$ converges uniformly on $S_i$, and we identify $\mathbf{u}_i$ with its continuous sum.
	Each $\ell_i^{p_1, p_2}$ is a bounded linear functional on $C^0(S_i)$, so we may interchange it with the sum, giving
	\[
		\ell^{p_1, p_2}(\mathbf{u})
		= \sum_{i \in \set{1, \, 2}} \ell_i^{p_1, p_2}(\mathbf{u}_i)
		= \sum_{i \in \set{1, \, 2}} \sum_{j = 0}^{+\infty} \widehat{u}_{i, j}\, \ell_i^{p_1, p_2}(\mathbf{e}_{i, j}).
	\]
	Hence $\ell^{p_1, p_2}$ is well-defined on $X_{*, \gamma}$.

\end{remark}

\subsection{The weighted combination \texorpdfstring{$\lambda_{j, q}$}{lambda}}
\label{sub:weighted combination}

Theorem~\ref{thm:standard stadium finite-dimensional deformations} below and Theorem~\ref{thm:standard stadium rigidity} of the next section are read off from one and the same weighted combination of the rows of $\mathcal{J}$ indexed by $\Gamma^{q, 2q}$, $\Gamma^{q, 3q}$ and $\Gamma^{q, 4q}$.
We fix that combination here, together with the two arithmetic bounds it obeys; the two theorems then differ only in how they estimate the error committed in replacing the collision angles by their limiting values.

For $q \in \n$ and $m \in \set{2, \, 3, \, 4}$ set
\begin{equation}\label{eq:def:weights x2 x3 x4}
	x_2^{(q)} \define (-1)^{q + 1}\frac{256}{27\pi}, \qquad
	x_3^{(q)} \define (-1)^{q + 1}\frac{2}{\pi}, \qquad
	x_4^{(q)} \define -x_2^{(q)},
\end{equation}
and, whenever the orbits $\Gamma^{q, mq}$ exist and are regular,
\begin{equation}\label{eq:def:lambda j q}
	\Diffl{j}{q}{m} \define \ell_1^{q, mq}(\mathbf{e}_{1, j}) - \ell_2^{q, mq}(\mathbf{e}_{2, j}),
	\qquad
	\lambda_{j, q} \define \sum_{m \in \set{2, \, 3, \, 4}} x_m^{(q)}\, \Diffl{j}{q}{m}
	\qquad (j \in \n_0).
\end{equation}
By direct computation the weights satisfy
\begin{equation}\label{eq:weights identities}
	\sum_{m \in \set{2, \, 3, \, 4}} x_m^{(q)} = (-1)^{q+1}\frac{2}{\pi},
	\qquad
	\sum_{m \in \set{2, \, 3, \, 4}} x_m^{(q)}\bigl(1 - m^{-2}\bigr) = 0.
\end{equation}
The second identity is what the weights are chosen for: it cancels every error term proportional to $1 - m^{-2}$, and such terms carry the leading correction in the angle asymptotics of Section~\ref{sec:Hybrid Periodic Orbit on Bunimovich Stadia}.
The two identities leave a one-parameter family of admissible weights; \eqref{eq:def:weights x2 x3 x4} is the member singled out by $x_2^{(q)} = -x_4^{(q)}$, a choice discussed in Remark~\ref{rem:choice of weights}.

At the limiting angles $\pi/(2 p_i)$ the entry $\lambda_{sq, q}$ reduces, up to sign, to the value of the arithmetic function
\begin{equation}\label{eq:def:P}
	P(s) \define \frac{128}{27}\, d_2(s) + d_3(s) - \frac{128}{27}\, d_4(s),
	\qquad
	d_m(s) \define
	\begin{cases}
		(-1)^s & \text{if } m \nmid s, \\
		(-1)^s - (-1)^{s/m} & \text{if } m \mid s,
	\end{cases}
\end{equation}
and the two properties recorded next are all that either theorem uses.

\begin{lemma}\label{lem:P bounds}
	Let $P$ be as in~\eqref{eq:def:P}. Then:
	\begin{enumerate}
		\item $\abs{P(s)} \leqslant 1$ if $s$ is odd or $8 \mid s$; $\abs{P(s)} \leqslant \tfrac{155}{27}$ if $s \equiv 2 \pmod{4}$; and $\abs{P(s)} \leqslant \tfrac{256}{27}$ if $s \equiv 4 \pmod{8}$.

		\item For every $\gamma > 1$,
		\[
			\sum_{s = 2}^{+\infty} s^{-\gamma}\, \abs{P(s)}
			\leqslant
			\Bigl[(1 - 2^{-\gamma})\bigl(1 + \tfrac{155}{27}\, 2^{-\gamma} + \tfrac{256}{27}\, 2^{-2\gamma}\bigr) + 2^{-3\gamma}\Bigr] \zeta(\gamma) - 1,
		\]
		and for $\gamma = \tfrac{7}{2}$ the right-hand side is bounded by $0.6252$.
	\end{enumerate}
\end{lemma}

\begin{proof}
	(i) Since $d_3(s) = (-1)^s$ for $3 \nmid s$ and $d_3(s) = 0$ for $3 \mid s$, we have $d_3(s) \in \set{-1, \, 0}$ for $s$ odd and $d_3(s) \in \set{0, \, 1}$ for $s$ even.
	We distinguish four cases according to the $2$-adic valuation of $s$.
	If $s$ is odd, then $d_2(s) = d_4(s) = -1$, so $P(s) = d_3(s)$ and $\abs{P(s)} \leqslant 1$.
	If $s = 2t$ with $t$ odd, then $d_2(s) = 2$ and $d_4(s) = 1$, so $P(s) = \tfrac{128}{27} + d_3(s)$ and $\abs{P(s)} \leqslant \tfrac{155}{27}$.
	If $s = 4t$ with $t$ odd, then $d_2(s) = 0$ and $d_4(s) = 2$, so $P(s) = d_3(s) - \tfrac{256}{27}$ and $\abs{P(s)} \leqslant \tfrac{256}{27}$.
	If $8 \mid s$, then $d_2(s) = d_4(s) = 0$, so $P(s) = d_3(s)$ and $\abs{P(s)} \leqslant 1$.

	(ii) Grouping the terms $s \geqslant 2$ according to the four cases of~(i) and applying those bounds,
	\[
		\sum_{s = 2}^{+\infty} s^{-\gamma}\, \abs{P(s)}
		\leqslant
		\sum_{\substack{s \geqslant 3 \\ s\, \mathrm{odd}}} s^{-\gamma}
		+ \frac{155}{27} \sum_{\substack{t \geqslant 1 \\ t\, \mathrm{odd}}} (2t)^{-\gamma}
		+ \frac{256}{27} \sum_{\substack{t \geqslant 1 \\ t\, \mathrm{odd}}} (4t)^{-\gamma}
		+ \sum_{k = 3}^{+\infty} \sum_{\substack{t \geqslant 1 \\ t\, \mathrm{odd}}} (2^k t)^{-\gamma}.
	\]
	The first three sums are evaluated by $\sum_{t \geqslant 1,\, t\, \mathrm{odd}} t^{-\gamma} = (1 - 2^{-\gamma})\,\zeta(\gamma)$, the trailing $-1$ in the stated bound removing the term $s = 1$ from the first.
	In the fourth, the geometric series $\sum_{k \geqslant 3} 2^{-k\gamma} = 2^{-3\gamma}/(1 - 2^{-\gamma})$ cancels the factor $1 - 2^{-\gamma}$, leaving $2^{-3\gamma}\, \zeta(\gamma)$.
	Collecting the four contributions gives the displayed bound.
\end{proof}

\subsection{Finite-dimensionality of the kernel}
\label{sub:Analysis of linearized isospectral operators on the deformation functions}

The lemma below establishes Condition~(H2) of Subsection~\ref{sub:class Mr and deformation function}.
Indeed, for the deformation function of a dynamically isospectral family, Condition~(H1) holds by that subsection, and the hypothesis $\ell^{q, q}(\mathbf{n}) = 0$ holds by Lemma~\ref{lem:isospectral orbit functionals}.

\begin{lemma} \label{lem:standard stadium circular orbits}
	Let $\mathbf{n} = \mathbf{n}_1 \cup \mathbf{n}_2$ with $\mathbf{n}_1$ and $\mathbf{n}_2$ satisfying Condition~(H1), and suppose that $\ell^{q, q}(\mathbf{n}) = 0$ for all $q \in \n$.
	Then for all $j \in \n_{0}$,
	\[
		\widehat{\mathbf{n}}_{1, j} = -\widehat{\mathbf{n}}_{2, j},
	\]
	where $\widehat{\mathbf{n}}_{i, j}$ are the Fourier coefficients of $\mathbf{n}_i$ in the basis $\mathbf{e}_{i, j}$ for $i \in \set{1, \, 2}$.
	In particular, the function $\mathbf{n}$ satisfies Condition~(H2).
\end{lemma}
\begin{proof}
	Since each $\mathbf{n}_i$ is $C^{4}$ and even with $\mathbf{n}_i'(\pm \pi/2) = 0$, four integrations by parts give $\widehat{\mathbf{n}}_{i, j} = \mathcal{O}(j^{-4})$, as in Remark~\ref{rem:operator well-defined on subspace}; in particular $\sum_{j} \abs[\big]{ \widehat{\mathbf{n}}_{i, j} } < +\infty$, and the Fourier series of $\mathbf{n}_i$ converges uniformly to $\mathbf{n}_i$.
	For each $q \in \n$, consider the hybrid periodic orbit $\Gamma^{q, q}$; its existence, its regularity, and the angles $\varphi_1 = \varphi_2 = \pi/(2q)$ are given by Lemma~\ref{lem:Taylor expansion of collision angle}.
	Recall from~\eqref{eq:def of Fourier basis e_ij} that $\mathbf{e}_{i, j}(s_i) = \cos(2j s_i)$.
	Hence
	\[
		\ell_1^{q, q}(\mathbf{e}_{1, j}) = \ell_2^{q, q}(\mathbf{e}_{2, j})
		= (-1)^{j + j/q}\, \delta_{q \mid j}\, q \sin\frac{\pi}{2q}
		\qquad \text{for each } j \in \n_0.
	\]
	Applying the finite sum $\ell^{q, q}$ termwise to the uniformly convergent Fourier series of $\mathbf{n}$, we expand the hypothesis $\ell^{q, q}(\mathbf{n}) = 0$ as
	\begin{equation} \label{eq:temp:lem:standard stadium circular orbits:n1+n2=0}
			0 = \sum_{j = 0}^{+\infty} \widehat{\mathbf{n}}_{1, j} \ell_1^{q, q}(\mathbf{e}_{1, j}) + \sum_{j = 0}^{+\infty} \widehat{\mathbf{n}}_{2, j} \ell_2^{q, q}(\mathbf{e}_{2, j})
			= q \sin \frac{\pi}{2q} \sum_{s = 0}^{+\infty} (-1)^{sq + s} \parentheses[\big]{ \widehat{\mathbf{n}}_{1, sq} + \widehat{\mathbf{n}}_{2, sq} }
	\end{equation}
	for each $q \in \n$.
	Set $w_j \define \widehat{\mathbf{n}}_{1, j} + \widehat{\mathbf{n}}_{2, j}$ for $j \in \n_0$.
	The factor $q \sin \frac{\pi}{2q}$ is nonzero; dividing~\eqref{eq:temp:lem:standard stadium circular orbits:n1+n2=0} by it and isolating the $s = 0$ term, we obtain
	\[
		w_0 = -\sum_{s = 1}^{+\infty} (-1)^{s(q+1)}\, w_{sq} = \mathcal{O}(q^{-4}),
	\]
	since $w_{sq} = \mathcal{O}\parentheses[\big]{ (sq)^{-4} }$; as $w_0$ does not depend on $q$, letting $q \to +\infty$ gives $w_0 = 0$, which is Condition~(H2).
	Returning to a fixed $q \in \n$ and substituting $j = sq$ for $s \geqslant 1$, we obtain
	\begin{equation} \label{eq:temp:lem:standard stadium circular orbits:sum j}
		\sum_{j = 1}^{+\infty} (-1)^{j + \frac{j}{q}} \delta_{q \mid j} \parentheses[\big]{ \widehat{\mathbf{n}}_{1, j} + \widehat{\mathbf{n}}_{2, j} } = 0
		\qquad \text{for each } q \in \n.
	\end{equation}
	Let $\mathbf{w} \define (w_j)_{j=1}^{+\infty}$.
	Since $\widehat{\mathbf{n}}_{i, j} = \mathcal{O}(j^{-4}) = o(j^{-\gamma})$, the sequence $\mathbf{w}$ lies in $h_{\gamma, 1}$.
	Multiplying~\eqref{eq:temp:lem:standard stadium circular orbits:sum j} by $(-1)^{1+q}$, we rewrite it as the matrix equation
	\begin{equation} \label{eq:dn1=-dn2}
		\Delta \mathbf{w} = 0,
	\end{equation}
	where $\Delta = \parentheses[\big]{ \Delta_{qj} }_{q,j \in \n}$ is defined by $\Delta_{qj} \define (-1)^{ 1 + q + j + \frac{j}{q} }\delta_{q \mid j}$.
	The diagonal entries are $\Delta_{qq} = 1$, and the non-zero off-diagonal entries occur at $j = sq$ with $s \geqslant 2$.
	The matrix norm $\matrixnorm{\cdot}[\gamma]$ then gives
	\[
		\matrixnorm{\Delta - \operatorname{Id}}[\gamma] 
		= \sup_{q \in \n} \sum_{j \neq q} q^\gamma j^{-\gamma} \abs{\Delta_{qj}}
		= \sup_{q \in \n} \sum_{s = 2}^{+\infty} q^\gamma (sq)^{-\gamma} 
		= \sum_{s = 2}^{+\infty} s^{-\gamma} = \zeta(\gamma) - 1.
	\]
	Since $\gamma > 3$, we have $\matrixnorm{\Delta - \operatorname{Id}}[\gamma] \leqslant \zeta(\gamma) - 1 < \zeta(3) - 1 < 0.21 < 1$; moreover, each column of $\Delta - \operatorname{Id}$ is eventually zero, as $(\Delta - \operatorname{Id})_{qj} \neq 0$ forces $q \leqslant j/2$, so $\Delta$ preserves $h_{\gamma, 1}$ by the invariance criterion of Subsection~\ref{sub:linearized isospectral operator J} and is invertible on it via its Neumann series.
	It follows from~\eqref{eq:dn1=-dn2} that $\mathbf{w} = 0$, that is,
	\[
		\widehat{\mathbf{n}}_{1, j} = -\widehat{\mathbf{n}}_{2, j} \qquad \text{for each } j \in \n.
	\]
	Combined with $w_0 = 0$, this completes the proof.
\end{proof}

As a consequence, the deformation function of every dynamically isospectral family belongs to the kernel of $\mathcal{J}$.

\begin{lemma}    \label{lem:isospectral deformation in kernel}
	If a $C^{1}$ family $\set{\Omega_{\tau}}_{\abs{\tau} \leqslant 1}$ in $\mathcal{M}^{r}$ with $\Omega_{0} = \Omega_{*}$ is dynamically isospectral, then its deformation function $\mathbf{n} = \mathbf{n}_1 \cup \mathbf{n}_2$ lies in $C^{r}_{*}(S)$ and in $\ker \mathcal{J}$.
\end{lemma}
\begin{proof}
	Condition~(H1) holds for $\mathbf{n}$ by Subsection~\ref{sub:class Mr and deformation function}, and Lemma~\ref{lem:isospectral orbit functionals} gives $\ell^{p_1, p_2}(\mathbf{n}) = 0$ for every pair $\juxtapose{p_1}{p_2} \in \n$ such that $\Gamma^{p_1, p_2}$ exists and is regular.
	Since the circular orbits $\Gamma^{q, q}$ exist and are regular for every $q \in \n$, Lemma~\ref{lem:standard stadium circular orbits} shows that $\mathbf{n}$ satisfies Condition~(H2); hence $\mathbf{n} \in C^{r}_{*}(S)$ and $\mathbf{n} \in \ker \mathcal{J}$.
\end{proof}

The kernel of $\mathcal{J}$ is moreover finite-dimensional for every $a > 0$.

\begin{theorem}    \label{thm:standard stadium finite-dimensional deformations}
	For every $a > 0$, the kernel of the linearized isospectral operator $\mathcal{J}$ (defined in~\eqref{eq:def linearized isospectral operator}) on $C^r_*(S)$ is finite-dimensional.
\end{theorem}

Combined with Lemma~\ref{lem:isospectral deformation in kernel}, this shows that the deformation functions of the dynamically isospectral $C^{1}$ families in $\mathcal{M}^{r}$ with $\Omega_{0} = \Omega_{*}$ span a finite-dimensional space, as stated in Theorem~\ref{thm:intro finite dimensional}.
\begin{proof}
	Let $\mathbf{n} = \mathbf{n}_1 \cup \mathbf{n}_2 \in \ker \mathcal{J}$. Then $\ell^{q, q}(\mathbf{n}) = 0$ for all $q \in \n$, so Lemma~\ref{lem:standard stadium circular orbits} gives $\widehat{\mathbf{n}}_{1, j} = -\widehat{\mathbf{n}}_{2, j}$ for all $j \in \n_0$. Write $\widehat{\mathbf{n}}_{j} \define \widehat{\mathbf{n}}_{1, j}$.

	Take $\gamma=\tfrac{7}{2}$ and regard $\mathbf{n}$ as an element of $X_{*,\gamma}$ (see Remark~\ref{rem:operator well-defined on subspace}).
	Fix an integer $q_0 \in \n$ sufficiently large that the hybrid periodic orbit $\Gamma^{q, mq}$ exists and is regular for each $q \geqslant q_0$ and each $m \in \set{2, \, 3, \, 4}$; such a $q_0$ exists by Lemma~\ref{lem:Taylor expansion of collision angle}, for every $a > 0$, and we may enlarge it further below.
	For each $q \geqslant q_0$ and each $m \in \set{2, \, 3, \, 4}$, the identity $\ell^{q,mq}(\mathbf{n})=0$ together with the termwise expansion of $\ell^{q,mq}$ from Remark~\ref{rem:operator well-defined on subspace} and the relation $\widehat{\mathbf{n}}_{2, j}=-\widehat{\mathbf{n}}_{j}$ gives
	\begin{equation}\label{eq:thm:finite dim:basic equations}
		\sum_{j = 0}^{+\infty} \widehat{\mathbf{n}}_{j} \bigl(\ell_1^{q, mq}(\mathbf{e}_{1, j})-\ell_2^{q, mq}(\mathbf{e}_{2, j}) \bigr)=0.
	\end{equation}

	Taking the linear combination of~\eqref{eq:thm:finite dim:basic equations} for $m \in \set{2, \, 3, \, 4}$ with the weights~\eqref{eq:def:weights x2 x3 x4} gives
	\begin{equation}\label{eq:thm:finite dim:lambda system}
		\sum_{j = 0}^{+\infty} \lambda_{j,q} \, \widehat{\mathbf{n}}_{j}=0 \quad  \text{for each } q \geqslant q_0,
	\end{equation}
	with $\lambda_{j, q}$ as in~\eqref{eq:def:lambda j q}.

	The key estimate is
	\begin{equation}\label{eq:thm:finite dim:matrix norm bound}
		\matrixnorm{K}[\gamma, \, q_0]
		= \sup_{q \geqslant q_0} \sum_{\substack{j \geqslant q_0 \\ j \neq q}} q^{\gamma} j^{-\gamma} \abs{\lambda_{j, q}/\lambda_{q, q}} < 1,
	\end{equation}
	where $K$ is the off-diagonal matrix with entries $K_{q, j} = \lambda_{j, q}/\lambda_{q, q}$ for $j \neq q$ and $K_{q, q} = 0$, formally defined in~\eqref{eq:thm:finite dim:K matrix def} below.
	Once~\eqref{eq:thm:finite dim:matrix norm bound} is established, $\operatorname{Id} + K$ is invertible on $h_{\gamma, q_0}$ by Neumann series (the required invariance $K(h_{\gamma, q_0}) \subseteq h_{\gamma, q_0}$ is verified at the end of the proof), which determines all coefficients $\widehat{\mathbf{n}}_{j}$ with $j \geqslant q_0$ from the finitely many low-frequency ones $\widehat{\mathbf{n}}_{0}, \ldots, \widehat{\mathbf{n}}_{q_0 - 1}$; hence $\ker \mathcal{J}$ is finite-dimensional.

	We begin with the diagonal coefficients $\lambda_{q,q}$.
	Applying Lemma~\ref{lem:estimate of linear functions} with $j=q$, we obtain, for each $m\in\set{2, \, 3, \, 4}$,
	\[
		\ell_1^{q,mq}(\mathbf{e}_{1, q}) = (-1)^{q+1}\frac{\pi}{2}+\mathcal{O}(q^{-2}),
		\qquad
		\ell_2^{q,mq}(\mathbf{e}_{2, q}) = \mathcal{O}(q^{-2}).
	\]
	Thus, after multiplying by $x_m^{(q)}$, summing over $m\in\set{2, \, 3, \, 4}$, and using the first identity in \eqref{eq:weights identities}, we obtain
	\[
		\lambda_{q,q}
		= (-1)^{q+1}\frac{\pi}{2}\sum_{m\in\set{2, \, 3, \, 4}}x_m^{(q)}+\mathcal{O}(q^{-2})
		= 1+\mathcal{O}(q^{-2}).
	\]
	Increasing $q_0$, if necessary, we may therefore assume that
	\begin{equation}\label{eq:thm:finite dim:diag}
		\lambda_{q,q} \geqslant \frac{99}{100}
		\qquad \text{for all } q\geqslant q_0.
	\end{equation}

	Next, we estimate the off-diagonal coefficients $\lambda_{j,q}$ with $j \neq q$.
	For all $q\geqslant q_0$ we split the indices $j \geqslant q_0$ in \eqref{eq:thm:finite dim:lambda system} as
	\[
		j=sq \ \text{for some } s\geqslant 2
		\qquad\text{or}\qquad
		q\nmid j.
	\]

	\emph{Case~(a).} $q \mid j$.
	Assume $j = sq$ with $s \geqslant 2$. We claim that
	\begin{equation}\label{eq:thm:finite dim:multiples sum}
		\sup_{q\geqslant q_0}\sum_{s=2}^{+\infty}s^{-\gamma}\abs[\Big]{\frac{\lambda_{sq,q}}{\lambda_{q,q}}}
		\leqslant 0.64.
	\end{equation}
	By Remark~\ref{rem:estimate of linear functions:special cases}~(i), for each $m\in\set{2, \, 3, \, 4}$ we may write
	\[
		\ell_1^{q,mq}(\mathbf{e}_{1,sq})
		=
		(-1)^{sq+s}\frac{\pi}{2}+E_{1,m}(s,q),
	\]
	and
	\[
		\ell_2^{q,mq}(\mathbf{e}_{2,sq})
		=
		\mathbf{1}_{m\mid s}(-1)^{sq+\frac{s}{m}}\frac{\pi}{2}+E_{2,m}(s,q),
	\]
	where $\abs{E_{1,m}(s,q)}+\abs{E_{2,m}(s,q)}\leqslant C s^2 q^{-2}$ for some constant $C$ independent of $s$, $q$, and $m \in \set{2, \, 3, \, 4}$.
	Therefore
	\begin{equation}\label{eq:thm:finite dim:multiples decomposition}
		\lambda_{sq,q}
		=
		(-1)^{sq+q+1} P(s)
		+R_{s,q},
	\end{equation}
	with $P$ as in~\eqref{eq:def:P} and $\abs{R_{s,q}}\leqslant C s^2 q^{-2}$ for all $s\geqslant 2$ and $q\geqslant q_0$.
	Since $\lambda_{q,q}\geqslant \frac{99}{100}$ by~\eqref{eq:thm:finite dim:diag}, it suffices to bound $\sum_{s\geqslant 2} s^{-\gamma}\abs{\lambda_{sq,q}}$.
	By \eqref{eq:thm:finite dim:multiples decomposition},
	\begin{equation}\label{eq:thm:finite dim:remainder estimate}
		\sum_{s=2}^{+\infty}s^{-\gamma}\abs{R_{s,q}}
		\leqslant
		C q^{-2}\sum_{s=2}^{+\infty}s^{2-\gamma}
		=
		\mathcal{O}\parentheses[\big]{q^{-2}},
	\end{equation}
	since $\gamma=\tfrac{7}{2}>3$.

	By Lemma~\ref{lem:P bounds}~(ii) and the remainder estimate~\eqref{eq:thm:finite dim:remainder estimate},
	\[
		\sum_{s=2}^{+\infty}s^{-\gamma}\abs{\lambda_{sq,q}}
		\leqslant 0.6252 + \mathcal{O}\parentheses[\big]{q^{-2}}
		\qquad (q\geqslant q_0),
	\]
	so that, enlarging $q_0$ to absorb the remainder,
	\[
		\sum_{s\geqslant 2} s^{-\gamma}\abs{\lambda_{sq,q}} < 0.63
		\quad \text{for all } q\geqslant q_0.
	\]
	Combined with $\lambda_{q,q}\geqslant \frac{99}{100}$ from \eqref{eq:thm:finite dim:diag} and $0.63/0.99 < 0.64$, this proves the claim.

	\smallskip

	\emph{Case~(b).} $q \nmid j$.
	We claim that
	\begin{equation}\label{eq:thm:finite dim:nonmultiples sum}
		\sup_{q\geqslant q_0}\sum_{\substack{j\geqslant q_0\\ q\nmid j}} q^\gamma j^{-\gamma}\abs[\Big]{\frac{\lambda_{j,q}}{\lambda_{q,q}}}
		\leqslant 0.01.
	\end{equation}
	Set $J\define \lfloor q^{2/3}\rfloor$.
	We bound the sum separately on the three ranges $j>q^2$, $q_0\leqslant j\leqslant J$, and $J<j\leqslant q^2$. It suffices to show that each contribution is $\mathcal{O}(q^{-c})$ for some $c>0$. Once this is done, enlarging $q_0$ gives \eqref{eq:thm:finite dim:nonmultiples sum}.

	First suppose $j>q^2$. By \eqref{eq:expression of linear functional}, we have $\abs{ \ell_i^{q,mq}(\mathbf{e}_{i, j})}\leqslant p_i\sin\varphi_i^{q,mq}=\mathcal{O}(1)$. The weights in \eqref{eq:def:weights x2 x3 x4} are bounded. Thus $\abs{\lambda_{j,q}}=\mathcal{O}(1)$. Therefore
	\begin{equation}\label{eq:thm:finite dim:nonmultiples large j contribution}
		\sum_{j>q^2} q^\gamma j^{-\gamma}\abs[\Big]{\frac{\lambda_{j,q}}{\lambda_{q,q}}}
		=\mathcal{O}\parentheses[\big]{q^{2-\gamma}}.
	\end{equation}
	On the range $q_0\leqslant j\leqslant q^2$ with $q\nmid j$, Lemma~\ref{lem:estimate of linear functions} and the second identity in \eqref{eq:weights identities} give
	\[
		\lambda_{j,q}
		=
		\sum_{m \in \set{2, \, 3, \, 4}} x_m^{(q)}\frac{a_2^{m,j,q}-b_2^{m,j,q}}{q^2}
		+\mathcal{O}\parentheses[\big]{j^2q^{-4}},
	\]
	where $a_2^{m,j,q}$ and $b_2^{m,j,q}$ are defined in \eqref{eq:lem:estimate of linear functions:coefficients left:q not mid j} and \eqref{eq:lem:estimate of linear functions:coefficients right:mq not mid j} respectively. Substituting these explicit expressions gives
	\[
		\lambda_{j,q}
		=
		(-1)^{j+1}\frac{2c_3 j}{q^2}
		\sum_{m\in\set{2, \, 3, \, 4}}x_m^{(q)}\parentheses[\big]{1-m^{-2}}
		\parentheses[\Bigg]{
			\frac{\sin\frac{\pi}{2q}}{\sin\frac{j\pi}{q}}
			+\frac{\sin\frac{\pi}{2mq}}{\sin\frac{j\pi}{mq}}
		}
		+\mathcal{O}\parentheses[\big]{j^2q^{-4}}.
	\]
	Using the identity $\sum_{m \in \set{2, \, 3, \, 4}} x_m^{(q)}(1 - m^{-2}) = 0$, we obtain
	\begin{equation}\label{eq:thm:finite dim:nonmultiples lambda intermediate}
		\lambda_{j,q}
		=
		(-1)^{j+1}\frac{2c_3 j}{q^2}
		\sum_{m\in\set{2, \, 3, \, 4}}x_m^{(q)}\parentheses[\big]{1-m^{-2}}
		\frac{\sin\frac{\pi}{2mq}}{\sin\frac{j\pi}{mq}}
		+\mathcal{O}\parentheses[\big]{j^2q^{-4}}.
	\end{equation}
	Subtracting $\frac{1}{2j}$ inside the bracket (allowed by the same identity), we obtain
	\begin{equation}\label{eq:thm:finite dim:nonmultiples lambda cancellation}
		\lambda_{j,q}
		=
		(-1)^{j+1} \frac{2c_3 j}{q^2}
		\sum_{m\in\set{2, \, 3, \, 4}}x_m^{(q)}\parentheses[\big]{1-m^{-2}}
		\parentheses[\Bigg]{
			\frac{\sin\frac{\pi}{2mq}}{\sin\frac{j\pi}{mq}}-\frac{1}{2j}
		}
		+\mathcal{O}\parentheses[\big]{j^2q^{-4}}.
	\end{equation}

	\smallskip

	\emph{Claim.} There exists $C>0$ such that for all sufficiently large $q$, all $m\in\set{2, \, 3, \, 4}$ and all integers $1\leqslant j\leqslant J = \lfloor q^{\frac{2}{3}}\rfloor$,
	\begin{equation}\label{eq:thm:finite dim:nonmultiples claim}
		\abs[\Bigg]{
			\frac{\sin\frac{\pi}{2mq}}{\sin\frac{j\pi}{mq}}-\frac{1}{2j}
		}
		\leqslant C\frac{j}{q^2}.
	\end{equation}
	\emph{Proof of the claim.} Set $\varpi\define \frac{\pi}{mq}$ and $\operatorname{sinc} u \define \frac{\sin u}{u}$. Then
	\[
		\frac{\sin\frac{\varpi}{2}}{\sin(j\varpi)}
		=
		\frac{\varpi/2}{j\varpi}\cdot \frac{\operatorname{sinc}\parentheses{\varpi/2}}{\operatorname{sinc}(j\varpi)}
		=
		\frac{1}{2j}\cdot \frac{\operatorname{sinc}\parentheses{\varpi/2}}{\operatorname{sinc}(j\varpi)}.
	\]
	For $1\leqslant j\leqslant J$, the bound $\abs{j\varpi}\leqslant \frac{\pi}{m}q^{-1/3}$ gives $\abs{j\varpi}\leqslant 1$ for all sufficiently large $q$. The function $\operatorname{sinc}$ is smooth, takes the value $1$ at $0$ and has vanishing derivative there, so on $\set{u : \abs{u} \leqslant 1}$ we have $\operatorname{sinc} u=1+\mathcal{O}(u^2)$, with $\operatorname{sinc}$ bounded away from $0$. It follows that
	\[
		\frac{\operatorname{sinc}\parentheses{\varpi/2}}{\operatorname{sinc}(j\varpi)}-1=\mathcal{O}\parentheses[\big]{j^2\varpi^2},
	\]
	and hence
	\[
		\frac{\sin\frac{\varpi}{2}}{\sin(j\varpi)}-\frac{1}{2j}
		=
		\frac{1}{2j}\parentheses[\bigg]{\frac{\operatorname{sinc}\parentheses{\varpi/2}}{\operatorname{sinc}(j\varpi)}-1}
		=
		\mathcal{O}\parentheses[\big]{j\varpi^2}
		=
		\mathcal{O}\parentheses[\Big]{\frac{j}{q^2}}.
	\]
	This proves the claim.

	\smallskip

	The factor $2c_3 j / q^2$ in \eqref{eq:thm:finite dim:nonmultiples lambda cancellation} is $\mathcal{O}(j/q^2)$. By \eqref{eq:thm:finite dim:nonmultiples claim}, the parenthesized difference in the summand has the same bound for $q_0\leqslant j\leqslant J$. The weights are bounded, and the error term in \eqref{eq:thm:finite dim:nonmultiples lambda cancellation} is already of order $\mathcal{O}(j^2 q^{-4})$. Noting that $\lambda_{j,q} = \mathcal{O}(j^2 q^{-4})$ in this range and that $\sum_{j \geqslant q_0} j^{2-\gamma}$ converges (since $\gamma > 3$), we obtain
	\begin{equation}\label{eq:thm:finite dim:nonmultiples small j contribution}
		\sum_{q_0\leqslant j\leqslant J} q^\gamma j^{-\gamma}\abs[\Big]{\frac{\lambda_{j,q}}{\lambda_{q,q}}}
		=
		\mathcal{O}\parentheses[\big]{q^{\gamma-4}}.
	\end{equation}

	Next, we consider the range $J < j \leqslant q^2$ with $q \nmid j$. Since $\abs{\sin\frac{\pi}{2mq}} \leqslant \frac{\pi}{2mq} = \mathcal{O}\parentheses[\big]{q^{-1}}$, it follows from~\eqref{eq:thm:finite dim:nonmultiples lambda intermediate} that
	\[
		\lambda_{j,q}
		=
		\mathcal{O}\parentheses[\Big]{\frac{j}{q^3}}\sum_{m\in\set{2, \, 3, \, 4}}\frac{1}{\abs{\sin\frac{j\pi}{mq}}}
		+\mathcal{O}\parentheses[\big]{j^2q^{-4}}.
	\]
	Consequently, there exists $C>0$ (independent of $q$) such that
	\[
		\sum_{\substack{J< j\leqslant q^2\\ q\nmid j}} q^\gamma j^{-\gamma}\abs[\Big]{\frac{\lambda_{j,q}}{\lambda_{q,q}}}
		\leqslant C q^{\gamma-3}\sum_{m\in\set{2, \, 3, \, 4}}\sum_{\substack{J<j\leqslant q^2\\ q\nmid j}}j^{1-\gamma}\frac{1}{\abs{\sin\frac{j\pi}{mq}}}
		+C q^{\gamma-4}\sum_{\substack{J<j\leqslant q^2\\ q\nmid j}}j^{2-\gamma}.
	\]
	Since $\gamma>3$ and $J=\lfloor q^{2/3}\rfloor$, the second term is $\mathcal{O}\parentheses[\big]{q^{\gamma-4}J^{3-\gamma}}=\mathcal{O}\parentheses[\big]{q^{\gamma/3-2}}$.

	It remains to bound the first term. 
	Writing $j = k(mq) + r$ with $k \geqslant 0$ and $1 \leqslant r \leqslant mq - 1$, we have $\abs{\sin\frac{j\pi}{mq}} = \abs{\sin\frac{r\pi}{mq}}$. The bound $\sin\theta \geqslant \frac{2}{\pi}\min\set{\theta, \, \pi - \theta}$ on $[0, \pi]$ gives the pointwise estimate
	\[
		\frac{1}{\abs{\sin\frac{r\pi}{mq}}} \leqslant \frac{mq}{2 \min\set{r, \, mq - r}}.
	\]
	Summing on $r \in [1, \lfloor mq/2 \rfloor]$ and doubling by symmetry gives
	\begin{equation}\label{eq:sin sum}
		\sum_{r = 1}^{mq - 1} \frac{1}{\abs{\sin\frac{r\pi}{mq}}} = \mathcal{O}\parentheses[\big]{q \log q}.
	\end{equation}

	For $k \geqslant 1$, we have $j \geqslant kmq$, so $j^{1-\gamma} \leqslant (kmq)^{1-\gamma}$. Combined with~\eqref{eq:sin sum}, we obtain
	\[
		\sum_{r = 1}^{mq - 1} \frac{(kmq + r)^{1-\gamma}}{\abs{\sin\frac{r\pi}{mq}}}
		\leqslant
		(kmq)^{1-\gamma} \sum_{r = 1}^{mq - 1} \frac{1}{\abs{\sin\frac{r\pi}{mq}}}
		=
		\mathcal{O}\parentheses[\big]{k^{1-\gamma}\, q^{2-\gamma} \log q}.
	\]
	Since $\gamma > 3$, the series $\sum_{k \geqslant 1} k^{1-\gamma}$ converges, and summing over $k$ gives $\mathcal{O}(q^{2-\gamma} \log q)$.

	For $k = 0$, we have $j = r$ with $J < r \leqslant mq - 1$, and we split this range at $r = \lfloor mq/2 \rfloor$.
	For $J < r \leqslant \lfloor mq/2 \rfloor$, we have $\min\set{r, \, mq - r} = r$, so the bound becomes $1/\abs{\sin\frac{r\pi}{mq}} \leqslant mq/(2r)$. It follows that
	\[
		\sum_{J < r \leqslant \lfloor mq/2 \rfloor} \frac{r^{1-\gamma}}{\abs{\sin\frac{r\pi}{mq}}}
		\leqslant
		\frac{mq}{2}\sum_{r > J} r^{-\gamma}
		=
		\mathcal{O}\parentheses[\big]{qJ^{1-\gamma}}.
	\]
	For $\lfloor mq/2 \rfloor < r \leqslant mq - 1$, we have $\min\set{r, \, mq - r} = mq - r$ and $r^{1-\gamma} \leqslant (mq/2)^{1-\gamma}$. Setting $r' = mq - r$ gives
	\[
		\sum_{\lfloor mq/2 \rfloor < r \leqslant mq - 1} \frac{r^{1-\gamma}}{\abs{\sin\frac{r\pi}{mq}}}
		\leqslant
		(mq/2)^{1-\gamma} \cdot \frac{mq}{2} \sum_{r' = 1}^{\lfloor mq/2 \rfloor} \frac{1}{r'}
		=
		\mathcal{O}\parentheses[\big]{q^{2-\gamma}\log q}.
	\]

	Combining these estimates gives
	\[
		\sum_{\substack{J < j \leqslant q^2 \\ q \nmid j}} j^{1-\gamma} \cdot \frac{1}{\abs{\sin\frac{j\pi}{mq}}}
		= \mathcal{O}\parentheses[\big]{q^{2-\gamma}\log q}
		+ \mathcal{O}\parentheses[\big]{qJ^{1-\gamma}}.
	\]
	Since $J=\lfloor q^{2/3}\rfloor$, we have $q^{\gamma-3}\cdot qJ^{1-\gamma}=\mathcal{O}\parentheses[\big]{q^{\gamma/3-4/3}}$.
	Multiplying by the prefactor $q^{\gamma-3}$ in the first term above and adding the second term gives
	\begin{equation}\label{eq:thm:finite dim:nonmultiples middle j contribution}
	\begin{aligned}
		\sum_{\substack{J<j\leqslant q^2\\ q\nmid j}} q^\gamma j^{-\gamma}\abs[\Big]{\frac{\lambda_{j,q}}{\lambda_{q,q}}}
		&=
		\mathcal{O}\parentheses[\big]{q^{-1}\log q}
		+\mathcal{O}\parentheses[\big]{q^{\gamma/3-4/3}}
		+\mathcal{O}\parentheses[\big]{q^{\gamma/3-2}}.
	\end{aligned}
	\end{equation}
	With $\gamma=\tfrac{7}{2}$, the estimates \eqref{eq:thm:finite dim:nonmultiples small j contribution}, \eqref{eq:thm:finite dim:nonmultiples middle j contribution}, and \eqref{eq:thm:finite dim:nonmultiples large j contribution} give contributions $\mathcal{O}(q^{-1/2})$, $\mathcal{O}(q^{-1/6})$, and $\mathcal{O}(q^{-3/2})$ from the ranges $q_0\leqslant j\leqslant J$, $J<j\leqslant q^2$ with $q\nmid j$, and $j>q^2$, respectively.
	Summing these three bounds gives
	\[
		\sum_{\substack{j\geqslant q_0\\ q\nmid j}} q^\gamma j^{-\gamma}\abs[\Big]{\frac{\lambda_{j,q}}{\lambda_{q,q}}}
		=
		\mathcal{O}\parentheses[\big]{q^{-1/6}}.
	\]
	For $q_0$ sufficiently large, this proves \eqref{eq:thm:finite dim:nonmultiples sum}.

	\smallskip

	By \eqref{eq:thm:finite dim:multiples sum} and \eqref{eq:thm:finite dim:nonmultiples sum}, for $q_0$ sufficiently large the off-diagonal matrix $K = \parentheses{K_{q, j}}_{q, j \geqslant q_0}$ defined by
	\begin{equation}\label{eq:thm:finite dim:K matrix def}
		K_{q, j} \define
		\begin{cases}
			\lambda_{j, q}/\lambda_{q, q} & \quad \text{if } j \neq q; \\
			0 & \quad \text{if } j = q
		\end{cases}
	\end{equation}
	satisfies $\matrixnorm{K}[\gamma, q_0] \leqslant 0.65 < 1$.

	Write the sequence $\mathbf{w} \define \parentheses{\widehat{\mathbf{n}}_{j}}_{j\geqslant 0}$ as $\mathbf{w}=\mathbf{w}^{\,\mathrm{low}}\oplus \mathbf{w}^{\,\mathrm{high}}$, where
	\[
		\mathbf{w}^{\,\mathrm{low}} \define \parentheses{ \widehat{\mathbf{n}}_{0},\widehat{\mathbf{n}}_{1},\dots,\widehat{\mathbf{n}}_{q_0-1} }^{T}
		\quad \text{ and } \quad
		\mathbf{w}^{\,\mathrm{high}} \define \parentheses{ \widehat{\mathbf{n}}_{q_0},\widehat{\mathbf{n}}_{q_0+1},\dots }^{T}.
	\]
	Dividing~\eqref{eq:thm:finite dim:lambda system} by $\lambda_{q, q}$ and applying the decomposition $\mathbf{w} = \mathbf{w}^{\,\mathrm{low}} \oplus \mathbf{w}^{\,\mathrm{high}}$ gives, for each $q \geqslant q_0$,
	\[
		\mathbf{w}^{\,\mathrm{high}} + K\,\mathbf{w}^{\,\mathrm{high}}
		= -K_{\mathrm{low}}\,\mathbf{w}^{\,\mathrm{low}},
	\]
	where $K_{\mathrm{low}}$ is the matrix indexed by $q\geqslant q_0$ and $0\leqslant j\leqslant q_0-1$ with entries $K_{qj}=\frac{\lambda_{j,q}}{\lambda_{q,q}}$.

	We show $K_{\mathrm{low}}\,\mathbf{w}^{\,\mathrm{low}} \in h_{\gamma,q_0}$. For each fixed $1 \leqslant j < q_0$, Remark~\ref{rem:estimate of linear functions:special cases}~(ii) gives
	\[
		\lambda_{j,q}=\mathcal{O}_j\parentheses[\big]{q^{-4}}
		\qquad \text{as } q\to+\infty.
	\]
	For $j = 0$, Corollary~\ref{cor:zeroth mode difference} gives $\ell_1^{q,mq}(\mathbf{e}_{1,0}) - \ell_2^{q,mq}(\mathbf{e}_{2,0}) = -\beta(1-m^{-2})\,q^{-2} + \mathcal{O}(q^{-4})$ with $\beta$ independent of $m$. Summing with weights $x_m^{(q)}$ and using $\sum_{m} x_m^{(q)}(1-m^{-2}) = 0$, the $q^{-2}$ term cancels and $\lambda_{0,q} = \mathcal{O}(q^{-4})$.
	The bound $\abs{\lambda_{q,q}} \geqslant 99/100$ for $q \geqslant q_0$ and the finiteness of the index set $\set{0, \dots, q_0-1}$ imply
	\[
		(K_{\mathrm{low}}\,\mathbf{w}^{\,\mathrm{low}})_q
		=
		\mathcal{O}\parentheses[\big]{q^{-4}}
		\qquad \text{as } q\to+\infty.
	\]
	Since $\gamma < 4$, we have $q^{\gamma}\, \abs{(K_{\mathrm{low}}\,\mathbf{w}^{\,\mathrm{low}})_q} = \mathcal{O}\bigl(q^{\gamma - 4}\bigr) = o(1)$ as $q \to +\infty$, hence $K_{\mathrm{low}}\,\mathbf{w}^{\,\mathrm{low}} \in h_{\gamma, q_0}$.

	Moreover, the matrix $K$ preserves $h_{\gamma, q_0}$.
	Indeed, for each fixed $j \geqslant q_0$, every $q > j$ satisfies $q \nmid j$, so the non-multiple estimate $\sum_{j \geqslant q_0,\, q \nmid j} q^{\gamma} j^{-\gamma} \abs{K_{q, j}} = \mathcal{O}(q^{-1/6})$ gives $q^{\gamma} \abs{K_{q, j}} \leqslant j^{\gamma}\, \mathcal{O}\bigl(q^{-1/6}\bigr) \to 0$, and the invariance criterion of Subsection~\ref{sub:linearized isospectral operator J} applies.
	The bound $\matrixnorm{K}[\gamma,q_0]<1$ therefore implies that $\operatorname{Id} + K$ is invertible on $h_{\gamma,q_0}$ by Neumann series.
	Therefore,
	\[
		\mathbf{w}^{\,\mathrm{high}}
		=-(\operatorname{Id}+K)^{-1}K_{\mathrm{low}}\,\mathbf{w}^{\,\mathrm{low}}.
	\]
	The full sequence $\set{\widehat{\mathbf{n}}_{j}}_{j\geqslant 0}$ is therefore uniquely determined by the finitely many low-frequency coefficients $\parentheses{ \widehat{\mathbf{n}}_{0},\widehat{\mathbf{n}}_{1},\dots,\widehat{\mathbf{n}}_{q_0-1} }$. Together with the relation $\widehat{\mathbf{n}}_{2, j}=-\widehat{\mathbf{n}}_{j}$ established at the beginning of the proof, this shows that the low-frequency projection is injective on $\ker\mathcal{J}$.
	Hence, $\ker\mathcal{J}$ is finite-dimensional and $\dim\ker\mathcal{J}\leqslant q_0$.
	The threshold $q_0$ depends on $a$: it must exceed the existence thresholds $q_{\ast}(m, a)$ of Lemma~\ref{lem:Taylor expansion of collision angle} and absorb remainders whose implicit constants involve $c_3 = \pi^{2}/(8a)$.
	The resulting bound on the dimension is therefore not uniform in $a$, although the finite-dimensionality itself holds for every $a > 0$.
\end{proof}

\begin{remark}\label{rem:choice of weights}
	We explain here why three orbit families ($m \in \set{2, \, 3, \, 4}$) are needed and how the symmetric weights~\eqref{eq:def:weights x2 x3 x4} are selected.
	If we used only two orbit families, say $m \in \set{2, \, 3}$, the two linear constraints in~\eqref{eq:weights identities} would uniquely determine the weights as $\parentheses[\big]{ x_2^{(q)}, \, x_3^{(q)} } = (-1)^{q+1}\frac{2}{\pi} \parentheses[\big]{ \frac{32}{5}, \, -\frac{27}{5} }$.
	With these weights, the four case bounds of Lemma~\ref{lem:P bounds} would become $\frac{32}{5}$, $\frac{64}{5}$, $\frac{27}{5}$, $\frac{27}{5}$ instead of $1$, $\frac{155}{27}$, $\frac{256}{27}$, $1$, and their sum would exceed $1.38$.
	The norm bound $\matrixnorm{K}[\gamma, q_0] < 1$ would then fail, preventing the inversion of $\operatorname{Id} + K$ via Neumann series.

	Introducing the third orbit family ($m = 4$) provides a one-parameter family of admissible weights.
	The dominant contribution to the bound of Lemma~\ref{lem:P bounds}~(ii) comes from the even values of $s$: with the weights~\eqref{eq:def:weights x2 x3 x4}, they contribute $0.522$ for $s \equiv 2 \pmod 4$, $0.077$ for $s \equiv 4 \pmod 8$, and $0.001$ for $8 \mid s$.
	The symmetric choice $x_2^{(q)} = -x_4^{(q)}$ controls the odd values of $s$: for $3 \nmid s$ the first identity in~\eqref{eq:weights identities} gives $P(s) = -1$ whatever the weights, while for $3 \mid s$ the symmetry forces $P(s) = 0$, so that $\abs{P(s)} \leqslant 1$ for every odd $s$ and the odd values contribute at most $0.028$.
	The total is $0.6252 < 1$, and the operator $\operatorname{Id} + K$ is invertible on $h_{\gamma, q_0}$.
\end{remark}

%% file: section/Injectivity.tex
\section{Triviality of the kernel}
\label{sec:triviality of the kernel}

In this section we prove that the kernel of the linearized isospectral operator $\mathcal{J}$ on $C^r_*(S)$ is trivial for stadia with sufficiently long flat edges.
The hypothesis that $a$ is large is essential to our method, as the perturbation parameter $c_3 = \pi^{2}/(8a)$ diverges as $a \to 0$.
In addition, the near-circle rigidity results of~\cite{MR3665005} do not apply to stadia with small $a$, since the jump discontinuity of the curvature at the four junctions persists for every $a > 0$.

\begin{theorem}    \label{thm:standard stadium rigidity}
	There exists $a_0 > 0$ such that for all $a \geqslant a_0$, the kernel of the linearized isospectral operator $\mathcal{J}$ on $C^r_*(S)$ is trivial.
\end{theorem}

The proof compares the linearized system with its limit as $a \to +\infty$, where it becomes explicit.
Below we introduce the objects of this comparison and record their basic properties; the supporting lemmas bound the infinite tail of the system (Lemma~\ref{lem:tail invertibility large a}) and the tail correction to its finite part (Lemma~\ref{lem:admissible threshold}).

\smallskip

Throughout this section, let $q_0 \define 2$, $\gamma = \tfrac{7}{2}$, and $c_3 = \pi^2/(8a)$.
All implicit constants in $\mathcal{O}(\cdot)$ are absolute; in particular, they are independent of $q$, $j$, $s$, $m$, or $a$.

We use throughout the weights~\eqref{eq:def:weights x2 x3 x4}, the weighted combination $\lambda_{j, q}$ of~\eqref{eq:def:lambda j q} and the identities~\eqref{eq:weights identities}, all fixed in Subsection~\ref{sub:weighted combination}.

The coefficients $\lambda_{j, q}$ in~\eqref{eq:def:lambda j q} depend on the parameter $a$ only through the collision angles $\varphi_i^{q, mq}$.
Indeed, by~\eqref{eq:expression of linear functional} and the definition~\eqref{eq:def:Phi} of $\Phi^{j, p}$, we have
\[
	\ell_i^{q, mq}(\mathbf{e}_{i, j}) = \Phi^{j, p_i}(\varepsilon_i),
	\qquad
	(p_1, p_2) = (q, mq),
	\qquad
	\varepsilon_i \define \varphi_i^{q, mq} - \frac{\pi}{2 p_i}.
\]
We record two basic properties of these functionals.
First, since the cosine sum in~\eqref{eq:expression of linear functional} consists of $p_i$ terms of modulus at most $1$, the inequality $\sin x \leqslant x$ yields the universal bound
\begin{equation}\label{eq:triviality:universal bound}
	\abs[\big]{\ell_i^{p_1, p_2}(\mathbf{e}_{i, j})}
	\leqslant p_i \sin\varphi_i^{p_1, p_2}
	\leqslant \frac{\pi}{2} + p_i\,\abs{\varepsilon_i}
	\qquad (j \in \n_0).
\end{equation}
Second, since $\varepsilon_i \to 0$ as $a \to +\infty$ by Lemmas~\ref{lem:short orbit control} and~\ref{lem:tail orbit asymptotics}, evaluating at $\varepsilon_i = 0$ yields the \emph{limiting values}.
These coincide with the coefficients $D_0^{j, p_i}$ computed in~\eqref{eq:lem:Taylor expansion of auxiliary variable:coefficients:p mid j} and~\eqref{eq:lem:Taylor expansion of auxiliary variable:coefficients:p not mid j}:
\begin{equation}\label{eq:thm:rigidity:leading order value}
	\ell_i^{p_1, p_2}(\mathbf{e}_{i, j})\Big|_{\varepsilon_i = 0}
	=
	\begin{cases}
		(-1)^{j + j/p_i}\, p_i \sin\dfrac{\pi}{2p_i} & \quad \text{if } p_i \mid j; \\[4pt]
		0 & \quad \text{if } p_i \nmid j,
	\end{cases}
	\qquad (j \in \n_0),
\end{equation}
with the value $p_i \sin\frac{\pi}{2 p_i}$ at $j = 0$.

To control deviations from the limiting values, we expand $\Phi^{j, p}$ by Taylor's theorem with Lagrange remainder:
\begin{equation}\label{eq:lem:low mode:taylor}
	\Phi^{j, p}(\varepsilon_i) - \Phi^{j, p}(0)
	= D_1^{j, p}\, \varepsilon_i + \tfrac{1}{2}\, \bigl(\Phi^{j, p}\bigr)''(\vartheta)\, \varepsilon_i^2,
\end{equation}
where $\vartheta$ lies between $0$ and $\varepsilon_i$, and $D_1^{j, p} = (\Phi^{j, p})'(0)$ is the linear coefficient from Lemma~\ref{lem:Taylor expansion of auxiliary variable}.
Differentiating $\Phi^{j, p}(\varepsilon) = \sin\bigl(\frac{\pi}{2p} + \varepsilon\bigr)\, f_p\bigl(\frac{j\pi}{p} + 2j\varepsilon\bigr)$ twice yields
\[
	\bigl(\Phi^{j, p}\bigr)''(\vartheta)
	= -\sin\parentheses[\Big]{\frac{\pi}{2p} + \vartheta} f_p + 4j \cos\parentheses[\Big]{\frac{\pi}{2p} + \vartheta} f_p' + 4 j^2 \sin\parentheses[\Big]{\frac{\pi}{2p} + \vartheta} f_p'',
\]
where $f_p$ and its derivatives are evaluated at $\frac{j\pi}{p} + 2j\vartheta$.
In view of the bounds $\abs{f_p^{(k)}} = \mathcal{O}(p^{k+1})$ from the proof of Lemma~\ref{lem:Taylor expansion of auxiliary variable} and the prefactor estimate $\sin\bigl(\frac{\pi}{2p} + \vartheta\bigr) \leqslant \frac{\pi}{2p} + \abs{\varepsilon_i} = \mathcal{O}(p^{-1})$ (which holds since $p\,\abs{\varepsilon_i} = \mathcal{O}(1)$ for all orbits by Lemmas~\ref{lem:short orbit control} and~\ref{lem:tail orbit asymptotics}), the three terms are bounded by $\mathcal{O}(1)$, $\mathcal{O}(j p^2)$, and $\mathcal{O}(j^2 p^2)$, respectively.
This yields $\abs[\big]{\bigl(\Phi^{j, p}\bigr)''(\vartheta)} = \mathcal{O}\bigl((1 + j^2)\, p^2\bigr)$ uniformly for $\vartheta$ between $0$ and $\varepsilon_i$.
Similarly, the first derivative satisfies $\abs[\big]{\bigl(\Phi^{j, p}\bigr)'(\varepsilon)} = \mathcal{O}(j\, p)$ for $j \geqslant 1$, uniformly for $\abs{\varepsilon} \leqslant \abs{\varepsilon_i}$.

We denote by $\lambda^{\infty}_{j, q}$ the limiting value of $\lambda_{j, q}$ obtained by evaluating each functional in~\eqref{eq:def:lambda j q} at $\varepsilon_i = 0$.
Since $\varepsilon_i \to 0$ as $a \to +\infty$ by Lemma~\ref{lem:tail orbit asymptotics}, we have $\lim_{a \to +\infty} \lambda_{j, q} = \lambda^{\infty}_{j, q}$ for each fixed pair $(j, q)$.
If $j \geqslant 1$ and $q \nmid j$, then $mq \nmid j$ for every $m \in \set{2, \, 3, \, 4}$; both terms in~\eqref{eq:def:lambda j q} therefore vanish at $\varepsilon_i = 0$, giving $\lambda^{\infty}_{j, q} = 0$.
On the diagonal $j = q$, the second limiting value vanishes since $mq \nmid q$, and the first weight identity in~\eqref{eq:weights identities} gives
\begin{equation}\label{eq:triviality:frozen diagonal}
	\lambda^{\infty}_{q, q}
	= \sum_{m \in \set{2, \, 3, \, 4}} x_m^{(q)}\, (-1)^{q + 1}\, q \sin\frac{\pi}{2q}
	= \frac{2}{\pi}\, q \sin\frac{\pi}{2q}
	\;\geqslant\; \frac{2\sqrt{2}}{\pi}
	\;>\; 0.9003
	\qquad (q \geqslant q_0),
\end{equation}
where the lower bound is attained at $q = q_0 = 2$, as the function $x \mapsto x \sin\frac{\pi}{2x}$ is strictly increasing on $[1, +\infty)$.

Whenever the orbits $\Gamma^{q, mq}$ with $q \geqslant q_0$ and $m \in \set{2, \, 3, \, 4}$ exist and no diagonal coefficient $\lambda_{q, q}$ vanishes, we define
\begin{equation}\label{eq:triviality:K def}
	K_{q, j} \define
	\begin{cases}
		\lambda_{j, q}/\lambda_{q, q} & \quad \text{if } j \neq q; \\
		0 & \quad \text{if } j = q
	\end{cases}
	\qquad (q, j \geqslant q_0),
	\qquad\quad
	K_{\mathrm{low}} \define \parentheses[\big]{ \lambda_{j, q} / \lambda_{q, q} }_{q \geqslant q_0, \, 0 \leqslant j < q_0}.
\end{equation}
The limiting matrices $K^{\infty} = (K^{\infty}_{q, j})_{q, j \geqslant q_0}$ and $K^{\infty}_{\mathrm{low}} = (K^{\infty}_{\mathrm{low}, q, j})_{q \geqslant q_0, \, 0 \leqslant j < q_0}$ are defined by the same formulas with $\lambda^{\infty}$ in place of $\lambda$.
Because the limiting values~\eqref{eq:thm:rigidity:leading order value} depend only on arithmetic divisibility and $\lambda^{\infty}_{q, q} > 0$ by~\eqref{eq:triviality:frozen diagonal}, the matrices $K^{\infty}$ and $K^{\infty}_{\mathrm{low}}$ are well-defined independently of $a$.
Recall from Subsection~\ref{sub:linearized isospectral operator J} the space $h_{\gamma, q_0}$, the matrix norm $\matrixnorm{\cdot}[\gamma, q_0]$, and the invariance criterion: every matrix $L = (L_{q, j})_{q, j \geqslant q_0}$ whose columns satisfy $q^{\gamma} \abs{L_{q, j}} \to 0$ as $q \to +\infty$ for each fixed $j$ defines a bounded operator on $h_{\gamma, q_0}$, with operator norm bounded by $\matrixnorm{L}[\gamma, q_0]$.
Both $K^{\infty}$ and $K$ satisfy this column decay condition.
Indeed, $K^{\infty}_{q, j} = 0$ whenever $q > j/2$, because the off-diagonal entries in row $q$ are supported on multiples $j \in \set{2q, \, 3q, \, \dots}$; thus each column of $K^{\infty}$ has only finitely many non-zero entries.
For $K$, the column decay follows from the estimates in the proof of Lemma~\ref{lem:tail invertibility large a}.
Finally, since $\lambda^{\infty}_{j, q} = 0$ for $1 \leqslant j < q_0 \leqslant q$, only the $j = 0$ column of $K^{\infty}_{\mathrm{low}}$ is non-zero.

Finally, let $\widetilde{A}^{\flat}$ be the $2 \times 2$ matrix with rows indexed, in this order, by the orbits $\Gamma^{1, 2}$ and $\Gamma^{1, 3}$, columns indexed by $j \in \set{0, \, 1}$, and entries
\[
	\widetilde{A}^{\flat}_{(1, m), j}
	\define \Diffl{j}{1}{m}\big|_{\varepsilon_i = 0}
	\qquad (m \in \set{2, \, 3},\ j \in \set{0, \, 1}).
\]
Here $p_1 = 1$ divides every $j$, while $p_2 = m$ divides $j = 0$ but not $j = 1$; hence, by the limiting values~\eqref{eq:thm:rigidity:leading order value} together with $2\sin\frac{\pi}{4} = \sqrt{2}$ and $3\sin\frac{\pi}{6} = \tfrac{3}{2}$, we have
\begin{equation}\label{eq:triviality:A flat}
	\widetilde{A}^{\flat} =
	\begin{pmatrix}
		1 - \sqrt{2} & 1 \\[2pt]
		-\tfrac{1}{2} & 1
	\end{pmatrix},
	\qquad
	\det \widetilde{A}^{\flat} = \tfrac{3}{2} - \sqrt{2} > 0.
\end{equation}
In particular, $\widetilde{A}^{\flat}$ is independent of $a$.
For a square matrix $B$, let $\sigma_{\min}(B)$ and $\sigma_{\max}(B)$ denote its smallest and largest singular values, let $\norm{B}_{\mathrm{op}} \define \sigma_{\max}(B)$ denote its operator norm with respect to the Euclidean norm, and let $\norm{B}_F \define \bigl(\sum_{i, j} B_{i j}^2\bigr)^{1/2}$ denote its Frobenius norm.
The two singular values of a $2 \times 2$ matrix $B$ satisfy $\sigma_{\min} \sigma_{\max} = \abs{\det B}$ and $\sigma_{\min}^2 + \sigma_{\max}^2 = \norm{B}_F^2$; hence $\sigma_{\min}(B) \geqslant \abs{\det B}/\norm{B}_F$.
Since $\norm{\widetilde{A}^{\flat}}_F^2 = (1 - \sqrt{2})^2 + 1 + \frac{1}{4} + 1 = \frac{21}{4} - 2\sqrt{2}$, we obtain
\begin{equation}\label{eq:triviality:sigma flat}
	\sigma_{\min}\bigl(\widetilde{A}^{\flat}\bigr)
	\geqslant \frac{\abs{\det \widetilde{A}^{\flat}}}{\norm{\widetilde{A}^{\flat}}_F}
	= \frac{\tfrac{3}{2} - \sqrt{2}}{\sqrt{\tfrac{21}{4} - 2\sqrt{2}}}
	\geqslant 5.5 \times 10^{-2}.
\end{equation}

The following lemma establishes the uniform lower bound on the diagonal entries and controls the low-frequency columns $j \in \set{0, \, 1}$.

\begin{lemma}
\label{lem:low mode estimates}
	There exists a constant $a_{\mathrm{low}}(q_0) \geqslant 12$ such that for all $a \geqslant a_{\mathrm{low}}(q_0)$ and all $q \geqslant q_0$, the following statements hold:
	\begin{enumerate}
		\smallskip

		\item\label{itm:low mode diagonal}
		$\abs[\big]{\lambda_{q, q} - \lambda^{\infty}_{q, q}} = \mathcal{O}(c_3\, q^{-2})$; in particular, $\abs{\lambda_{q, q}} \geqslant \tfrac{7}{8}$, so that the matrices $K$ and $K_{\mathrm{low}}$ in~\eqref{eq:triviality:K def} are well-defined;

		\smallskip

		\item\label{itm:low mode one}
		$\lambda^{\infty}_{1, q} = 0$ and $\abs{\lambda_{1, q}} = \mathcal{O}(c_3\, q^{-4})$;

		\smallskip

		\item\label{itm:low mode zero}
		$\abs[\big]{\lambda^{\infty}_{0, q}} = \mathcal{O}(q^{-4})$ and $\abs[\big]{\lambda_{0, q} - \lambda^{\infty}_{0, q}} = \mathcal{O}(c_3\, q^{-4})$.
	\end{enumerate}
\end{lemma}
\begin{proof}
	Since $a \geqslant 12$, the periodic orbits $\Gamma^{q, mq}$ exist and are regular, Lemma~\ref{lem:tail orbit asymptotics} applies, and $c_3 = \pi^2/(8a) \leqslant \pi^2/96 < 1$.
	Throughout the proof, we expand each functional by the first-order expansion~\eqref{eq:lem:low mode:taylor}.

	\smallskip

	\ref{itm:low mode diagonal} (The diagonal).
	For $j = q$, we have $p_1 = q \mid q$, so~\eqref{eq:lem:Taylor expansion of auxiliary variable:coefficients:p mid j} gives $D_1^{q, q} = (-1)^{q + 1}\, q \cos\frac{\pi}{2q}$, whence $\abs{D_1^{q, q}} \leqslant q$.
	For the second functional, $p_2 = mq \nmid q$, and~\eqref{eq:lem:Taylor expansion of auxiliary variable:coefficients:p not mid j} yields
	\[
		D_1^{q, mq}
		= (-1)^{q}\, mq \cdot q\, \frac{2 \sin\frac{\pi}{2mq}}{\sin\frac{\pi}{m}},
		\qquad
		\abs[\big]{D_1^{q, mq}} \leqslant \frac{\pi q}{\sin\frac{\pi}{m}} \leqslant \sqrt{2}\, \pi\, q,
	\]
	using $2mq \sin\frac{\pi}{2mq} \leqslant \pi$ and $\sin\frac{\pi}{m} \geqslant \sin\frac{\pi}{4} = \frac{\sqrt{2}}{2}$ for $m \in \set{2, \, 3, \, 4}$.
	Applying the angle bound $\abs{\varepsilon_i} \leqslant C_0\, c_3\, q^{-3}$ from Lemma~\ref{lem:tail orbit asymptotics} together with $\bigl(\Phi^{q, p_i}\bigr)'' = \mathcal{O}(q^4)$, the expansion~\eqref{eq:lem:low mode:taylor} gives, for each $m \in \set{2, \, 3, \, 4}$,
	\[
		\abs[\Big]{ \Diffl{q}{q}{m} - \Diffl{q}{q}{m}\big|_{\varepsilon_i = 0} }
		= \mathcal{O}(q)\cdot\mathcal{O}(c_3\, q^{-3}) + \mathcal{O}(q^4)\cdot\mathcal{O}(c_3^2\, q^{-6})
		= \mathcal{O}(c_3\, q^{-2}).
	\]
	Because $c_3 \leqslant 1$, summing these estimates against the bounded weights $x_m^{(q)}$ proves the comparison in~\ref{itm:low mode diagonal}.
	Combining this bound with $\lambda^{\infty}_{q, q} \geqslant \frac{2\sqrt{2}}{\pi} > 0.9003$ from~\eqref{eq:triviality:frozen diagonal} and noting that $\frac{2\sqrt{2}}{\pi} - \frac{7}{8} > \frac{1}{40}$, we can choose $a_{\mathrm{low}}(q_0)$ sufficiently large so that
	\begin{equation}\label{eq:lem:tail invertibility:diagonal lower bound}
		\abs{\lambda_{q, q}} \geqslant \frac{7}{8}
		\qquad (q \geqslant q_0).
	\end{equation}

	\smallskip

	\ref{itm:low mode one} (The column $j = 1$).
	Since $q \geqslant q_0 > 1$, neither $q$ nor $mq$ divides $1$; thus $\lambda^{\infty}_{1, q} = 0$, and both functionals vanish at $\varepsilon_i = 0$.
	By the double-angle formula $\sin\frac{\pi}{p} = 2 \sin\frac{\pi}{2p} \cos\frac{\pi}{2p}$, the coefficients in~\eqref{eq:lem:Taylor expansion of auxiliary variable:coefficients:p not mid j} simplify to
	\begin{equation}\label{eq:lem:low mode:D1 at one}
		D_1^{1, p}
		= -\, p\, \frac{2\sin\frac{\pi}{2p}}{\sin\frac{\pi}{p}}
		= -\frac{p}{\cos\frac{\pi}{2p}}
		= -p\,\bigl(1 + \mathcal{O}(p^{-2})\bigr)
		\qquad (p \geqslant 2).
	\end{equation}
	Substituting the asymptotic expansions of Lemma~\ref{lem:tail orbit asymptotics}, we obtain
	\[
		D_1^{1, q}\, \varepsilon_1
		= c_3\,\bigl(1 - m^{-2}\bigr)\, q^{-2} + \mathcal{O}(c_3\, q^{-4}),
		\qquad
		D_1^{1, mq}\, \varepsilon_2
		= -\, c_3\,\bigl(1 - m^{-2}\bigr)\, q^{-2} + \mathcal{O}(c_3\, q^{-4}),
	\]
	while the second-order terms in~\eqref{eq:lem:low mode:taylor} contribute $\mathcal{O}(q^2)\cdot\mathcal{O}(c_3^2\, q^{-6}) = \mathcal{O}(c_3\, q^{-4})$.
	Hence
	\[
		\Diffl{1}{q}{m}
		= D_1^{1, q}\, \varepsilon_1 - D_1^{1, mq}\, \varepsilon_2 + \mathcal{O}(c_3\, q^{-4})
		= 2\, c_3\,\bigl(1 - m^{-2}\bigr)\, q^{-2} + \mathcal{O}(c_3\, q^{-4}),
	\]
	and the second weight identity in~\eqref{eq:weights identities} cancels the leading term:
	\[
		\lambda_{1, q}
		= 2\, c_3\, q^{-2} \sum_{m \in \set{2, \, 3, \, 4}} x_m^{(q)}\,\bigl(1 - m^{-2}\bigr) + \mathcal{O}(c_3\, q^{-4})
		= \mathcal{O}(c_3\, q^{-4}).
	\]

	\smallskip

	\ref{itm:low mode zero} (The column $j = 0$).
	Here both functionals reduce to $\ell_i^{q, mq}(\mathbf{e}_{i, 0}) = p_i \sin\varphi_i^{q, mq}$.
	For the limiting values, the expansion $p \sin\frac{\pi}{2p} = \frac{\pi}{2} - \frac{\pi^3}{48}\, p^{-2} + \mathcal{O}(p^{-4})$ gives
	\[
		\Diffl{0}{q}{m}\big|_{\varepsilon_i = 0}
		= q \sin\frac{\pi}{2q} - mq \sin\frac{\pi}{2mq}
		= -\frac{\pi^3}{48}\,\bigl(1 - m^{-2}\bigr)\, q^{-2} + \mathcal{O}(q^{-4}),
	\]
	and the second weight identity in~\eqref{eq:weights identities} yields $\lambda^{\infty}_{0, q} = \mathcal{O}(q^{-4})$.
	For the comparison, the limiting parts cancel exactly:
	\[
		\Diffl{0}{q}{m} - \Diffl{0}{q}{m}\big|_{\varepsilon_i = 0}
		= q\, \cos\frac{\pi}{2q}\cdot \varepsilon_1 - mq\, \cos\frac{\pi}{2mq}\cdot \varepsilon_2
		+ \mathcal{O}\bigl(q\,\varepsilon_1^2 + mq\, \varepsilon_2^2\bigr).
	\]
	Using Lemma~\ref{lem:tail orbit asymptotics} and the expansion $\cos\frac{\pi}{2p_i} = 1 + \mathcal{O}(q^{-2})$, we find that each of the terms $q\, \cos\frac{\pi}{2q}\cdot \varepsilon_1$ and $-mq\, \cos\frac{\pi}{2mq}\cdot \varepsilon_2$ equals $-c_3\,(1 - m^{-2})\, q^{-2} + \mathcal{O}(c_3\, q^{-4})$, while the quadratic remainder is $\mathcal{O}(q\, c_3^2\, q^{-6}) = \mathcal{O}(c_3\, q^{-4})$.
	Hence
	\[
		\Diffl{0}{q}{m} - \Diffl{0}{q}{m}\big|_{\varepsilon_i = 0}
		= -2\, c_3\,\bigl(1 - m^{-2}\bigr)\, q^{-2} + \mathcal{O}(c_3\, q^{-4}),
	\]
	and summing against the weights $x_m^{(q)}$ cancels the leading term once again, leaving $\abs{\lambda_{0, q} - \lambda^{\infty}_{0, q}} = \mathcal{O}(c_3\, q^{-4})$.
\end{proof}

The next lemma is the main estimate of the tail analysis.

\begin{lemma}\label{lem:tail invertibility large a}
	Let $q_0 = 2$ and $\gamma = \tfrac{7}{2}$.
	The limiting matrix $K^{\infty}$ satisfies $\matrixnorm{K^{\infty}}[\gamma, q_0] \leqslant 0.7323$.
	Moreover, there exists $a_K \geqslant a_{\mathrm{low}}(q_0)$, with $a_{\mathrm{low}}(q_0)$ as in Lemma~\ref{lem:low mode estimates}, such that for each $a \geqslant a_K$ the following hold:
	\begin{enumerate}
		\smallskip

		\item\label{itm:tail invertibility orbits}
		the hybrid periodic orbits $\Gamma^{q, mq}$ exist and are regular for all $q \geqslant q_0$ and all $m \in \set{2, \, 3, \, 4}$;

		\smallskip

		\item\label{itm:tail invertibility comparison}
		$\matrixnorm{K - K^{\infty}}[\gamma, q_0] \leqslant C_K\, c_3$ for a constant $C_K$ depending only on $q_0$; in particular,
		\[
			\matrixnorm{K - K^{\infty}}[\gamma, q_0] \leqslant 0.0677
			\qquad \text{and} \qquad
			\matrixnorm{K}[\gamma, q_0] \leqslant 0.8.
		\]
	\end{enumerate}
\end{lemma}

Since $0.7323 < 1$ and $0.8 < 1$, the Neumann series shows that $\mathrm{Id} + K^{\infty}$ is invertible on $h_{\gamma, q_0}$, and $\mathrm{Id} + K$ is invertible for each $a \geqslant a_K$, with
\begin{equation}\label{eq:triviality:resolvent bounds}
	\matrixnorm[\big]{(\mathrm{Id} + K^{\infty})^{-1}}[\gamma, q_0] \leqslant \frac{1}{1 - 0.7323} \leqslant 4,
	\qquad
	\matrixnorm[\big]{(\mathrm{Id} + K)^{-1}}[\gamma, q_0] \leqslant \frac{1}{1 - 0.8} = 5.
\end{equation}

\begin{proof}
	We bound the limiting matrix, then the difference $K - K^{\infty}$, and conclude by fixing $a_K$.

	Since $\lambda^{\infty}_{j, q} = 0$ for $q \nmid j$ and the diagonal entry of $K^{\infty}$ is $0$ by definition, row $q$ of $K^{\infty}$ is supported on the multiples $j = sq$ with $s \geqslant 2$; since $q^{\gamma} (sq)^{-\gamma} = s^{-\gamma}$, we have
	\[
		\matrixnorm{K^{\infty}}[\gamma, q_0]
		= \sup_{q \geqslant q_0}\, \frac{1}{\lambda^{\infty}_{q, q}} \sum_{s = 2}^{+\infty} s^{-\gamma}\, \abs[\big]{\lambda^{\infty}_{sq, q}}.
	\]
	By the limiting values~\eqref{eq:thm:rigidity:leading order value},
	\begin{equation}\label{eq:lem:tail invertibility:frozen multiples}
		\lambda^{\infty}_{sq, q}
		= \sum_{m \in \set{2, \, 3, \, 4}} x_m^{(q)} \parentheses[\Big]{ (-1)^{sq + s}\, q \sin\frac{\pi}{2q} - \mathbf{1}_{m \mid s}\, (-1)^{sq + s/m}\, mq \sin\frac{\pi}{2mq} }.
	\end{equation}
	We approximate both factors $p \sin\frac{\pi}{2p}$ in~\eqref{eq:lem:tail invertibility:frozen multiples} by $\frac{\pi}{2}$. From $u - \tfrac{u^3}{6} \leqslant \sin u < u$ at $u = \frac{\pi}{2p}$, we obtain
	\begin{equation}\label{eq:lem:tail invertibility:sine replacement}
		0 < \frac{\pi}{2} - p \sin\frac{\pi}{2p} \leqslant \frac{\pi^3}{48}\, p^{-2}
		\qquad \text{for every real } p \geqslant 1.
	\end{equation}
	Using $(-1)^{sq + s} - \mathbf{1}_{m \mid s}\,(-1)^{sq + s/m} = (-1)^{sq}\, d_m(s)$ with $d_m$ as in~\eqref{eq:def:P}, and the products $x_m^{(q)} \cdot \frac{\pi}{2} = (-1)^{q+1}\bigl(\tfrac{128}{27}, \, 1, \, -\tfrac{128}{27}\bigr)$ for $m = 2, 3, 4$, we obtain
	\begin{equation}\label{eq:lem:tail invertibility:P decomposition}
		\lambda^{\infty}_{sq, q} = (-1)^{sq + q + 1}\, P(s) + E(s, q),
	\end{equation}
	The error $E(s, q)$ collects the two replacements~\eqref{eq:lem:tail invertibility:sine replacement}. The $q$-replacement is common to the three values of $m$ and carries the weight $\abs[\big]{\sum_m x_m^{(q)}} = \frac{2}{\pi}$, while the $mq$-replacements carry the individual weights $\abs[\big]{x_m^{(q)}}$ and gain a factor $m^{-2}$, with $\sum_m \abs[\big]{x_m^{(q)}}\, m^{-2} = \frac{2}{\pi} \cdot \frac{43}{27}$.
	Hence
	\begin{equation}\label{eq:lem:tail invertibility:E bound}
		\abs{E(s, q)}
		\leqslant \frac{\pi^2}{24} \parentheses[\Big]{ 1 + \frac{43}{27} } q^{-2}
		= \frac{35\pi^2}{324}\, q^{-2}
		< 1.07\, q^{-2},
		\qquad \text{uniformly in } s.
	\end{equation}
	Summing the bounds of Lemma~\ref{lem:P bounds}~(i) and the error~\eqref{eq:lem:tail invertibility:E bound} against $s^{-\gamma}$, and evaluating the resulting series by Lemma~\ref{lem:P bounds}~(ii), we obtain
	\begin{equation}\label{eq:lem:tail invertibility:K infty bound}
		\sum_{s = 2}^{+\infty} s^{-\gamma}\, \abs[\big]{\lambda^{\infty}_{sq, q}}
		\leqslant 0.6252 + 1.07\, \bigl(\zeta(\gamma) - 1\bigr)\, q^{-2}
		\qquad \bigl(\gamma = \tfrac{7}{2}\bigr).
	\end{equation}
	For $q \geqslant q_0 = 2$ the remainder is at most $1.07\,\bigl(\zeta(\tfrac{7}{2}) - 1\bigr)/4 < 0.0340$, so the right-hand side is at most $0.6252 + 0.0340 = 0.6592$.
	Dividing by the diagonal lower bound~\eqref{eq:triviality:frozen diagonal} gives
	\[
		\matrixnorm{K^{\infty}}[\gamma, q_0]
		\;\leqslant\; \frac{0.6592}{0.9003}
		\;<\; 0.7323.
	\]

	Assume $a \geqslant a_{\mathrm{low}}(q_0)$, so that the orbits exist, Lemmas~\ref{lem:tail orbit asymptotics} and~\ref{lem:low mode estimates} apply, $c_3 \leqslant 1$, and $\abs{\lambda_{q, q}} \geqslant \tfrac{7}{8}$.
	We show that
	\begin{equation}\label{eq:lem:tail invertibility:K comparison}
		\matrixnorm{K - K^{\infty}}[\gamma, q_0] \leqslant C_K\, c_3.
	\end{equation}
	Fix $q \geqslant q_0$ and split the column index $j \neq q$ into non-multiples and multiples of $q$.

	\smallskip

	\emph{Non-multiples ($q \nmid j$).}
	Here we have $\lambda^{\infty}_{j, q} = 0$ and hence $K^{\infty}_{q, j} = 0$; since $\abs{\lambda_{q, q}} \geqslant \tfrac{7}{8}$, it suffices to prove
	\begin{equation}\label{eq:lem:tail invertibility:nonmultiple goal}
		\sum_{\substack{j \geqslant q_0 \\ q \nmid j}} q^{\gamma}\, j^{-\gamma}\, \abs{\lambda_{j, q}}
		= \mathcal{O}(c_3).
	\end{equation}
	We treat the ranges $j \leqslant J \define \lfloor q^{2/3} \rfloor$, $J < j \leqslant q^2$, and $j > q^2$ in turn.

	For $j \leqslant q^2$ we expand $\lambda_{j, q}$ to first order.
	Both functionals vanish at $\varepsilon_i = 0$ (as $q \nmid j$ forces $mq \nmid j$), and the coefficients of Lemma~\ref{lem:Taylor expansion of auxiliary variable} satisfy, for $p \in \set{q, \, mq}$ and $p \nmid j$,
	\[
		D_1^{j, p} = (-1)^j\, p\, j\, \frac{2 \sin\frac{\pi}{2p}}{\sin\frac{j\pi}{p}},
		\qquad
		\abs[\big]{D_1^{j, p}} \leqslant \frac{\pi\, j}{\abs[\big]{\sin\frac{j\pi}{p}}} \leqslant \frac{\pi\, j}{\sin\frac{\pi}{p}} \leqslant \frac{\pi}{2}\, p\, j,
	\]
	where $\abs{\sin\frac{j\pi}{p}} \geqslant \sin\frac{\pi}{p} \geqslant \frac{2}{p}$ for $p \nmid j$ (write $j = kp + r$ with $0 < r < p$).
	Substituting the expansions of Lemma~\ref{lem:tail orbit asymptotics} into~\eqref{eq:lem:low mode:taylor}, we find that the leading terms carry the sine ratios, while the $\mathcal{O}(c_3\, q^{-5})$ remainders of $\varepsilon_i$ contribute $\abs{D_1^{j, p}} \cdot \mathcal{O}(c_3\, q^{-5}) = \mathcal{O}(c_3\, j\, q^{-4})$ and the second-order terms contribute $\mathcal{O}(j^2 q^2) \cdot \mathcal{O}(c_3^2\, q^{-6}) = \mathcal{O}(c_3\, j^2\, q^{-4})$; altogether
	\begin{equation}\label{eq:lem:tail invertibility:nonmultiple expansion}
		\lambda_{j, q}
		= (-1)^{j+1}\, \frac{2\, c_3\, j}{q^2}
		\sum_{m \in \set{2, \, 3, \, 4}} x_m^{(q)}\, \bigl(1 - m^{-2}\bigr)
		\parentheses[\Bigg]{
			\frac{\sin\frac{\pi}{2q}}{\sin\frac{j\pi}{q}}
			+ \frac{\sin\frac{\pi}{2mq}}{\sin\frac{j\pi}{mq}}
		}
		+ \mathcal{O}\bigl(c_3\, j^2\, q^{-4}\bigr).
	\end{equation}
	The first fraction is independent of $m$, so the second weight identity in~\eqref{eq:weights identities} cancels it; the same identity also allows us to insert $-\tfrac{1}{2j}$ into the second fraction:
	\begin{equation}\label{eq:lem:tail invertibility:nonmultiple cancellation}
		\lambda_{j, q}
		= (-1)^{j+1}\, \frac{2\, c_3\, j}{q^2}
		\sum_{m \in \set{2, \, 3, \, 4}} x_m^{(q)}\, \bigl(1 - m^{-2}\bigr)
		\parentheses[\Bigg]{
			\frac{\sin\frac{\pi}{2mq}}{\sin\frac{j\pi}{mq}} - \frac{1}{2j}
		}
		+ \mathcal{O}\bigl(c_3\, j^2\, q^{-4}\bigr).
	\end{equation}

	\emph{Range $q_0 \leqslant j \leqslant J$.}
	We claim
	\begin{equation}\label{eq:lem:tail invertibility:sin ratio}
		\abs[\Bigg]{
			\frac{\sin\frac{\pi}{2mq}}{\sin\frac{j\pi}{mq}} - \frac{1}{2j}
		}
		\leqslant C\,\frac{j}{q^2}
		\qquad (m \in \set{2, \, 3, \, 4},\ 1 \leqslant j \leqslant J).
	\end{equation}
	Set $\varpi \define \pi/(mq)$ and $\operatorname{sinc} u \define \sin u / u$, so that $\sin\frac{\varpi}{2}\big/\sin(j\varpi) = \frac{1}{2j}\, \operatorname{sinc}(\varpi/2)\big/\operatorname{sinc}(j\varpi)$.
	For $1 \leqslant j \leqslant J$ we have $j\varpi \leqslant \frac{\pi}{2}\, q^{-1/3} \leqslant \frac{\pi}{2}$; on $\bigl[0, \frac{\pi}{2}\bigr]$ the function $\operatorname{sinc}$ is smooth, takes the value $1$ and has vanishing derivative at $0$, and satisfies $\operatorname{sinc} \geqslant \operatorname{sinc}\bigl(\frac{\pi}{2}\bigr) = \frac{2}{\pi}$, so $\operatorname{sinc}(\varpi/2)/\operatorname{sinc}(j\varpi) - 1 = \mathcal{O}(j^2 \varpi^2)$, and the left-hand side of~\eqref{eq:lem:tail invertibility:sin ratio} equals $\frac{1}{2j}\,\mathcal{O}(j^2 \varpi^2) = \mathcal{O}(j\, q^{-2})$.
	Substituting~\eqref{eq:lem:tail invertibility:sin ratio} into~\eqref{eq:lem:tail invertibility:nonmultiple cancellation} gives the pointwise bound $\lambda_{j, q} = \mathcal{O}(c_3\, j^2\, q^{-4})$ on this range, whose contribution to~\eqref{eq:lem:tail invertibility:nonmultiple goal} is
	\[
		q^{\gamma} \sum_{q_0 \leqslant j \leqslant J} j^{-\gamma}\, \mathcal{O}\bigl(c_3\, j^2\, q^{-4}\bigr)
		= \mathcal{O}\bigl(c_3\, q^{\gamma - 4}\bigr) \sum_{j \geqslant q_0} j^{2 - \gamma}
		= \mathcal{O}(c_3),
	\]
	since $3 < \gamma < 4$.

	\emph{Range $J < j \leqslant q^2$.}
	Bounding the bracket in~\eqref{eq:lem:tail invertibility:nonmultiple cancellation} term by term, with $\sin\frac{\pi}{2mq} \leqslant \frac{\pi}{2mq}$ for the sine ratio and $\frac{2 c_3 j}{q^2} \cdot \frac{1}{2j} = c_3\, q^{-2}$ for the inserted term, gives
	\[
		\abs{\lambda_{j, q}}
		\leqslant C\, c_3 \parentheses[\Bigg]{
			j\, q^{-3} \sum_{m = 2}^{4} \frac{1}{\abs[\big]{\sin\frac{j\pi}{mq}}}
			+ q^{-2}
			+ j^2\, q^{-4}
		}
		\leqslant C\, c_3 \parentheses[\Bigg]{
			j\, q^{-3} \sum_{m = 2}^{4} \frac{1}{\abs[\big]{\sin\frac{j\pi}{mq}}}
			+ j^2\, q^{-4}
		},
	\]
	where the middle term was absorbed into the first term, using $1\big/\abs[\big]{\sin\frac{j\pi}{mq}} \geqslant mq/(\pi j)$, which gives $j\, q^{-3} \big/ \abs[\big]{\sin\frac{j\pi}{mq}} \geqslant \frac{2}{\pi}\, q^{-2}$.
	The $j^2 q^{-4}$ term contributes $\mathcal{O}\bigl(c_3\, q^{\gamma - 4} \sum_{j > J} j^{2 - \gamma}\bigr) = \mathcal{O}(c_3)$ to~\eqref{eq:lem:tail invertibility:nonmultiple goal}, as before.
	For the first term we bound, for each $m \in \set{2, \, 3, \, 4}$,
	\[
		S_m \define \sum_{\substack{J < j \leqslant q^2 \\ q \nmid j}} \frac{j^{1 - \gamma}}{\abs[\big]{\sin\frac{j\pi}{mq}}}.
	\]
	Writing $j = k(mq) + r$ with $k \geqslant 0$ and $1 \leqslant r \leqslant mq - 1$ gives $\abs[\big]{\sin\frac{j\pi}{mq}} = \sin\frac{r\pi}{mq}$, and the bound $\sin\theta \geqslant \tfrac{2}{\pi}\min\set{\theta, \, \pi - \theta}$ on $[0, \pi]$ gives $1\big/\sin\frac{r\pi}{mq} \leqslant mq\big/\bigl(2\min\set{r, \, mq - r}\bigr)$, and therefore
	\begin{equation}\label{eq:lem:tail invertibility:cosecant sum}
		\sum_{r = 1}^{mq - 1} \frac{1}{\sin\frac{r\pi}{mq}} = \mathcal{O}(q \log q).
	\end{equation}
	For $k \geqslant 1$ we have $j \geqslant kmq$, so $j^{1 - \gamma} \leqslant (kmq)^{1 - \gamma}$; combined with~\eqref{eq:lem:tail invertibility:cosecant sum} and the convergence of $\sum_{k \geqslant 1} k^{1 - \gamma}$, the terms with $k \geqslant 1$ contribute $\mathcal{O}(q^{2 - \gamma}\log q)$ to $S_m$.
	For $k = 0$ we have $j = r$ with $J < r \leqslant mq - 1$, and we split at $\lfloor mq/2 \rfloor$; the range $r \leqslant \lfloor mq/2 \rfloor$ contributes at most $\tfrac{mq}{2}\sum_{r > J} r^{-\gamma} = \mathcal{O}\bigl(q\, J^{1 - \gamma}\bigr)$, and the range $r > \lfloor mq/2 \rfloor$, after the substitution $r' = mq - r$, contributes $\mathcal{O}(q^{2 - \gamma}\log q)$.
	Hence $S_m = \mathcal{O}(q^{2 - \gamma}\log q) + \mathcal{O}\bigl(q\, J^{1 - \gamma}\bigr)$, and since $q^{\gamma} j^{-\gamma} \cdot c_3\, j\, q^{-3} = c_3\, q^{\gamma - 3}\, j^{1 - \gamma}$, the first term contributes
	\[
		c_3\, q^{\gamma - 3} \sum_{m = 2}^{4} S_m
		= \mathcal{O}\bigl(c_3\, q^{-1}\log q\bigr) + \mathcal{O}\bigl(c_3\, q^{\gamma/3 - 4/3}\bigr)
		= \mathcal{O}(c_3),
	\]
	where $\gamma/3 - 4/3 = -\tfrac{1}{6} < 0$ at $\gamma = \tfrac{7}{2}$.

	\emph{Range $j > q^2$.}
	Each functional satisfies two bounds: the universal bound~\eqref{eq:triviality:universal bound}, which is $\mathcal{O}(1)$, since $p_i \abs{\varepsilon_i} \leqslant 4\, C_0\, c_3\, q^{-2} \leqslant C_0$ for $c_3 \leqslant 1$ and $q \geqslant 2$; and the mean value bound $\abs[\big]{\Phi^{j, p}(\varepsilon_i)} \leqslant \sup \abs[\big]{(\Phi^{j, p})'} \cdot \abs{\varepsilon_i} = \mathcal{O}(j p)\cdot\mathcal{O}(c_3 q^{-3}) = \mathcal{O}(c_3\, j\, q^{-2})$, using $\Phi^{j, p}(0) = 0$ and the derivative estimate $\abs[\big]{(\Phi^{j, p})'} = \mathcal{O}(j p)$.
	Hence
	\[
		\abs{\lambda_{j, q}}
		\leqslant C \min\set[\big]{1, \, c_3\, j\, q^{-2}},
	\]
	and we split the range at $j_* \define q^2 / c_3$.
	On $q^2 < j \leqslant j_*$, the bound $c_3\, j\, q^{-2}$ contributes
	\[
		q^{\gamma} \cdot c_3\, q^{-2} \sum_{j > q^2} j^{1 - \gamma}
		= \mathcal{O}\bigl(c_3\, q^{2 - \gamma}\bigr).
	\]
	On $j > j_*$, the universal bound contributes
	\[
		q^{\gamma} \sum_{j > j_*} j^{-\gamma}
		= \mathcal{O}\bigl(q^{\gamma}\, j_*^{1 - \gamma}\bigr)
		= \mathcal{O}\bigl(c_3^{\gamma - 1}\, q^{2 - \gamma}\bigr)
		= \mathcal{O}\bigl(c_3\, q^{2 - \gamma}\bigr),
	\]
	where the last step uses $c_3 \leqslant 1$ and $\gamma \geqslant 2$.
	Both contributions are $\mathcal{O}(c_3)$, and~\eqref{eq:lem:tail invertibility:nonmultiple goal} follows.

	\smallskip

	\emph{Multiples ($j = sq$, $s \geqslant 2$).}
	We claim that
	\begin{equation}\label{eq:lem:tail invertibility:multiple comparison}
		\abs[\big]{\lambda_{sq, q} - \lambda^{\infty}_{sq, q}}
		\leqslant C\, c_3\, (s^2 + 1)\, q^{-2}
		\qquad (s \geqslant 1),
	\end{equation}
	using the expansion~\eqref{eq:lem:low mode:taylor} at $j = sq$.
	For the first functional, $q \mid sq$ gives $\abs[\big]{D_1^{sq, q}} \leqslant q$; for the second, either $m \mid s$, in which case $mq \mid sq$ and $\abs[\big]{D_1^{sq, mq}} \leqslant mq$, or $m \nmid s$, in which case
	\[
		D_1^{sq, mq} = (-1)^{sq}\, mq \cdot sq\, \frac{2\sin\frac{\pi}{2mq}}{\sin\frac{s\pi}{m}},
		\qquad
		\abs[\big]{D_1^{sq, mq}} \leqslant \frac{\pi\, s q}{\sin\frac{\pi}{m}} \leqslant \sqrt{2}\,\pi\, s q,
	\]
	since $\abs[\big]{\sin\frac{s\pi}{m}} \geqslant \sin\frac{\pi}{m} \geqslant \frac{\sqrt{2}}{2}$ for $m \in \set{2, \, 3, \, 4}$ and $m \nmid s$.
	With $\abs{\varepsilon_i} \leqslant C_0\, c_3\, q^{-3}$, the first-order terms are $\mathcal{O}(c_3\, s\, q^{-2})$, and the second-order terms are $\mathcal{O}(s^2 q^4) \cdot \mathcal{O}(c_3^2\, q^{-6}) = \mathcal{O}(c_3\, s^2\, q^{-2})$ using $c_3 \leqslant 1$; this proves~\eqref{eq:lem:tail invertibility:multiple comparison}.
	We also record, from the decomposition~\eqref{eq:lem:tail invertibility:P decomposition} and the case bounds,
	\begin{equation}\label{eq:lem:tail invertibility:frozen multiples bound}
		\abs[\big]{\lambda^{\infty}_{sq, q}}
		\leqslant \abs{P(s)} + 1.07\, q^{-2}
		\leqslant \frac{256}{27} + \frac{1.07}{4}
		< 9.75
		\qquad (s \geqslant 2,\ q \geqslant q_0).
	\end{equation}

	For $s \geqslant 2$, combining the algebraic identity
	\[
		K_{q, sq} - K_{q, sq}^{\infty}
		= \frac{\lambda_{sq, q} - \lambda_{sq, q}^{\infty}}{\lambda_{q, q}}
		+ \lambda_{sq, q}^{\infty}\, \frac{\lambda_{q, q}^{\infty} - \lambda_{q, q}}{\lambda_{q, q}\, \lambda_{q, q}^{\infty}}
	\]
	with~\eqref{eq:lem:tail invertibility:multiple comparison}, \eqref{eq:lem:tail invertibility:frozen multiples bound}, and the diagonal bounds of Lemma~\ref{lem:low mode estimates}\,\ref{itm:low mode diagonal} and~\eqref{eq:triviality:frozen diagonal}, we obtain
	\[
		\abs[\big]{K_{q, sq} - K_{q, sq}^{\infty}} \leqslant C\, c_3\, (s^2 + 1)\, q^{-2}.
	\]
	Summing against $s^{-\gamma}$ (convergent, since $\gamma = \tfrac{7}{2} > 3$) gives
	\[
		\sup_{q \geqslant q_0} \sum_{s \geqslant 2} s^{-\gamma}\, \abs[\big]{K_{q, sq} - K_{q, sq}^{\infty}}
		= \mathcal{O}(c_3).
	\]
	Together with~\eqref{eq:lem:tail invertibility:nonmultiple goal} this proves~\eqref{eq:lem:tail invertibility:K comparison}.

	Every requirement so far bounds $a$ from below by a threshold depending only on $q_0$, and $c_3 = \pi^2/(8a)$ decreases in $a$.
	Define
	\[
		a_K \define \max \set[\bigg]{ a_{\mathrm{low}}(q_0), \, \frac{\pi^2\, C_K}{8 \times 0.0677} },
	\]
	with $C_K$ the constant of~\eqref{eq:lem:tail invertibility:K comparison}.
	For $a \geqslant a_K$, clause~\ref{itm:tail invertibility orbits} holds by Lemma~\ref{lem:tail orbit asymptotics}, since $a_{\mathrm{low}}(q_0) \geqslant 12$; the comparison~\eqref{eq:lem:tail invertibility:K comparison} gives $\matrixnorm{K - K^{\infty}}[\gamma, q_0] \leqslant C_K\, c_3 \leqslant 0.0677$; and the triangle inequality gives
	\[
		\matrixnorm{K}[\gamma, q_0]
		\leqslant \matrixnorm{K^{\infty}}[\gamma, q_0] + \matrixnorm{K - K^{\infty}}[\gamma, q_0]
		\leqslant 0.7323 + 0.0677
		= 0.8.
	\]
	The proof is complete.
\end{proof}

For $a \geqslant a_K$ we set
\[
	M \define (\mathrm{Id} + K)^{-1} K_{\mathrm{low}},
	\qquad
	M^{\infty} \define (\mathrm{Id} + K^{\infty})^{-1} K^{\infty}_{\mathrm{low}}.
\]
The matrix $M^{\infty}$ is independent of $a$.
By Lemma~\ref{lem:low mode estimates}, together with $c_3 \leqslant 1$ and the diagonal lower bounds of Lemma~\ref{lem:low mode estimates}\,\ref{itm:low mode diagonal} and~\eqref{eq:triviality:frozen diagonal}, the columns of $K_{\mathrm{low}}$ and $K^{\infty}_{\mathrm{low}}$ satisfy
\begin{equation}\label{eq:thm:rigidity:lambda jq decay claim}
	\abs[\big]{(K_{\mathrm{low}})_{q, j}} = \mathcal{O}\bigl(q^{-4}\bigr),
	\quad
	\abs[\big]{(K^{\infty}_{\mathrm{low}})_{q, 0}} = \mathcal{O}\bigl(q^{-4}\bigr),
	\quad
	(K^{\infty}_{\mathrm{low}})_{q, 1} = 0
	\quad (q \geqslant q_0,\ j \in \set{0, \, 1}).
\end{equation}
Since $\gamma < 4$, each column lies in $h_{\gamma, q_0}$, with
\begin{equation}\label{eq:thm:rigidity:Klow column norm}
	\matrixnorm{(K_{\mathrm{low}})_{\cdot, j}}[\gamma, q_0]
	= \sup_{q \geqslant q_0} q^{\gamma}\, \abs[\big]{(K_{\mathrm{low}})_{q, j}}
	= \mathcal{O}(1),
\end{equation}
and likewise for $K^{\infty}_{\mathrm{low}}$.
Applying the resolvent bounds~\eqref{eq:triviality:resolvent bounds} and the entrywise bound $\abs{v_{q}} \leqslant \matrixnorm{v}[\gamma, q_0]\, q^{-\gamma}$ gives
\begin{equation}\label{eq:thm:rigidity:decay of inv K low}
	\abs[\big]{M_{q, j}} = \mathcal{O}\bigl(q^{-\gamma}\bigr)
	\qquad (q \geqslant q_0,\ j \in \set{0, \, 1}).
\end{equation}
Left multiplication by $(\mathrm{Id} + K^{\infty})^{-1}$ acts on each column separately; since the $j = 1$ column of $K^{\infty}_{\mathrm{low}}$ vanishes by~\eqref{eq:thm:rigidity:lambda jq decay claim}, so does that of $M^{\infty}$:
\begin{equation}\label{eq:thm:rigidity:Mfix column decay}
	M^{\infty}_{q, 1} = 0,
	\qquad
	\abs[\big]{M^{\infty}_{q, 0}} = \mathcal{O}\bigl(q^{-\gamma}\bigr)
	\qquad (q \geqslant q_0).
\end{equation}

The \emph{limiting tail correction} is the $2 \times 2$ matrix
\begin{equation}\label{eq:thm:rigidity:Efix def}
	(E^{\infty})_{(1, m), j}
	\define
	-\sum_{q' = q_0}^{+\infty} \Diffl{q'}{1}{m}\big|_{\varepsilon_i = 0}\; M^{\infty}_{q', j}
	\qquad (m \in \set{2, \, 3},\ j \in \set{0, \, 1}).
\end{equation}
The series converges absolutely, since $\abs[\big]{\Diffl{q'}{1}{m}\big|_{\varepsilon_i = 0}} \leqslant \pi$ by~\eqref{eq:triviality:universal bound} and $\sum_{q'} q'^{-\gamma} < +\infty$.
By~\eqref{eq:thm:rigidity:Mfix column decay}, only the $j = 0$ column of $E^{\infty}$ is non-zero; in particular, the matrix $E^{\infty}$ has rank at most one.
Like $\widetilde{A}^{\flat}$, the matrix $E^{\infty}$ is independent of $a$.

The second supporting lemma bounds this tail correction matrix.

\begin{lemma}\label{lem:admissible threshold}
	Let $q_0 = 2$ and let $E^{\infty}$ be the limiting tail correction~\eqref{eq:thm:rigidity:Efix def}. Then
	\[
		\norm{E^{\infty}}_{\mathrm{op}} \leqslant 4.6 \times 10^{-2}.
	\]
\end{lemma}

\begin{proof}
	By~\eqref{eq:thm:rigidity:Mfix column decay}, only the $j = 0$ column of $E^{\infty}$ is non-zero, and it is given by
	\begin{equation}\label{eq:lem:Efix:column}
		(E^{\infty})_{(1, m), 0}
		= -\sum_{q' \geqslant q_0} \Diffl{q'}{1}{m}\big|_{\varepsilon_i = 0}\, g_{q'},
		\qquad
		g \define M^{\infty}_{\cdot, 0} = (\mathrm{Id} + K^{\infty})^{-1} \parentheses[\big]{ K_{\mathrm{low}}^{\infty} }_{\cdot, 0} .
	\end{equation}
	We first bound $g$ in $h_{\gamma, q_0}$ and then sum against the limiting coefficients.

	\smallskip

	\emph{The column $\parentheses[\big]{ K_{\mathrm{low}}^{\infty} }_{\cdot, 0}$.}
	For $p \geqslant 1$, set $x \define \frac{\pi}{2p} \in (0, \tfrac{\pi}{2}]$, so that $p \sin\frac{\pi}{2p} = \frac{\pi}{2}\, \frac{\sin x}{x}$.
	On this interval, the terms of the alternating series $\frac{\sin x}{x} = \sum_{k = 0}^{\infty} \frac{(-1)^k x^{2k}}{(2k + 1)!}$ strictly decrease in absolute value. Truncating after the term $k = 2$ yields
	\begin{equation}\label{eq:lem:Efix:sine expansion}
		p \sin\tfrac{\pi}{2p}
		= \tfrac{\pi}{2} - \tfrac{\pi^3}{48}\, p^{-2} + \tfrac{\pi^5}{3840}\, p^{-4} - \mathcal{R}(p),
		\qquad
		0 \leqslant \mathcal{R}(p) \leqslant \tfrac{\pi^7}{645120}\, p^{-6}.
	\end{equation}
	By the limiting values~\eqref{eq:thm:rigidity:leading order value} at $j = 0$, where both $p_1 = q'$ and $p_2 = mq'$ divide $j$,
	\[
		\Diffl{0}{q'}{m}\big|_{\varepsilon_i = 0}
		= q' \sin\tfrac{\pi}{2q'} - mq' \sin\tfrac{\pi}{2mq'}
		= -\tfrac{\pi^3}{48}\, (1 - m^{-2})\, q'^{-2} + \tfrac{\pi^5}{3840}\, (1 - m^{-4})\, q'^{-4} - \bigl( \mathcal{R}(q') - \mathcal{R}(mq') \bigr).
	\]
	Summing against the weights $x_m^{(q')}$ over $m \in \set{2, \, 3, \, 4}$, the second identity in~\eqref{eq:weights identities} cancels the term proportional to $q'^{-2}$.
	Using $\sum_m x_m^{(q')} (1 - m^{-4}) = (-1)^{q' + 1}\, \frac{115}{81 \pi}$, we obtain
	\[
		\lambda^{\infty}_{0, q'}
		= (-1)^{q' + 1}\, \frac{115\, \pi^4}{311040}\, q'^{-4}
		- \sum_{m \in \set{2, \, 3, \, 4}} x_m^{(q')} \bigl( \mathcal{R}(q') - \mathcal{R}(mq') \bigr),
	\]
	with the remainder bounded by $\frac{\pi^7}{645120} \bigl( \sum_{m \in \set{2, \, 3, \, 4}} \abs[\big]{x_m^{(q')}} (1 + m^{-6}) \bigr)\, q'^{-6}$.
	Since $q'^{-2} \leqslant \tfrac{1}{4}$ for $q' \geqslant q_0 = 2$, this gives
	\[
		\sup_{q' \geqslant q_0} q'^{4}\, \abs[\big]{\lambda^{\infty}_{0, q'}}
		\leqslant
		\frac{115\, \pi^4}{311040}
		+ \frac{1}{4} \cdot \frac{\pi^7}{645120} \sum_{m \in \set{2, \, 3, \, 4}} \abs[\big]{x_m^{(q')}} (1 + m^{-6})
		\leqslant 4.39 \times 10^{-2}.
	\]
	Since $q'^{\gamma} = q'^{4}\, q'^{-1/2} \leqslant q'^{4} / \sqrt{2}$ for $q' \geqslant q_0 = 2$, combining this estimate with the diagonal lower bound~\eqref{eq:triviality:frozen diagonal} yields
	\begin{equation}\label{eq:lem:Efix:Klow column}
		\matrixnorm[\big]{ \parentheses[\big]{ K_{\mathrm{low}}^{\infty} }_{\cdot, 0} }[\gamma, q_0]
		= \sup_{q' \geqslant q_0} q'^{\gamma}\, \frac{ \abs[\big]{\lambda^{\infty}_{0, q'}} }{ \lambda^{\infty}_{q', q'} }
		\leqslant \frac{4.39 \times 10^{-2}}{\sqrt{2} \cdot 0.9003}
		\leqslant 3.46 \times 10^{-2}.
	\end{equation}

	\smallskip

	\emph{The vector $g$.}
	By the resolvent bound~\eqref{eq:triviality:resolvent bounds}, $\matrixnorm{(\mathrm{Id} + K^{\infty})^{-1}}[\gamma, q_0] \leqslant 1/(1 - 0.7323) < 3.74$, so~\eqref{eq:lem:Efix:Klow column} gives $\matrixnorm{g}[\gamma, q_0] \leqslant 3.74 \cdot 3.46 \times 10^{-2} \leqslant 0.1295$, that is,
	\begin{equation}\label{eq:lem:Efix:g bound}
		\abs{g_{q'}} \leqslant 0.1295\, q'^{-\gamma}
		\qquad (q' \geqslant q_0).
	\end{equation}

	\smallskip

	\emph{The limiting values.}
	By~\eqref{eq:def:lambda j q} and~\eqref{eq:thm:rigidity:leading order value} with $(p_1, p_2) = (1, m)$, where $p_1 = 1$ divides every $q'$,
	\[
		\Diffl{q'}{1}{2}\big|_{\varepsilon_i = 0} = 1 - \mathbf{1}_{2 \mid q'}\, (-1)^{q' + q'/2}\, \sqrt{2},
		\qquad
		\Diffl{q'}{1}{3}\big|_{\varepsilon_i = 0} = 1 - \mathbf{1}_{3 \mid q'}\, (-1)^{q' + q'/3}\, \tfrac{3}{2}.
	\]
	Thus $\abs[\big]{\Diffl{q'}{1}{2}\big|_{\varepsilon_i = 0}}$ equals $1$ for $q'$ odd, $\sqrt{2} + 1$ for $q' \equiv 2 \pmod 4$, and $\sqrt{2} - 1$ for $4 \mid q'$; moreover, since $(-1)^{q' + q'/3} = 1$ whenever $3 \mid q'$, the coefficient $\abs[\big]{\Diffl{q'}{1}{3}\big|_{\varepsilon_i = 0}}$ equals $\tfrac{1}{2}$ for $3 \mid q'$ and $1$ otherwise.
	Summing against $q'^{-\gamma}$ over $q' \geqslant 2$ by residue classes, using $\sum_{t \geqslant 1,\, t\, \mathrm{odd}} t^{-\gamma} = (1 - 2^{-\gamma})\, \zeta(\gamma)$, we obtain
	\begin{align*}
		T_{1, 2} &\define \sum_{q' \geqslant q_0} q'^{-\gamma}\, \abs[\big]{\Diffl{q'}{1}{2}\big|_{\varepsilon_i = 0}}
		= \Bigl[ (1 - 2^{-\gamma}) \bigl( 1 + (\sqrt{2} + 1)\, 2^{-\gamma} \bigr) + (\sqrt{2} - 1)\, 4^{-\gamma} \Bigr]\, \zeta(\gamma) - 1
		\leqslant 0.2503, \\
		T_{1, 3} &\define \sum_{q' \geqslant q_0} q'^{-\gamma}\, \abs[\big]{\Diffl{q'}{1}{3}\big|_{\varepsilon_i = 0}}
		= \zeta(\gamma) - 1 - \tfrac{1}{2}\, 3^{-\gamma}\, \zeta(\gamma)
		\leqslant 0.115,
	\end{align*}
	where the numerical bounds use $\gamma = \tfrac{7}{2}$ and $\zeta(\tfrac{7}{2}) < 1.127$ (as $\sum_{n \leqslant 10} n^{-7/2} + \int_{10}^{+\infty} x^{-7/2}\, \mathrm{d}x < 1.127$).

	\smallskip

	Inserting~\eqref{eq:lem:Efix:g bound} into~\eqref{eq:lem:Efix:column} gives, for $m \in \set{2, \, 3}$,
	\[
		\abs[\big]{(E^{\infty})_{(1, m), 0}}
		\leqslant 0.1295\, T_{1, m}
		\leqslant 0.1295 \cdot 0.2503
		\leqslant 3.25 \times 10^{-2}.
	\]
	Since only the $j = 0$ column of $E^{\infty}$ is non-zero, its operator norm is the Euclidean norm of that column, and
	\[
		\norm{E^{\infty}}_{\mathrm{op}}
		= \norm[\big]{(E^{\infty})_{\cdot, 0}}_{\ell^2}
		\leqslant \sqrt{2} \cdot 3.25 \times 10^{-2}
		\leqslant 4.6 \times 10^{-2}.
		\qedhere
	\]
\end{proof}

We are ready to prove Theorem~\ref{thm:standard stadium rigidity}.

\begin{proof}[Proof of Theorem~\ref{thm:standard stadium rigidity}]
	Let $\mathbf{n} = \mathbf{n}_1 \cup \mathbf{n}_2 \in C^r_*(S)$ satisfy $\mathcal{J}(\mathbf{n}) = 0$, and fix $a \geqslant a_K$; the final threshold $a_0$ is determined at the end of the proof.
	The proof reduces the kernel equation to a $2 \times 2$ system $\widetilde{A}\, \mathbf{w}^{\,\mathrm{low}} = 0$ and proves $\widetilde{A}$ non-singular by comparison with the limiting matrix $\widetilde{A}^{(0)} = \widetilde{A}^{\flat} + E^{\infty}$.

	By Lemma~\ref{lem:standard stadium circular orbits},
	\[
		\widehat{\mathbf{n}}_{1, j} = -\widehat{\mathbf{n}}_{2, j}
		\qquad \text{for all } j \in \n_0.
	\]
	Since $\mathbf{n} \in C^{4}_*(S) \subseteq X_{*, \gamma}$ (Remark~\ref{rem:operator well-defined on subspace}), the isospectral identity $\ell^{q, mq}(\mathbf{n}) = 0$ may be expanded termwise whenever the orbit $\Gamma^{q, mq}$ exists. Substituting $\widehat{\mathbf{n}}_{2, j} = -\widehat{\mathbf{n}}_{1, j}$, we obtain
	\begin{equation}\label{eq:thm:rigidity:basic equations}
		\sum_{j = 0}^{+\infty} \widehat{\mathbf{n}}_{1, j}\, \Diffl{j}{q}{m} = 0.
	\end{equation}
	For $q \geqslant q_0$, all three orbits $\Gamma^{q, mq}$ with $m \in \set{2, \, 3, \, 4}$ exist and are regular by Lemma~\ref{lem:tail invertibility large a}\,\ref{itm:tail invertibility orbits}; combining~\eqref{eq:thm:rigidity:basic equations} over $m$ with the weights $x_m^{(q)}$ of~\eqref{eq:def:lambda j q} gives
	\begin{equation}\label{eq:thm:rigidity:lambda system}
		\sum_{j=0}^{+\infty}\lambda_{j,q}\,\widehat{\mathbf{n}}_{1, j}=0
		\qquad (q \geqslant q_0).
	\end{equation}
	Write
	\[
		\mathbf{w}^{\,\mathrm{low}} \define (\widehat{\mathbf{n}}_{1, 0}, \, \widehat{\mathbf{n}}_{1, 1})^T,
		\qquad
		\mathbf{w}^{\,\mathrm{high}} \define (\widehat{\mathbf{n}}_{1, 2}, \, \widehat{\mathbf{n}}_{1, 3}, \, \ldots)^T
	\]
	for the \emph{low-frequency coefficients} and the \emph{high-frequency tail}.
	The inclusion $C^{4}_*(S) \subseteq X_{*, \gamma}$ also gives $j^{\gamma}\, \widehat{\mathbf{n}}_{1, j} \to 0$, that is, $\mathbf{w}^{\,\mathrm{high}} \in h_{\gamma, q_0}$.
	Dividing row $q$ of~\eqref{eq:thm:rigidity:lambda system} by $\lambda_{q, q}$ (non-zero by Lemma~\ref{lem:low mode estimates}\,\ref{itm:low mode diagonal}) and splitting the sum at $q_0$ and at the diagonal gives $\bigl(K_{\mathrm{low}}\, \mathbf{w}^{\,\mathrm{low}}\bigr)_q + \bigl((\mathrm{Id} + K)\,\mathbf{w}^{\,\mathrm{high}}\bigr)_q = 0$ for each $q \geqslant q_0$.
	Applying the inverse from~\eqref{eq:triviality:resolvent bounds} gives the \emph{tail relation}
	\begin{equation}\label{eq:thm:rigidity:tail relation}
		\mathbf{w}^{\,\mathrm{high}} = -(\mathrm{Id}+K)^{-1} K_{\mathrm{low}}\, \mathbf{w}^{\,\mathrm{low}} = -M\, \mathbf{w}^{\,\mathrm{low}}.
	\end{equation}
	In particular, the tail $\mathbf{w}^{\,\mathrm{high}}$ is determined by $\mathbf{w}^{\,\mathrm{low}}$, and it remains to show $\mathbf{w}^{\,\mathrm{low}} = 0$.

	The short orbits $\Gamma^{1, 2}$ and $\Gamma^{1, 3}$ exist and are regular by Lemma~\ref{lem:short orbit control}, since $a \geqslant a_K \geqslant 12$, and their collision-angle corrections satisfy $\abs{\varepsilon_i} \leqslant 3 c_3$; by~\eqref{eq:triviality:universal bound}, we have $\abs[\big]{\Diffl{q'}{1}{m}} = \mathcal{O}(1)$ uniformly in $q' \in \n_0$.
	For $(q, m) = (1, m)$ with $m \in \set{2, \, 3}$, splitting~\eqref{eq:thm:rigidity:basic equations} at $q_0$ and substituting the tail relation~\eqref{eq:thm:rigidity:tail relation} termwise gives
	\begin{equation}\label{eq:thm:rigidity:head system}
		\sum_{j=0}^{1} \widetilde{A}_{(1,m),j}\, \widehat{\mathbf{n}}_{1, j} = 0,
		\qquad
		m \in \set{2, \, 3},
	\end{equation}
	where
	\begin{equation}\label{eq:thm:rigidity:head matrix}
		\widetilde{A}_{(1,m),j} \define \Diffl{j}{1}{m} - \sum_{q'=q_0}^{+\infty} \Diffl{q'}{1}{m}\, M_{q', j}
		\qquad (j \in \set{0, \, 1}).
	\end{equation}
	The interchange of the finite $j$-sum with the $q'$-series is justified by absolute convergence, since $\abs[\big]{\Diffl{q'}{1}{m}} = \mathcal{O}(1)$ and $\abs{M_{q', j}} = \mathcal{O}(q'^{-\gamma})$ by~\eqref{eq:thm:rigidity:decay of inv K low}.
	Since $q_0 = 2$, the system~\eqref{eq:thm:rigidity:head system} consists of two equations in the two unknowns $\widehat{\mathbf{n}}_{1, 0}$, $\widehat{\mathbf{n}}_{1, 1}$, with square matrix $\widetilde{A}$.

	Replacing every functional in~\eqref{eq:thm:rigidity:head matrix} by its limiting value~\eqref{eq:thm:rigidity:leading order value} turns $\Diffl{j}{1}{m}$ into $\widetilde{A}^{\flat}_{(1,m),j}$ and the tail sum into $-(E^{\infty})_{(1,m),j}$, by the definitions~\eqref{eq:triviality:A flat} and~\eqref{eq:thm:rigidity:Efix def}.
	The limiting head matrix is therefore
	\begin{equation}\label{eq:thm:rigidity:A0 decomposition}
		\widetilde{A}^{(0)} \define \widetilde{A}^{\flat} + E^{\infty}.
	\end{equation}
	By~\eqref{eq:triviality:sigma flat}, Lemma~\ref{lem:admissible threshold}, and Weyl's singular-value perturbation inequality (cf.~\cite[Corollary~7.3.5]{Horn2013MatrixAnalysis}),
	\begin{equation}\label{eq:thm:rigidity:sigma0}
		\sigma_0 \define \sigma_{\min}\bigl(\widetilde{A}^{(0)}\bigr) \;\geqslant\; \sigma_{\min}\bigl(\widetilde{A}^{\flat}\bigr) - \norm{E^{\infty}}_{\mathrm{op}} \;\geqslant\; 5.5 \times 10^{-2} - 4.6 \times 10^{-2} \;\geqslant\; 9 \times 10^{-3} \;>\; 0.
	\end{equation}

	Set $\Delta \widetilde{A} \define \widetilde{A} - \widetilde{A}^{(0)}$.
	By~\eqref{eq:thm:rigidity:head matrix} and~\eqref{eq:thm:rigidity:A0 decomposition}, each entry decomposes as
	\begin{equation}\label{eq:thm:rigidity:perturbation decomposition}
		(\Delta \widetilde{A})_{(1,m),j}
		= \underbrace{ \Diffl{j}{1}{m} - \Diffl{j}{1}{m}\big|_{\varepsilon_i=0} }_{\text{collision-angle error}}
		\;-\; \underbrace{\sum_{q'=q_0}^{+\infty} \bigl( \Diffl{q'}{1}{m}\, M_{q',j} - \Diffl{q'}{1}{m}\big|_{\varepsilon_i=0}\, M^\infty_{q',j} \bigr)}_{\text{tail-correction variation}}.
	\end{equation}
	We show that every entry is $\mathcal{O}(c_3)$.

	\emph{Collision-angle error.}
	Here $j \in \set{0, \, 1}$ and $p_i \in \set{1, \, m}$ are bounded, the coefficients $D_1^{j, p_i}$ of Lemma~\ref{lem:Taylor expansion of auxiliary variable} are bounded by an absolute constant, as is $\bigl(\Phi^{j, p_i}\bigr)''$, and $\abs{\varepsilon_i} \leqslant 3 c_3$; the expansion~\eqref{eq:lem:low mode:taylor} gives
	\[
		\abs[\Big]{ \Diffl{j}{1}{m} - \Diffl{j}{1}{m}\big|_{\varepsilon_i=0} } = \mathcal{O}(c_3).
	\]

	\emph{Tail-correction variation.}
	The sum in~\eqref{eq:thm:rigidity:perturbation decomposition} splits as $\mathrm{(I)} + \mathrm{(II)}$, where
	\[
		\mathrm{(I)} \define \sum_{q' \geqslant q_0} \Diffl{q'}{1}{m}\,\bigl( M_{q',j} - M^{\infty}_{q',j} \bigr),
		\qquad
		\mathrm{(II)} \define \sum_{q' \geqslant q_0} \bigl( \Diffl{q'}{1}{m} - \Diffl{q'}{1}{m}\big|_{\varepsilon_i=0} \bigr)\, M^{\infty}_{q',j}.
	\]

	\emph{Bound on $\mathrm{(I)}$.}
	Write $\Delta K \define K - K^{\infty}$ and $\Delta K_{\mathrm{low}} \define K_{\mathrm{low}} - K^{\infty}_{\mathrm{low}}$.
	Lemma~\ref{lem:tail invertibility large a}\,\ref{itm:tail invertibility comparison} gives $\matrixnorm{\Delta K}[\gamma, q_0] \leqslant C_K\, c_3$.
	For $\Delta K_{\mathrm{low}}$, combining the algebraic identity
	\[
		(\Delta K_{\mathrm{low}})_{q', j}
		= \frac{\lambda_{j, q'} - \lambda^{\infty}_{j, q'}}{\lambda_{q', q'}}
		+ \lambda^{\infty}_{j, q'}\, \frac{\lambda^{\infty}_{q', q'} - \lambda_{q', q'}}{\lambda_{q', q'}\, \lambda^{\infty}_{q', q'}}
	\]
	with the column bounds of Lemma~\ref{lem:low mode estimates}\,\ref{itm:low mode one}--\ref{itm:low mode zero} and the diagonal bounds of Lemma~\ref{lem:low mode estimates}\,\ref{itm:low mode diagonal} and~\eqref{eq:triviality:frozen diagonal}, we obtain
	\begin{equation}\label{eq:thm:rigidity:delta K bounds}
		\abs[\big]{(\Delta K_{\mathrm{low}})_{q', j}} = \mathcal{O}\bigl(c_3\, q'^{-4}\bigr),
		\qquad \text{hence} \qquad
		\matrixnorm{(\Delta K_{\mathrm{low}})_{\cdot, j}}[\gamma, q_0] = \mathcal{O}(c_3)
		\qquad (j \in \set{0, \, 1}).
	\end{equation}
	Combining the resolvent identity
	\[
		M - M^{\infty}
		=
		-(\mathrm{Id}+K)^{-1}\,\Delta K\,(\mathrm{Id}+K^{\infty})^{-1} K_{\mathrm{low}}
		\;+\;
		(\mathrm{Id}+K^{\infty})^{-1}\,\Delta K_{\mathrm{low}},
	\]
	column-wise with the resolvent bounds~\eqref{eq:triviality:resolvent bounds}, the column norms~\eqref{eq:thm:rigidity:Klow column norm}, and~\eqref{eq:thm:rigidity:delta K bounds}, we obtain
	\[
		\matrixnorm{(M - M^{\infty})_{\cdot,j}}[\gamma, q_0]
		= \mathcal{O}(c_3),
		\qquad\text{hence}\qquad
		\abs[\big]{M_{q',j} - M^{\infty}_{q',j}} = \mathcal{O}\bigl(c_3\, q'^{-\gamma}\bigr).
	\]
	Since $\abs[\big]{\Diffl{q'}{1}{m}} = \mathcal{O}(1)$,
	\[
		\abs{\mathrm{(I)}}
		= \mathcal{O}(c_3) \sum_{q' \geqslant q_0} q'^{-\gamma}
		= \mathcal{O}(c_3).
	\]

	\emph{Bound on $\mathrm{(II)}$.}
	By~\eqref{eq:thm:rigidity:Mfix column decay}, we have $M^{\infty}_{q', 1} = 0$, so $\mathrm{(II)} = 0$ for $j = 1$.
	For $j = 0$ we claim the linear bound
	\begin{equation}\label{eq:thm:rigidity:II linear bound}
		\abs[\big]{\Diffl{q'}{1}{m} - \Diffl{q'}{1}{m}\big|_{\varepsilon_i=0}}
		\leqslant
		C(q_0)\, c_3\, q'
		\qquad (q' \geqslant q_0).
	\end{equation}
	Consider the first-order coefficients at $j = q'$, where $p_i \in \set{1, \, m}$.
	If $p_i \mid q'$ (in particular whenever $p_i = 1$) then $\abs[\big]{D_1^{q', p_i}} = p_i \cos\frac{\pi}{2 p_i} \leqslant 3$.
	If $p_i \nmid q'$, which forces $p_i = m \geqslant 2$, then, by the double-angle identity as in~\eqref{eq:lem:low mode:D1 at one} and $\abs[\big]{\sin\frac{q'\pi}{p_i}} \geqslant \sin\frac{\pi}{p_i}$,
	\[
		\abs[\big]{D_1^{q', p_i}}
		= p_i\, q'\, \frac{2 \sin\frac{\pi}{2 p_i}}{\abs[\big]{\sin\frac{q' \pi}{p_i}}}
		\leqslant \frac{p_i\, q'}{\cos\frac{\pi}{2 p_i}}
		\leqslant 2\sqrt{3}\, q'.
	\]
	In either case $\abs[\big]{D_1^{q', p_i}} \leqslant 2\sqrt{3}\, q'$, while $\bigl(\Phi^{q', p_i}\bigr)'' = \mathcal{O}(q'^2)$ and $\abs{\varepsilon_i} \leqslant 3 c_3$.
	For $q' \leqslant 1/c_3$, the expansion~\eqref{eq:lem:low mode:taylor} gives
	\[
		\abs[\big]{\Diffl{q'}{1}{m} - \Diffl{q'}{1}{m}\big|_{\varepsilon_i=0}}
		\leqslant C(q_0)\,\bigl( c_3\, q' + c_3^2\, q'^2 \bigr)
		\leqslant 2\, C(q_0)\, c_3\, q',
	\]
	using $c_3\, q' \leqslant 1$; for $q' > 1/c_3$, the universal bound~\eqref{eq:triviality:universal bound} gives $\abs[\big]{\Diffl{q'}{1}{m} - \Diffl{q'}{1}{m}\big|_{\varepsilon_i=0}} \leqslant C(q_0) \leqslant C(q_0)\, c_3\, q'$.
	This proves~\eqref{eq:thm:rigidity:II linear bound}.
	Combining this with the column decay~\eqref{eq:thm:rigidity:Mfix column decay}, we obtain, for $j = 0$,
	\[
		\abs{\mathrm{(II)}}
		= \mathcal{O}(c_3) \sum_{q' \geqslant q_0} q'^{\,1-\gamma}
		= \mathcal{O}(c_3),
	\]
	where the series converges since $\gamma = \tfrac{7}{2} > 2$.

	Combining the three estimates, we conclude that every entry of $\Delta \widetilde{A}$ is $\mathcal{O}(c_3)$; since $\widetilde{A}$ is a $2 \times 2$ matrix,
	\[
		\norm{\Delta \widetilde{A}}_{\mathrm{op}} \leqslant C_1\, c_3
	\]
	for a constant $C_1$ depending only on $q_0$.
	Choose
	\[
		a_0 \define
		\max \set[\bigg]{
			a_K, \,
			\frac{\pi^2\, C_1}{4\, \sigma_0}
		}.
	\]
	For each $a \geqslant a_0$,
	\[
		\norm{\Delta \widetilde{A}}_{\mathrm{op}}
		\leqslant C_1\, c_3
		= \frac{\pi^2\, C_1}{8\, a}
		\leqslant \frac{\sigma_0}{2},
	\]
	and Weyl's inequality applied to $\widetilde{A} = \widetilde{A}^{(0)} + \Delta \widetilde{A}$ gives
	\[
		\sigma_{\min}\bigl(\widetilde{A}\bigr) \;\geqslant\; \sigma_0 - \norm{\Delta \widetilde{A}}_{\mathrm{op}} \;\geqslant\; \frac{\sigma_0}{2} \;>\; 0.
	\]
	Hence $\widetilde{A}$ is non-singular, and the head system~\eqref{eq:thm:rigidity:head system} forces $\mathbf{w}^{\,\mathrm{low}} = 0$; the tail relation~\eqref{eq:thm:rigidity:tail relation} then gives $\mathbf{w}^{\,\mathrm{high}} = 0$.
	Thus every Fourier coefficient of $\mathbf{n}_1$, and hence of $\mathbf{n}_2$, vanishes; since the Fourier series of $\mathbf{n}_i$ converges uniformly to $\mathbf{n}_i$ (Remark~\ref{rem:operator well-defined on subspace}), we conclude $\mathbf{n}_1 \equiv \mathbf{n}_2 \equiv 0$.
	The proof is complete.
\end{proof}

Theorem~\ref{thm:main rigidity} now follows.

\begin{proof}[Proof of Theorem~\ref{thm:main rigidity}]
	Take $a_{0}$ as in Theorem~\ref{thm:standard stadium rigidity}, and let $\set{\Omega_{\tau}}_{\abs{\tau} \leqslant 1}$ be a dynamically isospectral $C^{1}$ family in $\mathcal{M}^{r}$ with $\Omega_{0} = \Omega_{*}$.
	By Lemma~\ref{lem:isospectral deformation in kernel}, its deformation function lies in $\ker \mathcal{J}$, which is trivial by Theorem~\ref{thm:standard stadium rigidity}.
\end{proof}

%% file: section/Acknowledgements.tex
\section*{Acknowledgements}

The author is deeply grateful to Vadim Kaloshin for suggesting this problem, and for his insightful perspectives and invaluable advice throughout this research. 
He expresses his sincere gratitude to Jianyu Chen for illuminating discussions and generous guidance during the early stages of this work, and for his warm hospitality during the author's visit to Soochow University.

The author would like to thank his Ph.D. advisor, Zhiqiang Li, for his support in the preliminary stage of the project, and Boxi Liu for helpful discussions on the preliminary notes.

The author is especially indebted to Otto Vaughn Osterman for numerous insightful discussions and valuable suggestions on the manuscript. 

Part of this work was carried out while the author was visiting the Institute of Science and Technology Austria (ISTA); the author is grateful to the institute for its warm hospitality and excellent research environment.

%% file: ref.bib
@article{MR3665005,
  title={Dynamical spectral rigidity among {$\mathbb{Z}_2$}-symmetric strictly convex domains close to a circle},
  author={De Simoi, Jacopo and Kaloshin, Vadim and Wei, Qiaoling},
  addendum={Appendix B coauthored with H. Hezari},
  journal={Annals of Mathematics},
  shortjournal={{Ann. of Math. (2)}},
  volume={186},
  number={1},
  pages={277--314},
  year={2017},
  doi={10.4007/annals.2017.186.1.7},
  url={https://doi.org/10.4007/annals.2017.186.1.7}
}

@article{ChenKaloshinZhang2023,
  title={Length spectrum rigidity for piecewise analytic {B}unimovich billiards},
  author={Chen, Jianyu and Kaloshin, Vadim and Zhang, Hong-Kun},
  journal={Communications in Mathematical Physics},
  shortjournal={{Comm. Math. Phys.}},
  volume={404},
  number={1},
  pages={1--50},
  year={2023}
}

@article{Kac_1966,
  title={Can one hear the shape of a drum?},
  author={Kac, Mark},
  journal={The American Mathematical Monthly},
  shortjournal={{Amer. Math. Monthly}},
  volume={73},
  number={4},
  pages={1--23},
  year={1966},
  doi={10.1080/00029890.1966.11970915},
  url={https://doi.org/10.1080/00029890.1966.11970915}
}

@article{Bunimovich_1979,
  title={On the ergodic properties of nowhere dispersing billiards},
  author={Bunimovich, Leonid A.},
  journal={Communications in Mathematical Physics},
  shortjournal={{Comm. Math. Phys.}},
  volume={65},
  number={3},
  pages={295--312},
  year={1979},
  doi={10.1007/BF01197884},
  url={https://doi.org/10.1007/BF01197884}
}

@article{Lazutkin_1973,
  title={The existence of caustics for a billiard problem in a convex domain},
  author={Lazutkin, Vladimir F.},
  journal={Mathematics of the USSR-Izvestiya},
  shortjournal={{Math. USSR-Izv.}},
  volume={7},
  number={1},
  pages={185--214},
  year={1973},
  doi={10.1070/IM1973v007n01ABEH001932},
  url={https://doi.org/10.1070/IM1973v007n01ABEH001932}
}

@article{AnderssonMelrose1977,
  title={The propagation of singularities along gliding rays},
  author={Andersson, Karl Gustav and Melrose, Richard B.},
  journal={Inventiones Mathematicae},
  shortjournal={{Invent. Math.}},
  volume={41},
  number={3},
  pages={197--232},
  year={1977}
}

@book{PetkovStoyanov1992,
  title={Geometry of Reflecting Rays and Inverse Spectral Problems},
  author={Petkov, Vesselin M. and Stoyanov, Luchezar N.},
  series={Pure and Applied Mathematics (New York)},
  publisher={John Wiley \& Sons, Ltd.},
  address={Chichester},
  year={1992}
}

@article{GordonWebbWolpert1992,
  title={One cannot hear the shape of a drum},
  author={Gordon, Carolyn and Webb, David L. and Wolpert, Scott},
  journal={Bulletin of the American Mathematical Society. New Series},
  shortjournal={{Bull. Amer. Math. Soc. (N.S.)}},
  volume={27},
  number={1},
  pages={134--138},
  year={1992}
}

@article{HezariZelditch2012,
  title={{$C^{\infty}$} spectral rigidity of the ellipse},
  author={Hezari, Hamid and Zelditch, Steve},
  journal={Analysis \& PDE},
  shortjournal={{Anal. PDE}},
  volume={5},
  number={5},
  pages={1105--1132},
  year={2012}
}

@article{HezariZelditch2022,
  title={One can hear the shape of ellipses of small eccentricity},
  author={Hezari, Hamid and Zelditch, Steve},
  journal={Annals of Mathematics},
  shortjournal={{Ann. of Math. (2)}},
  volume={196},
  number={3},
  pages={1083--1134},
  year={2022}
}

@article{Zelditch2014survey,
  title={Survey on the inverse spectral problem},
  author={Zelditch, Steve},
  journal={ICCM Notices},
  shortjournal={{ICCM Not.}},
  volume={2},
  number={2},
  pages={1--20},
  year={2014}
}

@article{AvilaDeSimoiKaloshin2016,
  title={An integrable deformation of an ellipse of small eccentricity is an ellipse},
  author={Avila, Artur and De Simoi, Jacopo and Kaloshin, Vadim},
  journal={Annals of Mathematics},
  shortjournal={{Ann. of Math. (2)}},
  volume={184},
  number={2},
  pages={527--558},
  year={2016}
}

@article{KaloshinSorrentino2022survey,
  title={Inverse problems and rigidity questions in billiard dynamics},
  author={Kaloshin, Vadim and Sorrentino, Alfonso},
  journal={Ergodic Theory and Dynamical Systems},
  shortjournal={{Ergodic Theory Dynam. Systems}},
  volume={42},
  number={3},
  pages={1023--1056},
  year={2022}
}

@article{KaloshinSorrentino2018,
  title={On the local {B}irkhoff conjecture for convex billiards},
  author={Kaloshin, Vadim and Sorrentino, Alfonso},
  journal={Annals of Mathematics},
  shortjournal={{Ann. of Math. (2)}},
  volume={188},
  number={1},
  pages={315--380},
  year={2018}
}

@article{HuangKaloshinSorrentino2018,
  title={Nearly circular domains which are integrable close to the boundary are ellipses},
  author={Huang, Guan and Kaloshin, Vadim and Sorrentino, Alfonso},
  journal={Geometric and Functional Analysis},
  shortjournal={{Geom. Funct. Anal.}},
  volume={28},
  number={2},
  pages={334--392},
  year={2018}
}

@article{Koval2026,
  title={Local strong {B}irkhoff conjecture and local spectral uniqueness of almost every ellipse},
  author={Koval, Illya},
  journal={Inventiones Mathematicae},
  shortjournal={{Invent. Math.}},
  volume={244},
  number={1},
  pages={221--298},
  year={2026}
}

@article{HuangKaloshin2019,
  title={On the finite dimensionality of integrable deformations of strictly convex integrable billiard tables},
  author={Huang, Guan and Kaloshin, Vadim},
  journal={Moscow Mathematical Journal},
  shortjournal={{Mosc. Math. J.}},
  volume={19},
  number={2},
  pages={307--327},
  year={2019}
}

@misc{FierobeKaloshinSorrentino2025,
  title={A billiard table close to an ellipse is deformationally spectrally rigid among dihedrally symmetric domains},
  author={Fierobe, Corentin and Kaloshin, Vadim and Sorrentino, Alfonso},
  year={2025},
  eprint={2511.20062},
  eprinttype={arxiv},
  eprintclass={math.DS}
}

@article{DeSimoiKaloshinLeguil2023,
  title={Marked length spectral determination of analytic chaotic billiards with axial symmetries},
  author={De Simoi, Jacopo and Kaloshin, Vadim and Leguil, Martin},
  journal={Inventiones Mathematicae},
  shortjournal={{Invent. Math.}},
  volume={233},
  number={2},
  pages={829--901},
  year={2023}
}

@article{GuilleminKazhdan1980,
  title={Some inverse spectral results for negatively curved 2-manifolds},
  author={Guillemin, Victor and Kazhdan, David},
  journal={Topology},
  shortjournal={{Topology}},
  volume={19},
  number={3},
  pages={301--312},
  year={1980}
}

@article{GuillarmouLefeuvre2019,
  title={The marked length spectrum of {A}nosov manifolds},
  author={Guillarmou, Colin and Lefeuvre, Thibault},
  journal={Annals of Mathematics},
  shortjournal={{Ann. of Math. (2)}},
  volume={190},
  number={1},
  pages={321--344},
  year={2019}
}

@book{Horn2013MatrixAnalysis,
  title={Matrix Analysis},
  author={Horn, Roger A. and Johnson, Charles R.},
  edition={2nd},
  publisher={Cambridge University Press},
  location={New York},
  year={2013}
}

@misc{KovalVig2024,
  title={Silent orbits and cancellations in the wave trace},
  author={Koval, Illya and Vig, Amir},
  year={2024},
  eprint={2408.09238},
  eprinttype={arxiv},
  eprintclass={math.SP}
}

@article{Osterman2023,
  title={On length spectrum rigidity of dispersing billiard systems},
  author={Osterman, Otto Vaughn},
  journal={Journal of Modern Dynamics},
  shortjournal={{J. Mod. Dyn.}},
  volume={19},
  pages={847--878},
  year={2023},
  doi={10.3934/jmd.2023025},
  url={https://doi.org/10.3934/jmd.2023025}
}
